\pdfoutput=1

\documentclass[depthtwo,microtype,bib,noswap]{nicestyle}

\title{Cuspidal Calogero--Moser and Lusztig families\protect\\ for Coxeter groups} 
\author{Gwyn Bellamy and Ulrich Thiel}
\bibliography{biblo}

\allowdisplaybreaks
\usepackage{ytableau}
\ytableausetup{boxsize=0.8em}
\ytableausetup{centertableaux}

\usepackage{array}
\newcolumntype{L}[1]{>{\raggedright\let\newline\\\arraybackslash\hspace{0pt}}m{#1}}
\newcolumntype{C}[1]{>{\centering\let\newline\\\arraybackslash\hspace{0pt}}m{#1}}
\newcolumntype{R}[1]{>{\raggedleft\let\newline\\\arraybackslash\hspace{0pt}}m{#1}}

\begin{document}  

\blfootnote{Date: May 26, 2016 (first version: May 3, 2015)} 
\blfootnote{Final version to appear in J. Algebra}
\blfootnote{\textsc{Gwyn Bellamy}, School of Mathematics and Statistics, University Gardens, University of Glasgow, Glasgow, G12 8QW, UK, \texttt{gwyn.bellamy@glasgow.ac.uk}}
\blfootnote{\textsc{Ulrich Thiel}, Universität Stuttgart, Fachbereich Mathematik, Institut für Algebra und Zahlentheorie, Lehrstuhl für Algebra, Pfaffenwaldring 57, 70569 Stuttgart, Germany,  \texttt{thiel@mathematik.uni-stuttgart.de}}

\begin{abstract}
The goal of this paper is to compute the cuspidal Calogero--Moser families for all infinite families of finite Coxeter groups, at all parameters. We do this by first computing the symplectic leaves of the associated Calogero--Moser space and then by classifying certain ``rigid'' modules. Numerical evidence suggests that there is a very close relationship between Calogero--Moser families and Lusztig families. Our classification shows that, additionally, the cuspidal Calogero–Moser families equal cuspidal Lusztig families for the infinite families of Coxeter groups. 
\end{abstract}

\maketitle
\thispagestyle{empty}

\section*{Introduction}

Based on the relationship between Dunkl operators, the Knizhnik–Zamolodchikov connection,  and Hecke algebras, it became apparent very soon after the introduction of rational Cherednik algebras by Etingof and Ginzburg \cite{EG} that there is a very close connection between these algebras and cyclotomic Hecke algebras \cite{BMR}. This connection is encoded in the Knizhnik–Zamolodchikov functor, introduced in \cite{GGOR}, and is a key tool in the representation theory of rational Cherednik algebras at $t \neq 0$. 

In the quasi-classical limit $t = 0$ the Knizhnik–Zamolodchikov functor no longer exists and no functorial connection to Hecke algebras is currently known. Astonishingly, as first noticed by Gordon and Martino \cite{GordonMartinoCM}, it seems that there is still, none the less, a close relationship between rational Cherednik algebras in $t=0$ and Hecke algebras, suggesting that there may be an asymptotic Knizhnik–Zamolodchikov functor in the quasi-classical limit. The aim of this article is to add weight to this expectation by comparing cuspidal Calogero–Moser families with cuspidal Lusztig families. 

\subsection*{Families} 
Etingof and Ginzburg \cite{EG} defined, for any finite reflection group $(\fh,W)$ and a function $\bc:\Ref(W) \rarr \bbC$ from the set of reflections of $W$ to the complex numbers which is invariant under $W$-conjugation, the rational Cherednik algebra $\H_\bc(W)$ at $t=0$. The spectrum of the centre of this algebra is an affine Poisson deformation $\cX_\bc(W)$ of the symplectic singularity $(\fh \times \fh^*)/W$, called the Calogero–Moser space.  This theory exists in particular for finite Coxeter groups $W$. In this case, one can also attach to $W$ the Hecke algebra $\mathcal{H}_L(W)$ depending on a weight function $L:W \rarr \bbR$. The space of weight functions $L$ and the space of real valued $\bc$-functions is the same so that one can relate invariants coming from Hecke algebras with those coming from rational Cherednik algebras.

Gordon \cite{Baby} has defined the notion of Calogero–Moser $\bc$-families of $\Irr(W)$, which on the geometric side correspond to the $\bbC^*$-fixed points of the Calogero–Moser space $\cX_\bc(W)$. Work of several people, in particular Gordon and Martino, has shown that:

\begin{fact*}\label{thm:familiesequal} 
If $W$ is a Coxeter group of type $A,B,D$, $I_2(m)$ or $H_3$, then the Lusztig $\bc$-families equal the Calogero–Moser $\bc$-families for all $\bc:\Ref(W) \rarr \bbR$. 
\end{fact*}

We refer to \S\ref{lusztig_fam_vs_cm_fam} for more details. It is conjectured by Gordon–Martino \cite{GordonMartinoCM} that this is indeed true for all finite Coxeter groups; see also Bonnafé–Rouquier \cite{BonnafeRouquier}. There is so far no conceptual explanation for this connection. Bonnafé and Rouquier \cite{BonnafeRouquier} furthermore constructed analogs of constructible characters and cells on the Calogero–Moser side, and collected evidence supporting their conjecture that these notions coincide with Lusztig's notions; see also \cite{Bonnafe:2013aa}.

\subsection*{Cuspidal families}The key to defining constructible representations and Lusztig families for Hecke algebras is Lusztig's truncated induction, also called $\bj$-induction. This leads to the concept of \textit{cuspidal} Lusztig families, which are those that cannot be described as being $\bj$-induced from a family for a proper parabolic subgroup. Cuspidal families play a key role in describing certain unipotent representations for the corresponding finite groups of Lie type. In \cite{Cuspidal} the first author also introduced the notion of cuspidal Calogero–Moser families. This time the definition is geometric: a family is cuspidal if the support of every module in the family is a zero-dimensional symplectic leaf of the Calogero–Moser space. In this article we determine the cuspidal Calogero–Moser families for the Coxeter groups of type $A,B,D$ and $I_2(m)$. Our main result states (see \S\ref{cuspidal_families}):

\begin{maintheorem}\label{thm:mainresultintro} 
If $W$ is of type $A,B,D$ or $I_2(m)$, then the \textit{cuspidal} Lusztig $\bc$-families equal the \textit{cuspidal} Calogero–Moser $\bc$-families for all $\bc:\Ref(W) \rarr \bbR$.
\end{maintheorem}

The proof follows from a case-by-case analysis in sections \S\ref{type_A} to \S\ref{dihedral_groups} using theoretical methods we develop in section \S\ref{calculating}. Based on this theorem we make the following conjecture.

\begin{mainconjecture}\label{cm_cusp_lusztigz_cusp_conjecture}
For any finite Coxeter group the cuspidal Lusztig $\bc$-families equal the cuspidal Calogero–Moser $\bc$-families for all real parameters $\bc$.
\end{mainconjecture}

Because of Theorem \ref{thm:mainresultintro} this conjecture remains open only for the six exceptional Coxeter groups $H_3,H_4,F_4,E_6,E_7,E_8$.

\subsection*{Rigid representations} The main ingredient for calculating the cuspidal Calogero–Moser families, and hence confirming Theorem \ref{thm:mainresultintro}, is the notion of a \word{rigid} module: a $\H_{\bc}(W)$-module is said to be rigid if it is irreducible as a $W$-module. These have already played a role in the representation theory of rational Cherednik algebras at $t  \neq 0$, see e.g. \cite{Finitedimreps} or \cite{MontaraniEtingof}, and at $t = 0$ they were studied by the second author in \cite{Thiel.UOn-restricted-ration}.  The terminology comes from the theory of module varieties. Namely, for any $d < |W|$, we show in Lemma \ref{lem:openrep} that the set $X$ of rigid modules in $\Rep_d(\H_\bc(W))$, the variety parameterizing representations of dimension $d$, is open. Therefore, though these modules often appear in families with respect to the parameter $\bc$, the module structure (for fixed parameter $\bc$) on a rigid module cannot be deformed to a continuous family. This is the first clue that there is a strong connection between rigid representations and zero-dimensional leaves of $\cX_{\bc}(W)$ (and hence to cuspidal Calogero-Moser families). 

In this article we classify the rigid modules for all non-exceptional Coxeter groups and all parameters. The importance of these modules is explained by our second main result which we prove in \S\ref{calculating}:   

\begin{maintheorem} \label{maintheorem:rigid_cuspidal}
Let $W$ be an arbitrary finite complex reflection group. If the simple module $L_\bc(\lambda)$, where $\lambda \in \Irr (W)$, is a rigid $\H_\bc(W)$-module, then the Calogero–Moser $\bc$-family to which it belongs is cuspidal. 
\end{maintheorem}

Rigid modules are easily computed, and using Theorem \ref{maintheorem:rigid_cuspidal} this allows us to identify certain cuspidal families. Remarkably, for the non-exceptional Coxeter groups we can show that the cuspidal Calogero–Moser families are precisely those containing the rigid modules. The cuspidal Lusztig families are similarly characterized.

\begin{remark*}
While this paper was in preparation, the preprint \cite{Ciubotaru:2015aa} appeared, where rigid modules also play a key role (though the definition there is slightly different). Based on the analogy with affine Hecke algebras, they are called "one-$W$-type" modules in \textit{loc. cit.}. In the preprint \cite{Ciubotaru:2015aa} the author gives a different notion of \textit{cuspidal} Calogero--Moser families. Namely, in \emph{loc. cit.} a family is said to be cuspidal if it contains a rigid module. By Theorem \ref{maintheorem:rigid_cuspidal}, every cuspidal family in our sense is cuspidal in the sense of \cite{Ciubotaru:2015aa}. However, it is clear that for most complex reflection groups that are not of Coxeter type  there exist many cuspidal families (in our sense) that are not cuspidal in the sense of \emph{loc. cit.}. Moreover, as shown in \emph{loc. cit.}, Conjecture \ref{cm_cusp_lusztigz_cusp_conjecture} is \textit{false} for the Weyl group of type $E_7$ if we use the definition of cuspidal used in \emph{loc. cit.}
\end{remark*}

\subsection*{Symplectic leaves}

As previously noted, the notion of cuspidal Calogero–Moser families depends on the fact that the Calogero–Moser space $\cX_{\bc}(W)$ is stratified by finitely many symplectic leaves. These leaves are naturally labeled by conjugacy classes of parabolic subgroups $(W')$ of $W$. There are two natural partial orderings on the set of symplectic leaves: a \textit{geometric} one given in terms of the closures of leaves, and another, \textit{algebraic} one given in terms of inclusions of parabolic subgroups. It is clear that the geometric ordering refines the algebraic ordering. 

Using results of Martino, we describe all symplectic leaves for the Coxeter groups of type $A,B,D$ and $I_2(m)$ in terms of the conjugacy classes of parabolic subgroups. We also describe the two orderings on the set of symplectic leaves in these cases (see Theorem \ref{thm:BSympleaves}, Theorem \ref{thm:Dleaves} and \cite[Tables 1,2]{Cuspidal}). Based on this we arrive at the following conjecture.

\begin{mainconjecture}\label{conj_leaves}
Let $W$ be a finite Coxeter group. 
\begin{enum_thm}
\item Each conjugacy class of parabolic subgroups $(W')$ labels at most one symplectic leaf. 
\item The geometric ordering on leaves equals the algebraic ordering.  
\end{enum_thm}
\end{mainconjecture}

We note that both statements of Conjecture \ref{conj_leaves} may fail if $W$ is not a Coxeter group.

\subsection*{Clifford Theory}

Our results for Coxeter groups of type $D$ are deduced from the corresponding results for the groups of type $B$ using the fact that $D_n \lhd B_n$. More generally, we consider a complex reflection group $(\h,W)$ and a normal subgroup $K \lhd W$ such that $(\h |_{K},K)$ is also a reflection group. This situation is also considered in \cite{CMPartitions} and by Liboz \cite{Liboz}.

Based on a suggestion of Rouquier, we show that $\Gamma := W/K$ acts on the Calogero–Moser space $\cX_{\bc}(K)$ such that $\cX_{\bc}(W) = \cX_{\bc}(K)/ \Gamma$. This allows us to deduce the Calogero–Moser families for $K$ from the Calogero–Moser families for $W$, generalising results of \cite{CMPartitions}. Cuspidal families and rigid representations behave well under this correspondence. We also describe the symplectic leaves in $\cX_{\bc}(K)$ in terms of those of $\cX_{\bc}(W)$.

\begin{ack}
The authors would like to thank Cédric Bonnafé and Meinolf Geck for many fruitful discussions. We also thank Dan Ciubotaru for informing us about his preprint \cite{Ciubotaru:2015aa} and his result that for $E_7$ the cuspidal Lusztig family does not contain rigid modules. Moreover, we would like to thank Gunter Malle for commenting on a preliminary version of this article. The second author was partially supported by the DFG Schwerpunktprogramm 1489.
\end{ack}

\newpage

\tableofcontents

\section{Calogero–Moser families}

We begin by recalling the definition of the main protagonists of this paper—the Calogero–Moser families for complex reflection groups. They are obtained from the block structure of the restricted rational Cherednik algebra studied by Gordon \cite{Baby}, which is a finite-dimensional quotient of the rational Cherednik algebra introduced by Etingof and Ginzburg \cite{EG}.

\subsection{Rational Cherednik algebras}

Let $(\h,W)$ be a finite complex reflection group. By this we mean that $W$ is a non-trivial finite subgroup of $\GL(\h)$ for some finite-dimensional complex vector space $\h$ such that $W$ is generated by its set $\Ref(W)$ of \word{reflections}, i.e., by those elements $s \in W$ such that $\Ker(\id_\fh - s)$ is of codimension one in $\fh$. Let $( \cdot, \cdot ) : \mathfrak{h} \times \mathfrak{h}^* \rightarrow \C$ be the natural pairing defined by $(y,x) = x(y)$. For $s \in \Ref(W)$ we fix $\alpha_s \in \mathfrak{h}^*$ to be a basis of the one-dimensional space $\Im (s - 1)|_{\mathfrak{h}^*}$ and $\alpha_s^{\vee} \in \mathfrak{h}$ to be a basis of the one-dimensional space $\Im (s - 1)|_{\mathfrak{h}}$, normalised so that $\alpha_s(\alpha_s^\vee) = 2$. Our discussion will not depend on the choice of  $\alpha_s$ and $\alpha_s^\vee$. Note that the group $W$ acts on $\Ref(W)$ by conjugation. Choose a function $\mathbf{c} : \Ref(W) \rightarrow \C$ which is invariant under $W$-conjugation (we say that $\bc$ is \word{$W$-equivariant}) and furthermore choose a complex number $t \in \bbC$. The \textit{rational Cherednik algebra} $\H_{t,\mathbf{c}}(W)$, as introduced by Etingof and Ginzburg \cite{EG}, is the quotient of the skew group algebra of the tensor algebra, $T(\mf{h} \oplus \mf{h}^*) \rtimes W$, by the ideal generated by the relations $[x,x'] = [y,y'] = 0$ for all $x,x' \in \mathfrak{h}^*$ and $y,y' \in \mathfrak{h}$, and 
\begin{equation}\label{eq:rel}
[y,x] = t(y,x)  - \!\! \sum_{s \in \Ref(W)} \bc(s) (y,\alpha_s)(\alpha_s^\vee,x) s \;, \quad \forall \ y \in \h, \ x \in \h^* \;. 
\end{equation}
We concentrate on the case $t=0$ and set $\H_\bc \dopgleich \H_{0,\bc}$. 
For any $\alpha \in \C \backslash \{ 0 \}$, the algebras $\H_{\alpha \mathbf{c}}(W)$ and $\H_{\mathbf{c}}(W)$ are naturally isomorphic. Therefore we are free to rescale $\mbf{c}$ by $\alpha$ whenever this is convenient. A fundamental result for rational Cherednik algebras, proved by Etingof and Ginzburg \cite[Theorem 1.3]{EG}, is that the \word{PBW property} holds for all $\mbf{c}$, i.e., the natural map 
\begin{equation}\label{eq:PBW}
 \C [\h] \otimes_\C \C W \otimes_\C \C [\h^*] \rarr \H_\bc(W) 
\end{equation}
is an isomorphism of $\bbC$-vector spaces. The rational Cherednik algebra is naturally $\Z$-graded by $\deg(x) = 1$ for $x \in \h^*$, $\deg(y) = -1$ for $y \in \h$, and $\deg(w) = 0$ for $w \in W$. We note that no such grading exists for general symplectic reflection algebras. 

\subsection{Calogero–Moser space}

The centre $\ZH_{\bc}(W)$ of $\H_{\bc}(W)$ is an affine domain. We shall denote by $\cX_{\mathbf{c}}(W) := \Spec(\ZH_{\mathbf{c}}(W))$ the corresponding affine variety. It is called the (generalized) \textit{Calogero–Moser space} associated to $W$ at parameter $\mathbf{c}$. These varieties define a flat family of deformations of $(\fh \oplus \fh^*)/W$ over the affine $\bbC$-space of dimension $|\Ref(W)/W|$. The following was shown for Coxeter groups in \cite[Proposition 4.15]{EG}, and the general case is due to \cite[Proposition 3.6]{Baby}. 

\begin{prop}[Etingof–Ginzburg, Gordon]
The subspace $D(W) \dopgleich \C[\h]^W \otimes_\bbC \C[\h^*]^W$ of $\H_{\mbf{c}}(W)$ is a central subalgebra and $\ZH_{\mbf{c}}(W)$ is a free $D(W)$-module of rank $|W|$.
\end{prop}

The inclusions $\C[\mathfrak{h}]^W \hookrightarrow Z_{\mathbf{c}}(W)$ and $\C[\mathfrak{h}^*]^W \hookrightarrow Z_{\mathbf{c}}(W)$ define finite surjective morphisms 
\[
\pi_\bc : \cX_{\mathbf{c}}(W) \twoheadrightarrow \h/W \quad \tn{and} \quad \varpi_\bc :\cX_{\mathbf{c}}(W) \twoheadrightarrow \h^*/W \;.
\]
We write 
$$
\Upsilon_\bc \dopgleich \pi_\bc \times \varpi_\bc \,  : \,  \cX_{\mathbf{c}}(W) \twoheadrightarrow \h / W \times \h^*/W
$$
for the product morphism. It is a finite, and hence closed, surjective morphism. Note that both $\ZH_\bc(W)$ and $D(W)$ are graded subalgebras of $\H_\bc(W)$. This implies that $\cX_\bc(W)$ and $\h/W \times \h^*/W$ carry a $\bbC^*$-action making $\Upsilon_\bc$ a $\bbC^*$-equivariant morphism.

\subsection{Restricted rational Cherednik algebras}\label{sec:RRCA}

The inclusion of algebras $D(W) \hookrightarrow \ZH_{\mathbf{c}}(W)$ allows us to define the \textit{restricted rational Cherednik algebra} $\overline{\H}_{\mbf{c}}(W)$ as the quotient
\begin{displaymath}
\overline{\H}_{\mbf{c}}(W) = \frac{\H_{\mathbf{c}}(W)}{D(W)_+ \cdot H_{\mathbf{c}}(W)} \;,
\end{displaymath}
where $D(W)_+$ denotes the ideal in $D(W)$ of elements with zero constant term. This algebra was originally introduced, and extensively studied, by Gordon \cite{Baby}. The PBW theorem implies that 
\begin{equation} \label{rrca_triangular}
\overline{\H}_{\mbf{c}}(W) \simeq \C [\h]^{\mrm{co} W} \otimes_\C \C W \otimes_\C \C [\h^*]^{\mrm{co}W}
\end{equation}
as $\bbC$-vector spaces. Here, 
$$
\C [\h]^{\mrm{co} W} = \C [\h] / \langle \C [\h]^W_+ \rangle
$$
is the \textit{coinvariant algebra} of $W$ and $\bbC \lbrack \fh^* \rbrack^{\mrm{co}W}$ is defined analogously. Since $W$ is a reflection group, the coinvariant algebra $\C [\h]^{\mrm{co} W}$ is of dimension $|W|$ and is isomorphic to the regular representation as a $W$-module. Thus, $\dim \overline{\H}_{\mbf{c}}(W) = |W|^3$. The restricted rational Cherednik algebra is a quotient of $\H_{\mbf{c}}(W)$ by an ideal generated by homogeneous elements and so it is also a graded algebra. This combined with the triangular decomposition (\ref{rrca_triangular}) of $\ol{\H}_\bc(W)$ implies that the representation theory of $\ol{\H}_\bc(W)$ has a rich combinatorial structure. The following is due to Gordon \cite{Baby}, based on an abstract framework by Holmes and Nakano \cite{HN}. 
First of all, note that the skew-group algebra $\C [\h^*]^{\mrm{co} W} \rtimes W$ is a graded subalgebra of $\overline{\H}_{\mbf{c}}(W)$. %

\begin{defn}
The \textit{baby Verma module} of $\overline{\H}_{\mbf{c}}(W)$ associated to a $W$-module $\lambda$ is 
$$
\Delta_\bc(\lambda) := \overline{\H}_{\mbf{c}}(W) \otimes_{\C [\h^*]^{\mrm{co} W} \rtimes W} \lambda \;,
$$
where $\C [\h^*]^{\mrm{co} W}_+$ acts on $\lambda$ as zero. 
\end{defn}

The baby Verma module $\Delta_\bc(\lambda)$ is naturally a graded $\overline{\H}_{\mbf{c}}(W)$-module, where $1 \otimes \lambda$ sits in degree zero. By studying quotients of baby Verma modules, it is possible to completely classify the simple $ \overline{\H}_{\mbf{c}}(W)$-modules. We denote by $\Irr W$ the set of simple $W$-modules (up to isomorphism). Similarly, we understand $\Irr \ol{\H}_\bc(W)$.

\begin{prop}[Gordon]\label{prop:simplebaby}
Let $\lambda, \mu \in \Irr W$.
\begin{enumerate}
\item The baby Verma module $\Delta_\bc(\lambda)$ has a simple head. We denote it by $L_\bc(\lambda)$.
\item $L_\bc(\lambda)$ is isomorphic to $L_\bc(\mu)$ if and only if $\lambda \simeq \mu$.  
\item The map $\Irr W  \rarr \Irr \ol{\H}_\bc(W)$, $\lambda \rarr L_\bc(\lambda)$, is a bijection.
\end{enumerate}
\end{prop}

The bijection in the proposition allows us to transform representation theoretic information about $\ol{\H}_\bc(W)$ into combinatorial $\bc$-dependent data about $W$. The Calogero–Moser families are the primary example of this process.

\subsection{Calogero–Moser families}\label{sec:defnCalogeroMoser}

Since the algebra $\overline{\H}_{\mbf{c}}(W)$ is finite-dimensional, it has a block decomposition $\overline{\H}_{\mbf{c}}(W) = \bigoplus_{i = 1}^k B_i$, with each $B_i$ an  indecomposable algebra. If $b_i$ is the identity element of $B_i$ then the identity element $1$ of $\overline{\H}_{\mbf{c}}(W)$ is the sum $1 = b_1 + \ds + b_k$ of the $b_i$. For each simple $\overline{\H}_{\mbf{c}}(W)$-module $L$, there exists a unique $i$ such that $b_i \cdot L \neq 0$. In this case we say that $L$ \textit{belongs to the block} $B_i$. By Proposition \ref{prop:simplebaby}, we can (and will) identify $\Irr  \overline{\H}_{\mbf{c}}(W)$ with $\Irr W$. Let $\Omega_\bc(W)$ be the set of equivalence classes of $\Irr W$ under the equivalence relation $\lambda \sim \mu$ if and only if $L_\bc(\lambda)$ and $L_\bc(\mu)$ belong to the same block. These equivalence classes are called the \textit{Calogero–Moser $\bc$-families} of $W$. 

These families have an important geometric interpretation. The image of the natural map 
\[
\ZH_{\mbf{c}}(W) / D(W)_+ \cdot \ZH_{\mbf{c}}(W) \rightarrow \overline{\H}_{\mbf{c}}(W)
\]
is clearly contained in the centre of $\overline{\H}_{\mbf{c}}(W)$. In general it is not equal to the centre of $\overline{\H}_{\mbf{c}}(W)$. However, it is a consequence of a theorem by M\"uller, see \cite[Corollary 2.7]{Ramifications}, that the primitive central idempotents of $\overline{\H}_{\mbf{c}}(W)$, the block idempotents $b_i$ above, are precisely the images of the primitive idempotents of $\ZH_{\mbf{c}}(W) / D(W)_+ \cdot \ZH_{\mbf{c}}(W)$. This shows that the natural map $\Irr W \rightarrow \Upsilon^{-1}_\bc(0)$, $\lambda \mapsto \Supp L_\bc(\lambda) = \chi_{L_\bc(\lambda)}$, factors through the Calogero–Moser partition. Here, $\Upsilon^{-1}_\bc(0)$ is considered as the set theoretic fibre over the origin $0$ of $\h/W \times \h^*/W$. In other words, we have a natural bijection between $\Omega_\bc(W)$ and $\Upsilon_\bc^{-1}(0)$. Now, recall that $\Upsilon_\bc$ is $\bbC^*$-equivariant. The only $\bbC^*$-fixed point of $\h/W \times \h^*/W$ is the origin $0$ and therefore $\Upsilon_\bc^{-1}(0) = \cX_\bc(W)^{\bbC^*}$. Hence, we can identify the Calogero–Moser families $\Omega_\bc(W)$ with the $\bbC^*$-fixed closed points of the Calogero–Moser space $\cX_\bc(W)$.

The next theorem follows from the fact that the Azumaya locus of $\H_\bc(W)$ is equal to the smooth locus of $\ZH_\bc(W)$, which in turn follows from results by Etingof–Ginzburg \cite[Theorem 1.7]{EG} and Brown (see \cite[Lemma 7.2]{Baby}).

\begin{theorem}[Etingof–Ginzburg, Brown] \label{singleton_smooth}
A $\bbC^*$-fixed closed point of $\cX_\bc(W)$ is smooth if and only if the corresponding Calogero–Moser family is a singleton, i.e. it consists only of one irreducible character of $W$.
\end{theorem}

\begin{example} \label{cherednik_0}
Consider the special case $\bc = 0$. In this case $\cX_{0}(W) = (\fh \oplus \fh^*)/W$. The quotient morphism $\fh \oplus \fh^* \rarr  (\fh \oplus \fh^*)/W$ is $\bbC^*$-equivariant and finite, hence $\cX_{0}(W)$ has only one $\bbC^*$-fixed closed point, namely the origin. In particular, there is only one Calogero–Moser family.
\end{example}

\section{Lusztig families} \label{lusztig_fam}

In this section we give a short summary of the other protagonist of this paper—Lusztig's families. We review some of the constructions involved in the definition of Lusztig families, such as truncated induction, as we will make use of these in the case-by-case analysis in sections \S\ref{type_A} to \S\ref{dihedral_groups}. For more details we refer to Lusztig's books \cite{Lusztig-characters-reductive-groups, LusztigUnequalparameters}, and also to \cite{Geck.M;Jacon.N11Representations-of-H} and \cite{Geck-left-cells-and-constructible}.

\subsection{Hecke algebras}
Throughout this section, let $(W,S)$ be a finite Coxeter system. We choose an $\bbR$-valued \word{weight function} $L$ on $(W,S)$, i.e., a function $L: W \rarr \bbR$ satisfying $L(ww') = L(w) + L(w')$ for all $w,w' \in W$ with $\ell(ww') = \ell(w) + \ell(w')$, where $\ell$ is the length function of $(W,S)$. 
Let $A \dopgleich \bbZ_W \lbrack \bbR \rbrack$ be the group ring of the additive group $\bbR$ over the subring $\bbZ_W$ of $\bbC$ generated by the values of the irreducible complex characters of $W$. This is an integral domain and we denote by $q^\alpha$ the element of $A$ corresponding to $\alpha \in \bbR$. Note that $q^\alpha q^\beta = q^{\alpha+\beta}$. Set $q_w \dopgleich q^{L(w)}$ for $w \in W$. Let $\mathcal{H} \dopgleich \mathcal{H}_L(W,S)$ be the \word{Hecke algebra} of $(W,S)$ over $A$ with respect to $L$. 
This is the free $A$-algebra with basis $\lbrace T_w \mid w \in W \rbrace$ whose multiplication is uniquely determined by the relations
\begin{equation}
T_sT_w = \left\lbrace \begin{array}{ll} T_{sw} & \tn{if } \ell(sw) > \ell(w) \\ T_{sw} + (q_s-q_s^{-1})T_w & \tn{if } \ell(sw) < \ell(w) \end{array} \right. 
\end{equation}
for all $s \in S$ and $w \in W$. 
It is a standard fact that the scalar extension $\mathcal{H}^K$ of $\mathcal{H}$ to the fraction field $K$ of $A$ is split semisimple. It is then a consequence of Tits's deformation theorem that there is a natural bijection between $\Irr W$ and $\Irr \mathcal{H}^K$. We write $E_q^\lambda$ for the simple $\mathcal{H}^K$-module corresponding to the simple $W$-module $\lambda$ under this bijection. It is also well-known that $\mathcal{H}$ is a symmetric $A$-algebra. This implies that the scalar extension $\mathcal{H}^K$ is symmetric and so by the theory in \cite[\S7]{GeckPfeiffer} there is a \word{Schur element} $\bs_\lambda \in A$ attached to every simple module $E_q^\lambda$. There is a unique element $\ba_\lambda \in \bbR_{\geq 0}$ satisfying $q^{2 \ba_\lambda} \bs_\lambda \in \bbZ_W \lbrack \bbR_{\geq 0} \rbrack$ and $q^{2\ba_\lambda}\bs_\lambda \equiv f_\lambda \ \mathsf{mod}\ \bbZ_W \lbrack \bbR_{>0} \rbrack$ for some $f_\lambda > 0$. This is called Lusztig's \word{$\ba$-invariant} of $\lambda$. The Schur elements and $\ba$-invariants are known for all Coxeter groups and all weight functions. Note that despite the notation the Schur elements $\bs_\lambda$ and the $\ba$-invariants $\ba_\lambda$ depend on $L$.

\subsection{Truncated induction}

Recall that if $I \subs S$ is any subset, then $(W_I,I)$ is naturally a Coxeter system, where $W_I$ is the group generated by $I$. This is called a (standard) \word{parabolic subgroup} of $(W,S)$. The restriction $L_I$ of our weight function $L$ to $W_I$ is a weight function on $(W_I,I)$. For any simple module $\mu$ of $W_I$, Lusztig defined the \word{truncated induction} (or \word{$\bj$-induction}) as
\begin{equation}
\bj_{W_I}^W \mu \dopgleich \sum_{ \substack{ \lambda \in \Irr W \\ \ba_\lambda = \ \ba_\mu} } \langle \Ind_{W_I}^W \mu, \lambda \rangle \lambda \;,
\end{equation}
where $\langle \Ind_{W_I}^W \mu, \lambda \rangle$ denotes the multiplicity of $\lambda$ in the induction of $\mu$ from $W_I$ to $W$. Keep in mind that the $\ba$-invariant $\ba_\mu$ is computed using the restriction $L_I$ of $L$ to $W_I$. It is shown in \cite[Lemma 3.5]{Geck-constructibles-leading} that for any $\mu \in \Irr W'$ there is a $\lambda \in \Irr W$ with $\ba_\lambda = \ba_\mu$ so that the above sum is never empty. This operation extends to a morphism $\bj_{W_I}^W: \rK_0 (W_I\tn{-}\mathsf{mod}) \rarr \rK_0(W\tn{-}\mathsf{mod})$ of Grothendieck groups. It is transitive in the sense that $\bj_{W_I}^W \circ \bj_{W_J}^{W_I} = \bj_{W_J}^W$ for $J \subs I \subs S$.

\subsection{Constructible characters and families}

Using truncated induction, Lusztig inductively defined the set $\Con_L(W)$ of \word{$L$-constructible representations} of $W$ as follows: if $W$ is trivial, then $\Con_L(W)$ consists of the unit representation, and otherwise $\Con_L(W)$ consists of the $W$-modules of the form $\bj_{W_I}^W E$ and $(\bj_{W_I}^W E) \otimes \mrm{sgn}_{W}$ for all proper subsets $I \subsetneq S$ and all $E \in \Con_{L_I}(W_I)$. Here $\mrm{sgn}_W$ is the sign representation of $(W,S)$. A key result shown by Lusztig, \cite[Proposition 22.3]{LusztigUnequalparameters}, says
\begin{equation}\label{key_fact}
\textit{for each $\lambda \in \Irr W$ there exists $E \in \Con_L(W)$ such that $\langle E, \lambda \rangle \neq 0$.} 
\end{equation}
The \word{constructible graph} is the graph $\mathcal{C}_L(W)$ with vertices $\Irr W$ and an edge between $\lambda$ and $\mu$ if and only if $\lambda \neq \mu$ and they both occur in an $L$-constructible representation of $W$. The connected components of this graph are called Lusztig's \word{$L$-families}. They define a partition of $\Irr W$. We denote the set of these families by $\lus_L(W)$. Lusztig's families are known for all finite Coxeter groups (see \cite[\S22]{LusztigUnequalparameters} and also \S\ref{type_A} to \S\ref{dihedral_groups}).

\begin{example} \label{lusztig_0}
Consider the special case $L=0$. The map $T_w \mapsto w$ extends to an algebra isomorphism from $\mathcal{H}_0(W,S)$ to the group algebra $AW$ which is compatible with the symmetrising traces. Hence, $\bs_\lambda = \frac{|W|}{\dim \lambda} \in \bbZ_W \lbrack \bbR_{>0} \rbrack$ by \cite[7.2.5]{GeckPfeiffer} and so $\ba_\lambda = 0$. This in turn immediately shows that $\bj_{W_I}^W \mu = \Ind_{W_I}^W \mu$ for any parabolic subgroup $W_I$ of $W$ and $\mu \in \Irr W_I$. We then see that there is only one constructible representation, namely the regular representation of $W$. In particular, the constructible graph is connected and there is only one Lusztig family.
\end{example}

\subsection{Calogero–Moser families vs. Lusztig families} \label{lusztig_fam_vs_cm_fam}

It is a standard fact that $W$ admits a reflection representation on the complex vector space $\fh$ of dimension equal to the size of $S$, namely the complexification of the \word{geometric representation}. The set $\Ref(W)$ of reflections then consists precisely of all conjugates of $S$ in $W$. Hence, to $W$ and a $W$-equivariant function $\bc:\Ref(W) \rarr \bbC$ we can attach the rational Cherednik algebra $\H_\bc(W)$ and have the notion of Calogero–Moser $\bc$-families $\Omega_\bc(W)$ of $\Irr W$. 
It follows from Matsumoto's lemma (see \cite[1.1.5]{Geck.M;Jacon.N11Representations-of-H})  that a weight function $L$ on $W$ is already uniquely determined by the values on the $W$-conjugacy classes of $S$, and that conversely every collection of elements $c_s \in \bbR$ for $s \in S$ with $c_s = c_t$ whenever $c_s$ and $c_t$ are conjugate defines a unique weight function on $(W,S)$. This shows that weight functions $L:W \rarr \bbR$, i.e., parameters for Hecke algebras attached to $(W,S)$, are nothing else than $W$-equivariant functions $\bc :\Ref(W) \rarr \bbR$, i.e., $\bbR$-valued parameters for rational Cherednik algebras attached to $W$. We will thus use both notions interchangeably.

\begin{leftbar}
\vspace{6pt}
Whenever we write $\bc \geq 0$, resp. $\bc > 0$, we mean that $\bc$ takes values in $\bbR_{\geq 0}$, resp. $\bbR_{>0}$. ${}_{}$Similarly, we write $L \geq 0$, resp. $L>0$, if $L(s) \geq 0$, resp. $L(s) > 0$, for all $s \in S$.  \vspace{4pt}
\end{leftbar}

We can twist by linear characters of $W$ in order to ensure that we are always in the situation $\bc \geq 0$. Namely, let $\delta : W \rightarrow \bbR^{\times}$ be a linear character. Clearly $\delta$ is uniquely defined by its values on $S$, where it is $\pm 1$. Conversely, for any assignment of $\pm 1$ to each element of $S$, such that $\delta(s) = \delta(s')$ if $s$ is conjugate to $s'$, we get a well-defined linear character of $W$. Then $T_w \mapsto \delta(w) T_w$ defines an algebra isomorphism $\mathcal{H}_L(W,S) \stackrel{\sim}{\longrightarrow} \mathcal{H}_{\delta L}(W,S)$. Given a representation $\lambda$ of $W$, ${}^{\delta} \lambda$ denotes the twist of $\lambda$ by $\delta$. It is immediate from the definition of Lusztig families that $\lambda$ and $\mu$ belong to the same $L$-family if and only if ${}^{\delta} \lambda$ and ${}^{\delta} \mu$ belong to the same $\delta L$-family. Moreover, a family $\mathcal{F}$ is $L$-cuspidal (see below) if and only if ${}^{\delta} \mathcal{F}$ is $\delta L$-cuspidal. 

Similarly, one can twist the rational Cherednik algebra by the character $\delta$, as explained in  \cite[4.6B]{BonnafeRouquier}. Again, the two representations $\lambda, \mu$ belong to the same $\bc$-family if and only if ${}^{\delta} \lambda$ and ${}^{\delta} \mu$ belong to the same $\delta \bc$-family. Moreover, a family $\mathcal{F}$ is $\bc$-cuspidal (see below) if and only if ${}^{\delta} \mathcal{F}$ is $\delta \bc$-cuspidal. Therefore, to prove Theorem \ref{thm:mainresultintro}, it suffices to make the following assumption, as in \cite{Geck.M;Jacon.N11Representations-of-H}:  

\begin{leftbar}
\vspace{4pt}
We assume that $L \geq 0$. 
\vspace{4pt}
\end{leftbar}

The following conjecture is due to Gordon–Martino \cite{GordonMartinoCM}. 

\begin{conjecture} \label{cm_equal_lusztig_conjecture}
For any finite Coxeter group $W$ and any real parameter $\bc$ we have $\Omega_\bc(W) \!=\! \lus_\bc(W)$, i.e. the Calogero–Moser $\bc$-families are the same as the Lusztig $\bc$-families.
\end{conjecture}

We note that this conjecture was formulated in \cite{GordonMartinoCM} for Weyl groups and weight functions taking values in $\bbQ_{> 0}$. Moreover, both in \cite{GordonMartinoCM} and \cite{BonnafeRouquier} it was conjectured that $\Omega_\bc(W)$ coincides with the partition of $\Irr W$ into \word{Kazhdan–Lusztig families}. Assuming Lusztig's conjectures P1 to P15 (see \cite[\S14]{LusztigUnequalparameters}), the Kazhdan–Lusztig families and the Lusztig families are equal (see \cite[Theorem 4.3]{Geck-left-cells-and-constructible}), so that the conjecture above (which is also formulated in precisely this way by Bonnafé \cite{Bonnafe:2013aa} for parameters $\bc > 0$) seems feasible. 

Let us record the following observation we obtain from Examples \ref{cherednik_0} and \ref{lusztig_0}.

\begin{lemma} \label{lus_eq_cm_at_0}
For any $W$ we have $\Omega_0(W) = \lus_0(W)$, i.e. Conjecture \ref{cm_equal_lusztig_conjecture} holds for $\bc=0$.
\end{lemma}

The work of Lusztig \cite{Lusztig-characters-reductive-groups, LusztigUnequalparameters}, Etingof–Ginzburg \cite{EG}, Gordon \cite{Baby}, Gordon–Martino \cite{GordonMartinoCM}, Martino \cite{Martino-blocks-gmpn}, the first author \cite{Bellamy:2010aa, CMPartitions}, and the second author \cite{Thiel:2014aa} shows that Conjecture \ref{cm_equal_lusztig_conjecture} holds in many cases.

\begin{theorem} \label{thm:cm_equal_lusztig}
If $W$ is of type $A$, $B$, $D$, $I_2(m)$, or $H_3$, then $\Omega_\bc(W) \!=\! \lus_\bc(W)$ for any $\bc:\Ref(W) \rarr \bbR$. 
\end{theorem}

Except for type $H_3$, which follows from \cite{Thiel:2014aa}, the proof of this theorem is also obtained here from \S\ref{type_A}, Corollary \ref{cor:LusztigCMtypeB}, Theorem \ref{thm:typeDcusp}, and Corollary \ref{dihedral_lusztig_qual_cm}.

\subsection{Cuspidal Lusztig families}

What is now relevant for us in this paper is that it can happen that a Lusztig family $\mathcal{F} \in \lus_L(W)$ is \word{$\bj$-induced} from a parabolic subgroup $W_I$ of $W$ in the sense that there is a Lusztig family $\mathcal{F}' \in \lus_{L_I}(W_I)$ such that $\bj_{W_I}^W$ induces a bijection between $\mathcal{F}'$ and $\mathcal{F}$ or between $\mathcal{F}'$ and $\mathcal{F} \otimes \mrm{sgn}_W$. Lusztig called a family \word{cuspidal} if it is not $\bj$-induced from a proper parabolic subgroup of $W$. Let $\lus_L^{\mrm{cusp}}(W) \subs \lus_L(W)$ be the set of cuspidal Lusztig families. These families are the building blocks of Lusztig families and it is most important to understand them. 

The following useful lemma is well-known.

\begin{lemma} \label{lusztig_rescaling}
For any $\alpha \in \bbR_{>0}$ we have $\Con_{\alpha L}(W)$ = $\Con_{L}(W)$, $\lus_{\alpha L}(W) = \lus_{L}(W)$ and $\lus_{\alpha L}^{\mrm{cusp}}(W) = \lus_{L}^{\mrm{cusp}}(W)$.
\end{lemma}

\begin{proof}
As in \cite[1.1.9]{Geck.M;Jacon.N11Representations-of-H} one can introduce a universal Hecke algebra $\ul{\mathcal{H}}$ over $\bbZ_W \lbrack \bbR^n \rbrack$, where $n$ is the number of $W$-conjugacy classes in $S$. The Hecke algebra $\mathcal{H}_L$ for a particular weight function $L:S \rarr \bbR$ is then obtained by specialisation of $\ul{\mathcal{H}}$. The algebra $\ul{\mathcal{H}}$ admits Schur elements $\ul{\bs}_\lambda \in \bbZ_W \lbrack \bbR^n \rbrack$ and it follows from the theory in \cite[\S7]{GeckPfeiffer} that $\ul{\bs}_\lambda$ specialises to the Schur element $\bs_\lambda$ of $\mathcal{H}_L$. From this one can deduce that the $\ba$-invariant $\ba_\lambda$ of $\mathcal{H}_{\alpha L}$ is obtained from the one of $\mathcal{H}_L$ by multiplication by $\alpha$. This immediately proves the claim.
\end{proof}

The key fact (\ref{key_fact}) implies:

\begin{lemma} \label{singleton_lusztig_not_cuspidal}
If $\mc{F} = \lbrace \lambda \rbrace$ is a Lusztig family such that $\lambda \in \Con_\bc(W)$, then $\mc{F}$ is not cuspidal.
\end{lemma}

\section{Cuspidal Calogero–Moser families} \label{cuspidal_families}

On the Calogero–Moser side we do not have anything similar to $\bj$-induction so far. However, the first author has introduced in \cite{Cuspidal} the notion of \textit{cuspidal} Calogero–Moser families. These are also minimal with respect to a certain condition, but this time they have a geometric interpretation via the Poisson structure on Calogero–Moser spaces. Despite their name, the two notions of cuspidality have, a priori, nothing in common. None the less, we will show that they coincide for all infinite families of Coxeter groups. In this paragraph we will review the foliation of Calogero–Moser spaces into symplectic leaves and the notion of cuspidal Calogero–Moser families.

\subsection{Poisson structure}

We consider again an arbitrary finite complex reflection group $(\fh,W)$. On the vector space $\fh \oplus \fh^*$ we have a natural $W$-invariant symplectic form $\omega$ defined by 
$$
\omega( (y,x), (y',x') ) \dopgleich x(y') - x'(y), \quad \forall  \ y,y' \in \fh, \ x,x' \in \fh^*.
$$
This induces a Poisson bracket $\lbrace \cdot, \cdot \rbrace$ on $\bbC \lbrack \fh \oplus \fh^* \rbrack$. Since the form $\omega$ is $W$-invariant, the Poisson bracket is $W$-invariant and restricts to the invariant ring $\bbC \lbrack \fh \oplus \fh^* \rbrack^W$ making the quotient variety $(\fh \oplus \fh^*)/W$ into a Poisson variety. 

The Calogero–Moser space $\cX_\bc(W)$ is a flat Poisson deformation of $(\fh \oplus \fh^*)/W$. The Poisson structure on $\cX_\bc(W)$ comes from the commutation in the rational Cherednik algebra at $t \neq 0$ as follows. Let $\bt$ be an indeterminate. Clearly, $\H_\bc(W) = \H_{\bt,\bc}(W)/\bt \H_{\bt,\bc}(W)$ and therefore we can lift elements $z_1,z_2 \in \ZH_\bc(W)$ to elements $\wh{z}_1,\wh{z}_2 \in \H_{\bt,\bc}(W)$. Now, define
\[
\lbrace z_1,z_2 \rbrace \dopgleich \lbrack \wh{z}_1,\wh{z}_2 \rbrack_{\bt = 0} \;,
\]
where $ \lbrack \wh{z}_1,\wh{z}_2 \rbrack$ is the commutator of $\wh{z}_1$ and $\wh{z}_2$ in $\H_{\bt,\bc}(W)$ and $\lbrack \wh{z}_1,\wh{z}_2 \rbrack_{\bt = 0}$ is the projection of this commutator to $\H_\bc(W) = \H_{\bt,\bc}(W)/\bt \H_{\bt,\bc}(W)$. This is indeed an element in $\ZH_\bc(W)$ and defines a Poisson structure on this ring. 

We recall that an ideal $I$ of an arbitrary Poisson algebra $A$ is a \word{Poisson ideal} if $\lbrace I, A \rbrace \subs I$, i.e., $I$ is stable under the Poisson bracket $\{ a, - \}$ for all $a \in A$. The \word{Poisson core} $\mathcal{P}(I)$ of an ideal $I$ of $A$ is the largest Poisson ideal contained in $I$. By a \word{Poisson prime (resp. maximal) ideal} we mean a prime (resp. maximal) ideal which is also a Poisson ideal. The Poisson core of any prime ideal is a Poisson prime ideal. We denote by $\PSpec(A)$ the set of all Poisson prime ideals of $A$ and by $\PMax(A)$ the set of all Poisson maximal ideals. 

\subsection{Symplectic leaves}

The (analytification of the) smooth part $(\cX_\bc(W))_{\mrm{sm}}$ of $\cX_\bc(W)$ is a Poisson manifold and admits a foliation into symplectic leaves; that is, a   stratification into smooth connected strata such that the rank of the bracket is maximal along strata. The strata are the \word{symplectic leaves} of the manifold (see \cite{Weinstein}). By continuing this process on the complement $\cX_\bc(W) \setminus (\cX_\bc)_{\mrm{sm}}$ we end up with a decomposition of $\cX_\bc(W)$ into symplectic leaves. Brown and Gordon \cite{PoissonOrders} have shown that the leaves obtained in this way are in fact \textit{algebraic}, i.e., locally closed in the Zariski topology and finite in number. The leaf of a closed point $\fm$ of $\cX_\bc(W)$ consists of all closed points $\fn \in \cX_\bc(W)$ such that $\fm$ and $\fn$ have the same Poisson core. Furthermore, it is shown in \textit{loc.\ cit.} that each leaf $\mathcal{L}$ is a smooth symplectic variety, and that the closure $\ol{\mathcal{L}}$ of the leaf $\mathcal{L}$ containing a closed point $\chi$ is the zero locus $\V(\mathcal{P}(\fm_\chi))$ of the Poisson core of its defining maximal ideal. This shows in particular that the closure of each symplectic leaf is an irreducible affine Poisson variety. 

\begin{lemma} \label{poisson_primes_leaves}
The set of symplectic leaves of $\cX_\bc(W)$ is naturally in bijection with the set $\PSpec(\ZH_\bc(W))$ of Poisson prime ideals of $\ZH_\bc(W)$.
\end{lemma}

\begin{proof}
Let $\mathcal{L}$ be a symplectic leaf. As we noted above, the closure $\ol{\mathcal{L}}$ is an irreducible affine variety and therefore the defining ideal $\fp_\mathcal{L} = \I(\ol{\mathcal{L}})$ is a prime ideal. Moreover, as $\ol{\mathcal{L}} = \V(\mathcal{P}(\fm_\chi))$ for any closed point $\chi$ of $\mathcal{L}$, it follows that $\fp_L = \mathcal{P}(\fm_\chi)$ is a Poisson prime ideal. The map $\mathcal{L} \mapsto \fp_\mathcal{L}$ is injective since if $\fp_\mathcal{L} = \fp_{\mathcal{L}'}$, then $\ol{\mathcal{L}} = \V(\fp_\mathcal{L}) = \V(\fp_{\mathcal{L}'}) = \ol{\mathcal{L'}}$, and this implies $\mathcal{L} = \mathcal{L}'$ as the symplectic leaves form a stratification of $\cX_\bc(W)$. Now, let $\fp$ be an arbitrary Poisson prime ideal of $\ZH_\bc(W)$. Then $\mathcal{L}_\fp \dopgleich \lbrace \fm \in \Max(\ZH_\bc(W)) \mid \mathcal{P}(\fm) = \fp \rbrace$ is a symplectic leaf by the description of symplectic leaves due to Brown and Gordon. By construction $\fp_{\mathcal{L}_\fp} = \fp$ and therefore the map $\mathcal{L} \mapsto \fp_\mathcal{L}$ is also surjective.   
\end{proof}

We immediately obtain the following.

\begin{corollary}
The set of zero-dimensional symplectic leaves of $\cX_\bc(W)$ is naturally in bijection with the set $\PMax(\ZH_\bc(W))$ of Poisson maximal ideals of $\ZH_\bc(W)$.
\end{corollary}

Analogous to Lusztig–Spaltenstein induction for nilpotent adjoint orbits of a reductive group, one can show that symplectic leaves are induced from zero-dimensional leaves for parabolic subgroups of $W$. Before we discuss this we give a short recollection about parabolic subgroups.

\subsection{Parabolic subgroups}

Recall that a \word{parabolic subgroup} of $W$ is the pointwise stabiliser $W_{\fh'}$ of a subspace $\fh'$ of $\fh$. By a theorem of Steinberg \cite[Theorem 1.5]{SteinbergDifferential} the pair $(\fh',W_{\fh'})$ is itself a complex reflection group. Moreover, $W_{\fh'}$ is the stabiliser $W_b$ of a generic point $b$ of $\fh'$. Hence, parabolic subgroups of $W$ are in fact the stabilisers of points of $\fh$. 

Define the \word{rank} of a complex reflection group $W$ to be the dimension of a faithful reflection representation of $W$ of minimal dimension. Let $W'$ be a parabolic subgroup of $W$. We write
$$
(\mathfrak{h}^{*W'})^\perp \dopgleich \{ y \in \mathfrak{h} \, | \, x(y) = 0 \, \textrm{ for all } \, x \in \mathfrak{h}^{*W'} \}.
$$ 
Then $\mathfrak{h} = \mathfrak{h}^{W'} \oplus (\mathfrak{h}^{*W'})^\perp$ is a decomposition of $\mathfrak{h}$ as a $W'$-module and $(\mathfrak{h}^{*W'})^\perp$ is a faithful reflection representation of $W'$ of minimal rank. Hence, the rank of $W'$ is $\dim (\mathfrak{h}^{*W'})^\perp$. We will always consider parabolic subgroups with this minimal reflection representation. In particular, if $\bc:\Ref(W) \rarr \bbC$ is a $W$-equivariant function, then the restriction $\bc'$ of $\bc$ to $\Ref(W')$ is a $W'$-equivariant function and we understand the rational Cherednik algebra $\H_{\bc'}(W')$ to be defined with respect to this reflection representation of $W'$.

The group $W$ acts on its set of parabolic subgroups by conjugation. Given a parabolic subgroup $W'$ the corresponding conjugacy class will be denoted $(W')$. We also require the partial ordering on conjugacy classes of parabolic subgroups of $W$ defined by $(W_1) \ge (W_2)$ if and only if $W_1$ is conjugate to a subgroup of $W_2$. The ordering is chosen in this way so that it agrees with a geometric ordering to be introduced in the next paragraph.

Finally, for a given parabolic subgroup $W'$ of $W$, we denote by $\mathfrak{h}^{W'}_{\textrm{reg}}$ the subset of $\h^{W'}$ consisting of those points whose stabiliser in $W$ is equal to $W'$. This is a locally closed subset of $\h$. We denote by $\Xi(W')$ the quotient $N_W(W') / W'$, where $N_W(W')$ is the normaliser of $W'$ in~$W$. The group $\Xi(W')$ acts freely on $\mathfrak{h}^{W'}_{\textrm{reg}}$. 

\begin{remark}
Suppose that $(W,S)$ is a Coxeter group. In \S\ref{lusztig_fam_vs_cm_fam} we already used the standard parabolic subgroups $W_I$ of $W$ for subsets $I \subs S$. Let $\fh$ be the (complexified) geometric representation of $W$ so that $(\fh,W)$ is a complex reflection group. Then $W_I$ is a parabolic subgroup of $W$ in the sense just defined. Moreover, it follows from Steinberg's theorem and \cite[Theorem 3.1]{Barcelo-Ihrig-Parabolics} that, up to conjugacy, the parabolic subgroups of $W$ are precisely the standard parabolic subgroups $W_I$. 
\end{remark}

\subsection{Parabolic subgroup attached to a symplectic leaf}

If $W'$ is a parabolic subgroup of $W$ then $\h_{\reg}^{W'}/W$ denotes the image of $\h_{\reg}^{W'}$ in $\h / W$. The symplectic leaves of $\cX_{\bc}(W)$ are natural labeled by conjugacy classes of parabolic subgroups. 

\begin{thm}\label{thm:induceleaves}
The following holds:
\begin{enum_thm}
\item \label{thm:induceleaves:para} For any symplectic leaf $\mc{L} \subs \cX_{\bc}(W)$ there exists a unique conjugacy class $W_\mathcal{L} \dopgleich (W')$ of parabolic subgroups of $W$ such that $\pi_\bc(\mc{L}) \cap \h^{W'}_{\reg} / W$ is dense in $\pi_\bc(\mc{L})$. 

\item \label{thm:induceleaves:leq} If $\mathcal{L},\mathcal{L}' \subs X_\bc(W)$ are symplectic leaves with $\mc{L} \subs \overline{\mc{L}'}$, then $W_\mathcal{L} \le W_{\mathcal{L}'}$.
\end{enum_thm}
\end{thm}

\begin{proof}
The bijection of Lemma \ref{poisson_primes_leaves} is denoted $\mf{p} \mapsto \mc{L}_{\mf{p}}$. The proof of \cite[Proposition 4.8]{Cuspidal} shows that, for each Poisson prime $\mf{p}$, there is a unique conjugacy class $(W')$ with $2 \dim \h = 2 \ \mathrm{rk} (W') + \dim \mc{L}_{\mf{p}}$ such that 
$$
\dim \pi_\bc(\mc{L}_{\mf{p}}) \cap \h_{\reg}^{W'}/W = \dim \h_{\reg}^{W'}/W.
$$
Since $\pi_\bc(\mc{L})$ is irreducible and $\dim \pi_\bc(\mc{L})  = \dim  \h_{\reg}^{W'}/W$, this implies that $ \pi_\bc(\mc{L}_{\mf{p}}) \cap \h_{\reg}^{W'}/W $ is dense in $\pi_\bc(\mc{L})$.
\end{proof}

\subsection{Cuspidal reduction I}

We recall the main results from \cite{Cuspidal}. For a closed point $\chi$ of $\cX_\bc(W)$ with defining maximal ideal $\mf{m}_{\chi}$ of $\ZH_\bc(W)$ we set 
\[
\H_{\mbf{c},\chi}(W) := \H_{\mbf{c}}(W) / \mf{m}_{\chi} \cdot \H_{\mbf{c}}(W) \;.
\]
This is a finite-dimensional $\bbC$-algebra. We call it \word{cuspidal} if $\chi$ is a zero-dimensional symplectic leaf of $\cX_\bc(W)$. 

\begin{thm} \label{parabolic_cuspidal_induction}
Let $\mathcal{L}$ be a symplectic leaf of $\cX_{\mathbf{c}}(W)$ of dimension $2l$ and $\chi$ a point on $\mathcal{L}$. Then there exists a parabolic subgroup $W'$ of $W$ of rank $\dim \h - l$ and a cuspidal algebra $H_{\mathbf{c}',\chi'}(W')$ such that
$$
\H_{\mathbf{c},\chi}(W) \simeq \Mat_{|W / W'|}(\H_{\mathbf{c}',\chi'}(W')).
$$
Moreover, there exists a functor $\Phi_{\chi', \chi}  :  \Lmod{\H_{\bc',\chi'}(W')} \stackrel{\sim}{\longrightarrow} \Lmod{\H_{\bc,\chi}(W)}$ defining an equivalence of categories such that
$$
\Phi_{\chi',\chi}(M) \simeq \Ind_{W'}^{W} \,  M \quad \forall \, M \in \Lmod{\H_{\bc',\chi'}(W')}
$$
as $W$-modules.
\end{thm}

Since there are only finitely many zero dimensional leaves in $\cX_{\mbf{c}}(W)$ the above result shows that to describe the $W$-module structure of all the simple modules for a particular rational Cherednik algebra one only needs to describe the $W'$-module structure of the cuspidal simple modules for each parabolic subgroup $W'$ of $W$. 

\subsection{Symplectic leaves and Calogero–Moser families}

As explained in \S\ref{sec:defnCalogeroMoser}, there is a natural bijection between the set $\Omega_\bc(W)$ of Calogero–Moser families and the points in $\Upsilon_\bc^{-1}(0)$. If $\fm_\mathcal{F}$ denotes the point of $\Upsilon_\bc^{-1}(0)$ corresponding to the family $\mathcal{F}$, then $\fm_\mathcal{F}$ lies on a unique symplectic leaf $\mathcal{L}_\mathcal{F}$ of $\cX_\bc(W)$. Using Theorem \ref{thm:induceleaves} we can attach a unique conjugacy class $W_\mathcal{F} \dopgleich W_{\mathcal{L}_\mathcal{F}}$ of parabolic subgroups of $W$ to $\mathcal{F}$. We define a partial ordering $\preceq$ on the Calogero–Moser families $\Omega_{\bc}(W)$ by 
$$
\mc{F} \preceq  \mc{F}' \quad \Longleftrightarrow \quad \mc{L}(\mc{F}) \subs \overline{\mc{L}(\mc{F}')}. 
$$

\begin{prop} \label{families_ordering}
The following holds for any $\mathcal{F},\mathcal{F}' \in \Omega_\bc(W)$:
\begin{enum_thm}
\item \label{families_ordering:antiref} $\mc{F} \preceq  \mc{F}' $ and $\mc{F}' \preceq  \mc{F}$ if and only if $\mc{F} = \mc{F}'$. 
\item $\mc{F} \preceq  \mc{F}' $ implies that $W_\mathcal{F} \le W_{\mathcal{F}'}$. 
\end{enum_thm}
\end{prop}

\begin{proof}
Part \ref{families_ordering:antiref} follows from directly from the definition of $\preceq $ and part (b) is a consequence of Theorem \ref{thm:induceleaves}\ref{thm:induceleaves:leq}.  
\end{proof}

We say that a Calogero–Moser family $\mathcal{F}$ is \word{cuspidal} if $\mathcal{L}_\mathcal{F}$ is a zero-dimensional leaf. By $\Omega_\bc^{\mrm{cusp}}(W)$ we denote the set of cuspidal Calogero–Moser $\bc$-families. It follows from Theorem \ref{thm:induceleaves} that $\PMax(\ZH_\bc(W)) \subs \Upsilon_\bc^{-1}(0)$. Hence, the set of zero-dimensional symplectic leaves of $\cX_\bc(W)$ is in bijection with $\Omega_\bc^{\mrm{cusp}}(W)$.

\begin{lemma} \label{singleton_cm_noncuspidal}
A singleton Calogero–Moser family is not cuspidal.
\end{lemma}

\begin{proof}
If $\mathcal{F}$ is a singleton Calogero–Moser family, then by Theorem \ref{singleton_smooth} the corresponding point $\fm_\mathcal{F}$ of $\cX_\bc(W)$ is smooth. Therefore it is contained in the unique open leaf of $\cX_\bc(W)$. Since $\dim \cX_\bc(W) > 0$, the open leaf is not zero-dimensional and hence the family is not cuspidal.
\end{proof}

The following well-known lemma is analogous to Lemma \ref{lusztig_rescaling}.

\begin{lemma} \label{cm_scaling}
For any $\alpha \in \bbC^\times$ there is a canonical algebra isomorphism $\H_\bc(W) \overset{\simeq}{\longrightarrow} \H_{\alpha \bc}(W)$, which induces an algebra isomorphism $\ol{\H}_\bc(W) \overset{\simeq}{\longrightarrow} \ol{\H}_{\alpha \bc}(W)$ and a Poisson isomorphism $\cX_\bc(W) \overset{\simeq}{\longrightarrow} \cX_{\alpha \bc}(W)$. Moreover, $\Omega_\bc(W) = \Omega_{\alpha \bc}(W)$ and $\Omega_{\bc}^{\mrm{cusp}}(W) = \Omega_{\alpha \bc}^{\mrm{cusp}}(W)$.
\end{lemma}

We can now state the main theorem of this paper

\begin{customthm}{\ref{thm:mainresultintro}}
If $W$ is of type $A,B,D$ or $I_2(m)$, then for any parameter $\bc \geq 0$ the \textit{cuspidal} Lusztig $\bc$-families of $W$ equal the \textit{cuspidal} Calogero–Moser $\bc$-families of $W$.
\end{customthm}

\begin{proof}
The Weyl groups of type $A$ are dealt with in \S\ref{type_A}. For type $B$, see Corollary \ref{thm:cuspidaltypeB}, and type $D$ is dealt with in Theorem \ref{thm:typeDcusp}. Finally, for the dihedral groups $I_2(m)$, see \S\ref{dihedral_lusztig_families}. 
\end{proof}

Based on this theorem we make the following conjecture.

\begin{customconj}{\ref{cm_cusp_lusztigz_cusp_conjecture}}
For any finite Coxeter group and any real parameter $\bc$ the \textit{cuspidal} Lusztig $\bc$-families equal the \textit{cuspidal} Calogero–Moser $\bc$-families. 
\end{customconj}

The proof of Theorem \ref{thm:mainresultintro} follows from a case-by-case analysis in sections \S\ref{type_A} to \S\ref{dihedral_groups} using several theoretical methods we develop in the next section. We will deduce in Lemma \ref{cuspidals_for_c_eq_0} that Conjecture \ref{cm_cusp_lusztigz_cusp_conjecture} holds for the special case $\bc=0$ for any $W$. Note that because of Lemma \ref{lusztig_rescaling} and Lemma \ref{cm_scaling} it is sufficient to prove the conjecture only up to multiplication of the parameter by positive real numbers.

\section{Calculating cuspidal Calogero–Moser families} \label{calculating}

To determine the cuspidal Calogero–Moser families we develop several theoretical methods—both of representation theoretic and geometric nature. On the one hand, we introduce the concept of rigid modules here and show that these always lie in a cuspidal family. On the other hand, we develop a Clifford theory for symplectic leaves. This allows us to deal with Weyl groups of type $D$ later. All this is done for complex reflection groups in general.

\subsection{Rigid modules}

The key to figuring out which Calogero–Moser families are cuspidal for Coxeter groups is the notion of \textit{rigid} $\H_\bc(W)$-modules. We show in Theorem \ref{maintheorem:rigid_cuspidal} below that every rigid module belongs to a cuspidal family. In all examples we consider it turns out that there is at most one cuspidal family. These two facts allow us to find all cuspidal families. 

\begin{definition}
A simple $\H_{\bc}(W)$-module $L$ is said to be \textit{rigid} if it is irreducible as a $W$-module. 
\end{definition}

This notion has played an important role for rational Cherednik algebras at $t = 1$, see e.g. \cite{Finitedimreps}. At $t = 0$, the second author investigated rigid modules in \cite{Thiel.UOn-restricted-ration}. Recently, they also played a prominent role in the work \cite{Ciubotaru:2015aa} of Ciubotaru on Dirac cohomology where they were called \word{one-$W$-type} modules. The terminology we adopt comes from the theory of module varieties, where it is standard. Intuitively, a rigid module is one that cannot be deformed (for fixed parameter $\bc$) to a continuous family of representation; see Lemma \ref{lem:openrep}. On the other hand, if a simple $\H_{\bc}(W)$-module is supported on a symplectic leaf of dimension greater than zero then one can deform the representation along the leaf. Therefore it is intuitively clear that rigid modules should be supported at zero dimensional leaves. Showing the precise connection between rigidity and cuspidality depends on the following theorem.

\begin{theorem} \label{parabolic_branching}
Let $W$ be a complex reflection group. Then no irreducible $W$-module is induced from a proper parabolic subgroup of $W$, i.e., $\Ind_{W'}^W \lambda$ is reducible for all parabolic subgroups $W' \subsetneq W$.
\end{theorem}

In order to give the proof of Theorem \ref{parabolic_branching}, we first give some preparatory lemmata. Let $G$ be a finite group. Given a character $\chi$ of $G$, we denote by $\ell (\chi)$ the length of $\chi$, i.e. if $\chi = \sum_{i=1}^n n_i \chi_i$ with $\chi_i \in \Irr(G)$, then $\ell (\chi) = \sum_{i=1}^n n_i$. Note that $(\chi,\chi) = \sum_{i=1}^n n_i^2$ and therefore $\sqrt{(\chi,\chi)} \leq \ell (\chi) \leq (\chi,\chi)$, where $(\cdot,\cdot)$ is the scalar product of characters.

We define the \word{branching index} of a subgroup $P$ of $G$ as 
\[
b_P(G) \dopgleich \min \lbrace \ell (\psi^G) \mid \psi \in \Irr(P) \rbrace \;,
\]
where $\psi^G \dopgleich \Ind_P^G \psi$. We say that $P$ is \word{branching} in $G$ if $b_P(G) > 1$, i.e., $\psi^G$ is reducible for all $\psi \in \Irr(P)$. We can now reformulate Theorem \ref{parabolic_branching} as saying that all proper parabolic subgroups of $W$ are branching.

\begin{lemma} \label{central_element_branching}
If $G$ has a central element which is not contained in $P$, then $P$ is branching.
\end{lemma}

\begin{proof}
Let $z \in Z(G) \setminus P$ and let $\psi \in \Irr(P)$. Note that $\,^zP = zPz^{-1} = P$ and therefore $P \cap \,^zP = P$. Similarly, we have $\,^z\psi = \psi$. Hence, $(\psi,\,^z \psi) = (\psi,\psi) = 1$ and therefore $\psi^G$ is not irreducible by \cite[10.25]{CurtisReiner}.
\end{proof}

\begin{lemma} \label{normal_subgroup_branching}
Let $N$ be a normal subgroup of $G$. Let $P$ be a subgroup of $G$ with branching index $b_P(G) > \lbrack G: N \rbrack$. Then $P \cap N$ is branching in $N$.
\end{lemma}

\begin{proof}
Suppose that $P \cap N$ is not branching in $N$. Then there exists some $\eta \in \Irr(P \cap N)$ with $\psi \dopgleich \eta^N \in \Irr(N)$. By Clifford theory for $N \unlhd G$, see \cite[Theorem 19.3]{Huppert-character-theory}, we have $(\eta^G,\eta^G) = (\psi^G,\psi^G) = \lbrack I_G(\psi) : N \rbrack$, where $I_G(\psi)$ is the inertia subgroup of $\psi$ in $G$. Hence, $\ell (\eta^G) \leq \lbrack I_G(\psi) : N \rbrack$. On the other hand, by Clifford theory for $N \cap P \unlhd P$ we have $(\eta^P,\eta^P) = \lbrack I_P(\eta):N \cap P \rbrack$. Hence, $\ell (\eta^P) \geq \sqrt{\lbrack I_P(\eta):N \cap P \rbrack}$ and therefore $\ell (\eta^G) \geq b_P(G) \cdot \sqrt{\lbrack I_P(\eta):N \cap P \rbrack}$. In total, we must have
\[
b_P(G) \leq \frac{\lbrack I_G(\psi) : N \rbrack}{\sqrt{\lbrack I_P(\eta):N \cap P \rbrack} } \leq \lbrack G:N \rbrack \;.
\]
Because of our assumption on $b_P(G)$ this is a contradiction.
\end{proof}

\begin{lemma} \label{P_branching_conjugates_branching}
Suppose that $N \unlhd G$. Then a subgroup $Q$ of $N$ is branching in $N$ if and only if all its $G$-conjugates are branching in $N$. 
\end{lemma}

\begin{proof}
This simply follows from the fact that $\Ind_{\,^g Q}^N \circ \Con_{g, Q} = \Con_{g,N} \circ \Ind_Q^N$ and that conjugation $\Con_{g,Q}$ with $g$ defines a bijection between $\Irr(Q)$ and $\Irr(\! \,^g Q)$ for all $g \in G$.
\end{proof}

For the proof of Theorem \ref{parabolic_branching} we will need the classification of complex reflection groups due to Shephard and Todd \cite{ST}, and in particular a description of the parabolic subgroups in the infinite series $G(m,m,n)$. We quickly recall the definition of these groups. Let $m,p,n \in \bbN_{>0}$ with $p$ dividing $m$ and let $\zeta \in \bbC$ be a primitive $m$-th root of unity. Then $G(m,p,n)$ is the subgroup of $\GL_n(\bbC)$ consisting of the generalised permutation matrices with entries in $\mu_m \dopgleich \langle \zeta \rangle$ such that the product of all non-zero entries is an $(m/p)$-th root of unity. The group $G(m,p,n)$ is a normal subgroup of index $p$ in $G(m,1,n)$. For a partition $\lambda$ of an integer $|\lambda| \leq n$ let $\fS_\lambda$ be the corresponding Young subgroup of the symmetric group $\fS_{|\lambda|}$. We have an obvious embedding $\fS_\lambda \times G(m,m,n-|\lambda|) \hookrightarrow G(m,m,n)$. The following lemma can be deduced from \cite[3.11]{Taylor-reflection-subgroups}.

\begin{lemma} \label{gmmn_parabolics}
Up to $G(m,1,n)$-conjugacy the parabolic subgroups of $G(m,m,n)$ are the standard parabolic subgroups $\fS_\lambda \times G(m,m,n-|\lambda|)$ for partitions $\lambda$ of $n$. 
\end{lemma}

We note that for the $G(m,m,n)$-conjugacy classes of parabolic subgroups of $G(m,m,n)$ some $G(m,1,n)$-conjugates of the above standard parabolic subgroups have to be taken into account (see \cite[3.11]{Taylor-reflection-subgroups}). For us, however, it is sufficient to know the $G(m,1,n)$-conjugacy classes because of Lemma \ref{P_branching_conjugates_branching}. By
Lemma \ref{gmmn_parabolics} the maximal parabolic subgroups of $G(m,m,n)$ are up to $G(m,1,n)$-conjugacy of the form $\fS_k \times G(m,m,n-k)$ for $1 \leq k \leq n$.

\begin{proof}[Proof of Theorem \ref{parabolic_branching}]

Clearly, we can assume that $W$ acts irreducibly on $\fh$ and that $P$ is a maximal parabolic subgroup. It is well-known (see \cite[Corollary 3.24]{LehrerTaylor}) that the centre $Z(W)$ of $W$ is a cyclic group $\Z_{\ell} = \langle \sigma \rangle$. If $\sigma \neq 1$, then $\sigma$ fixes only the origin and so $\sigma \notin W'$ for any proper parabolic subgroup of $W$. Hence, if $|Z(W)|>1$, then the claim holds by Lemma \ref{central_element_branching}. The classification of irreducible complex reflection groups shows that $|Z(W) | = 1$ implies that $W \simeq G(m,m,n)$ for some $m,n$. By Lemma \ref{gmmn_parabolics} and Lemma \ref{P_branching_conjugates_branching} we can assume that $P = \fS_k \times G(m,m,n-k)$, where $1 \leq k \leq n$. Let $\lambda \in \Irr(P)$.

We assume first that $m > 1$. The module $\pi_\lambda$ is isomorphic to $\pi_\lambda’ \boxtimes \pi_\mu$ for some $\pi_\lambda’ \in \Irr (\s_k)$ and $\pi_\mu \in \Irr (G(m,m,n-k))$. Note that $P \subset G(m,m,k) \times G(m,m,n-k) \subset W$. 

If $k > 1$ then it suffices to show that $\Ind_{\fS_k}^{G(m,m,k)} \pi_\lambda'$ is not irreducible. That is, we may assume $k = n$. The symmetric group $\s_n$ is a quotient of $G(m,m,n)$, the morphism given by sending an element to the underlying permutation. Then we may consider $\pi_\lambda'$ as an irreducible $G(m,m,n)$-module $\pi_\lambda''$. Clearly $\pi_\lambda'' |_{\s_n} = \pi_\lambda'$. Hence $\left[ \Ind_{\s_n}^{G(m,m,n)} \pi_\lambda' : \pi_\lambda'' \right] \ge 1$. On the other hand, $\dim \Ind_{\s_n}^{G(m,m,n)} \pi_\lambda' = m^{n-1} \dim \pi_\lambda'$. Hence it is not irreducible.
 
In the case $k = 1$, we have $P = G(m,m,n-1) \subset G(m,m,n)$. Let $Q \dopgleich G(m,1,n-1) \subs G(m,1,n)$ and note that $P = Q \cap G(m,m,n)$. If we can show that $b_Q(G(m,1,n)) \geq m+1$, then Lemma \ref{normal_subgroup_branching} shows that $P$ is branching in $G(m,m,n)$. But this follows from the branching rule (\cite[Theorem 10]{Pushkarev:wreath}) which shows that, when viewing $\lambda$ as an $m$-multipartition, we have at least $m+1$ constituents in $\Ind_{G(m,1,n-1)}^{G(m,1,n)}\pi_\lambda$ obtained by adding boxes to~$\lambda$.
 
Finally, we need to deal with the case $m = 1$, i.e. $W = \s_n$. In this case we have $P = \fS_k \times \fS_{n-k}$ and it is known that $\Ind_{\s_k \times \s_{n-k}}^{\s_n} \pi_\lambda \boxtimes \pi_\mu = \sum_{\nu} c^{\nu}_{\lambda,\mu} \pi_\nu$, where $c^{\nu}_{\lambda,\mu}$ are the Littlewood-Richardson coefficients. We need to show that $\sum_{\nu} c^{\nu}_{\lambda,\mu} > 1$. Presumably, this is well-known. We will deduce it from the fact that for the Weyl group of Type  $B_n$ we have
\begin{equation}\label{eq:typeBind}
\Ind_{B_k \times B_{n-k}}^{B_n} \pi_{(\lambda^{(1)},\lambda^{(2)})} \boxtimes \pi_{(\mu^{(1)},\mu^{(2)})} = \sum_{(\nu^{(1)},\nu^{(2)})} c^{\nu^{(1)}}_{\lambda^{(1)},\mu^{(1)}} c^{\nu^{(2)}}_{\lambda^{(2)},\mu^{(2)}} \pi_{(\nu^{(1)},\nu^{(2)})} \;.
\end{equation}
Take $\lambda^{(1)} = \lambda, \lambda^{(2)} = \emptyset, \mu^{(1)} = \mu$ and $\mu^{(2)} = \emptyset$. Then (\ref{eq:typeBind}) implies that it suffices to show that $\Ind_{B_k \times B_{n-k}}^{B_n} \pi_{(\lambda^{(1)},\lambda^{(2)})} \boxtimes \pi_{(\mu^{(1)},\mu^{(2)})}$ is not an irreducible $B_n$-module. But $B_n$ contains a non-trivial central element that does not belong to either $B_{n-k}$ or $B_k$. This implies by Lemma \ref{central_element_branching} that the induced module is not irreducible.
\end{proof}

\begin{prop}\label{prop:rigidxykill}
If $L$ is rigid, then $L \simeq L_\bc(\lambda)$ is a $\overline{\H}_{\bc}(W)$-module (isomorphic to $\lambda$ as a $W$-module), for some $\lambda \in \Irr W$. 
\end{prop}

\begin{proof}
If $L$ is not a simple $\overline{\H}_{\bc}(W)$-module, then either the set-theoretic support of $L$ as a $\C[\h]$-module is not contained in $\{ 0 \}$, or the set-theoretic support of $L$ as a $\C[\h^*]$-module is not contained in $\{ 0 \}$. Without loss of generality, we assume that the set-theoretic support of $L$ as a $\C[\h]$-module is not contained in $\{ 0 \}$. Thus, there exists some $b \neq 0$ in $\h$ such that $\mf{m}_b \cdot L \neq 0$, where $\mf{m}_b$ is the maximal ideal defining $b$ in $\h$. The stabiliser $W_b$ of $b$ is a proper subgroup of $W$. Thus, the Bezrukavnikov–Etingof isomorphism, see \cite[Theorem 4.3]{Cuspidal}, implies that $L \simeq \Ind_{W_b}^W L'$ for some $\H_{\bc'}(W_b)$-module $L'$. By Theorem \ref{parabolic_branching}, this implies that $L$ is not rigid. Thus, $L$ is a simple $\overline{\H}_{\bc}(W)$-module. The simple $\overline{\H}_{\bc}(W)$-modules are of the form $L_\bc(\lambda)$ and $\lambda$ always appears in the restriction to $W$ of $L_\bc(\lambda)$ with non-zero multiplicity. The result follows. 
\end{proof}

\begin{remark} \label{remark:rigidxykill}
Proposition \ref{prop:rigidxykill} implies that if $L$ is a rigid module then $\h \cdot L = 0 = \h^* \cdot L$. In particular, every rigid module is of ``one-$W$-type'', as recently defined in \cite{Ciubotaru:2015aa}.
\end{remark}
 
 The following lemma explains our choice of terminology since it is standard in finite-dimensional representation theory to say that a simple module $L$ of dimension $d$ for a finite-dimensional algebra $A$ is \textit{rigid} if the set of points $M$ in the representation scheme $\Rep_{d}(A)$ satisfying $M \simeq L$ is open. 

\begin{lemma}\label{lem:openrep}
Let $L$ be a rigid $\H_\bc(W)$-module and set $d := \dim L$. Let $\Rep_d(\H_\bc(W))$ be the scheme parameterizing all $d$-dimensional representations of $\H_{\bc}(W)$. Let $X$ be the set of points $M$ in $\Rep_d(\H_{\bc}(W))$ such that $M \simeq L$. Then $X$ is a connected component of $\Rep_d(\H_{\bc}(W))$.
\end{lemma}

\begin{proof}
Let $S$ be a reduced, irreducible affine $\C$-variety and $\mc{F}$ a flat family of $\H_{\bc}(W)$-modules over $S$ such that the fiber $\mc{F}_{s_0}$ is isomorphic to $L$, for some $s_0 \in S$. Then it suffices to prove that $\h$ and $\h^*$ act identically by zero on $\mc{F}$. Since $\C[S]$ is a domain, its radical is zero, and hence it suffices to show that $\h$ and $\h^*$ act as zero on every fiber $\mc{F}_s$ of $\mc{F}$ for $s \in \MaxSpec(\C[S])$. We may consider $\mc{F}$ as a flat family of $\C[\h] \rtimes W$-modules instead and prove the claim in this setting (repeating the argument for $\C[\h^*] \rtimes W$). Then the claim is a consequence of Theorem \ref{parabolic_branching} together with the (easy) classification of simple $\C[\h] \rtimes W$-modules. Firstly, since $S$ is connected $\mc{F}_s \simeq L$ as a $W$-module for all $s$. This is well-known and follows for instance from \cite[Corollary 1.4]{GabrielOpen}. Therefore, it suffices to show that if $M$ is any simple $\C[\h] \rtimes W$-module such that $\h^* \subset \C[\h]$ does not act identically zero, then $M \not\simeq L$. If $\h^*$ does not act identically zero then there exists a non-zero character $\chi : \C[\h] \rightarrow \C$ such that $M_{\chi} = \{ m \in M \ |  \ x \cdot m = \chi(x) m, \ \forall x \in \h^* \}$ is non-zero. Let $W' \subsetneq W$ be the stabilizer of $\chi \in \h$. Since $M$ is simple, $M_{\chi}$ is a simple $W'$-module and $M \simeq \Ind_{\C[\h] \rtimes W'}^{\C[\h] \rtimes W} M_{\chi}$. By Theorem \ref{parabolic_branching}, $M$ is not irreducible. Hence $M \not\simeq L$ as required. 

To deduce the statement of the lemma, take $S$ to be any irreducible component (with reduced scheme structure) of $\Rep_d(\H_{\bc}(W))$ containing $L$.  
\end{proof}

Notice that Lemma \ref{lem:openrep} shows that the set of all rigid modules in $\Rep_d(\H_{\bc}(W))$ is open. In general, the connected component $X$ has a very non-trivial scheme structure. This can be seen from Voigt's Lemma \cite{GabrielOpen} which implies that 
\begin{equation}\label{eq:Voigt}
\dim X - \dim X_{\mathrm{red}} = \dim \Ext^1_{\H_{\bc}(W)}(L,L) \;. 
\end{equation}
One can compute, using the projective resolution (2.5) of \cite[page 259]{EG}, that for a rigid module $L$ we have
$$
\Ext^{\bullet}_{\H_{\bc}(W)}(L,L) \simeq \wedge^{\bullet} \ V \otimes_W \End_{\C}(L) \;,
$$
where $V = \h \oplus \h^*$. In particular, it is easy to construct examples of rigid modules where the right hand side of (\ref{eq:Voigt}) is strictly positive. Also, the variety $\Rep_d(\H_{\bc}(W))$ can have many connected components. This can be seen, for instance, by considering the case $\bc = 0$. \\

Via the bijection $\Irr W \rarr \Irr \ol{\H}_\bc(W)$ given by Proposition \ref{prop:simplebaby}, the element $\lambda \in \Irr W$ is said to be \word{$\bc$-rigid} if $L_\bc(\lambda)$ is a rigid $\H_\bc(W)$-module. The following is the main theorem of this section. 

\begin{customthm}{\ref{maintheorem:rigid_cuspidal}} 
Let $W$ be a complex reflection group. If $\lambda \in \Irr W$ is $\bc$-rigid, then $\lambda$ lies in a cuspidal Calogero--Moser $\bc$-family.
\end{customthm}

\begin{proof}
Let $\mathcal{F}$ be the Calogero–Moser $\bc$-family of $L_\bc(\lambda)$ and let $\chi$ be the corresponding point of $\cX_\bc(W)$. Suppose that $\mathcal{F}$ is not cuspidal. Then by Theorem \ref{parabolic_cuspidal_induction} there is a parabolic subgroup $W'$ of $W$, a cuspidal symplectic leaf $\chi'$ of $\cX_{\bc'}(W')$, and an equivalence $\Phi_{\chi', \chi}  :  \Lmod{\H_{\bc',\chi'}(W')} \stackrel{\sim}{\longrightarrow} \Lmod{\H_{\bc,\chi}(W)}$ such that $\Phi_{\chi',\chi}(M) \simeq \Ind_{W'}^{W} \,  M$ as $W$-modules for all $M \in \Lmod{\H_{\bc',\chi'}(W')}$. In particular, there must exist a $W'$-module $M$ with $\Ind_{W'}^W M \simeq L_\bc(\lambda) \simeq \lambda$. But this is not possible by Theorem \ref{parabolic_branching}.
\end{proof}

Of course, the major advantage of rigid modules is that they are easily detected.

\begin{lemma} \label{rigidity_equation_lemma}
Let $\lambda:G \rarr \GL_r(\bbC)$ be an irreducible representation of $W$. Then $L_\bc(\lambda)$ is a rigid module for $\H_\bc(W)$ if and only if
\begin{equation} \label{rigidity_equation}
\sum_{s \in \Ref(W)} \bc(s)(y,\alpha_s)(\alpha_s^\vee,x) \lambda(s) = 0
\end{equation}
for all $y \in \fh$ and $x \in \fh^*$.
\end{lemma}

\begin{proof}
The module $L_\bc(\lambda)$ is rigid if and only if it is as a $W$-module isomorphic to $\lambda$. Moreover, by Remark \ref{remark:rigidxykill} both $\fh$ and $\fh^*$ act trivially on $L_\bc(\lambda)$. Hence, $L_\bc(\lambda)$ is rigid if and only if the representation $\wh{\lambda}: T(\fh \oplus \fh^*) \rtimes W \rarr \Mat_r(\bbC)$ with $\fh$ and $\fh^*$ acting trivially and $W$ acting by $\lambda$ descends to $\H_\bc(W)$. This is the case if and only if $\wh{\lambda}(\lbrack y,x \rbrack) = 0$, and this is equivalent to the asserted equation.
\end{proof}

\begin{lemma} \label{cuspidals_for_c_eq_0}
For any $W$ we have $\Omega_0^{\mrm{cusp}}(W) = \lus_0^{\mrm{cusp}}(W)$, i.e. Conjecture \ref{cm_cusp_lusztigz_cusp_conjecture} holds for $\bc = 0$. 
\end{lemma}

\begin{proof}
Recall from Lemma \ref{lus_eq_cm_at_0} that $\Omega_0(W) = \lus_0(W) = \lbrace \Irr W \rbrace$. Furthermore, recall from Example \ref{lusztig_0} that truncated induction is for $\bc=0$ just usual induction, i.e. $\bj_{W_I}^W = \Ind_{W_I}^W$ for a parabolic subgroup $W_I$ of $W$. Now, if the unique Lusztig family were not cuspidal, then the irreducible characters of $W$ would all be induced from a proper parabolic subgroup of $W$, but this is not possible by Theorem \ref{parabolic_branching}. Hence, the unique Lusztig family is cuspidal. On the other hand, all $\lambda \in \Irr W$ are rigid for $\bc=0$ by Lemma \ref{rigidity_equation_lemma}. Hence, each $\lambda \in \Irr W$ lies in a cuspidal Calogero–Moser family by Theorem \ref{maintheorem:rigid_cuspidal}. As there is just one Calogero–Moser family, this one is cuspidal. 
\end{proof}

\begin{remark} \label{conjectureB_remark}
Ciubotaru \cite{Ciubotaru:2015aa} has recently classified the rigid $\H_\bc(W)$-modules for all Weyl groups and all parameters. We will independently obtain this classification for non-exceptional Coxeter groups from sections \S\ref{type_A} to \S\ref{dihedral_groups}. Ciubotaru  furthermore shows for all Weyl groups at equal parameters—except $E_7$—that the rigid modules always lie in a single Calogero–Moser family, and that this family contains the (unique) cuspidal Lusztig family; for $F_4$ and $E_6$ this is in fact an equality. Using Theorem \ref{maintheorem:rigid_cuspidal}, this shows that one direction of Conjecture \ref{cm_cusp_lusztigz_cusp_conjecture} also holds for $F_4$ and $E_6$ for equal parameters. However, a classification of the cuspidal symplectic leaves is still open in all cases not covered by our Theorem \ref{thm:mainresultintro}.
\end{remark}

\subsection{Cuspidal reduction II}

For a conjugacy class $(W')$ of parabolic subgroups of $W$ we denote by $\PSpec_{(W')}(\ZH_\bc(W))$ the subset of $\PSpec (\ZH_\bc(W))$ of Poisson prime ideals $\mf{p}$ with $W_{\mathcal{L}_{\mf{p}}} = (W')$. This set might be empty. \\

If $W \subs G \subs N_{\GL(\fh)}(W)$ is a finite subgroup such that $\bc \!:\! \Ref(W) \rarr \bbC$ is $G$-invariant, then $G$ acts on $\H_\bc(W)$ by algebra automorphisms. This induces an action of $G$ by Poisson algebra automorphisms on $\ZH_\bc(W)$ and, since $W$ acts trivially, this action factors through $G/W$. Applied to a parabolic subgroup $W'$ of $W$ and an arbitrary $W$-invariant function $\bc:\Ref(W) \rarr \bbC$ this shows that $\Xi(W')$ acts on $\ZH_{\bc'}(W')$. Here, and below, $\bc'$ denotes the restriction of $\bc$ to $\Ref(W) \cap W' = \Ref(W')$. \\

The following was shown by Losev \cite[Theorem 1.3.2]{LosevSRAComplete}.

\begin{thm} \label{pmax_parabolic_orbit}
Let $W'$ be a parabolic subgroup of $W$. The group $\Xi(W')$ acts on the set $\PMax(\ZH_{\bc'}(W'))$ such that there is a bijection 
$$
 \PSpec_{(W')}(\ZH_{\bc}(W)) \stackrel{1:1}{\longleftrightarrow} \PMax(\ZH_{\bc'}(W')) / \Xi(W')\;.
$$ 
\end{thm}

Losev considers in \cite{LosevSRAComplete} a different completion of the rational Cherednik algebra than the one used in \cite{Cuspidal} (which is based on a construction by Bezrukavnikov and Etingof). Therefore we will now show that Theorem \ref{pmax_parabolic_orbit} still holds in the context of Bezrukavnikov–Etingof completions. 

Fix a parabolic subgroup $W'$ of $W$ and let $N := N_{W}(W')$. Let $U$ be an affine open subset of $\h / W$ such that $U \cap \h_{\reg}^{W'}/W$ is closed, but non-empty, in $U$. Let $V$ denote the preimage of $U$ in $\h$. Then $V$ is $W$-stable and $V_{W'} = \h_{\reg}^{W'} \cap V$ is closed in $V$. Let $\mf{k}$ denote the $W'$-module complement to $\h^{W'}$ in $\h$. It is an $N$-module.  

Let $A \dopgleich \C[U]$ and set $Z \dopgleich A \otimes_{\C[\h]^W} \ZH_{\bc}(W)$. The prime ideal of $A$ defining $U \cap \h_{\reg}^{W'}/W$ is denoted $\mf{q}$. Let $\widehat{A}_{\mf{q}}$ be the completion of $A$ along $\mf{q}$ and set $\widehat{\cX}_{\bc}(W) := \Spec (\widehat{A}_{\mf{q}} \otimes_A Z)$. Morally speaking, $\widehat{\cX}_{\bc}(W)$ should be thought of as the formal neighbourhood of $\pi^{-1}(\h_{\reg}^{W'}/W)$ in $\cX_{\bc}(W)$. However, since $Z$ is not a finite $A$-module, this is not strictly true. 

Let $A' = \C[ \mf{k}/{W'} \times V_{W'}]$ and $\mf{q}'$ the prime ideal defining $\{ 0 \} \times V_{W'}$ in $\mf{k}/{W'} \times V_{W'}$. Then
$$
\widehat{\cX}_{\bc'}(W',V) \dopgleich \Spec \left( \widehat{A}_{\mf{q}'}' \otimes_{A'} \ZH_{\bc'}(W') \otimes \C [ T^* V_{W'} ] \right),
$$
where $T^* V_{W'}$ is the cotangent bundle of $V_{W'}$. The group $\Xi(W')$ acts on $\widehat{\cX}_{\bc'}(W',V)$. The following is an analogue of the isomorphism $\Theta$ in section 3.7 of  \cite{BE}; a complete proof is given in  \cite{CMMotives}.

\begin{thm}\label{thm:generalcompletion}
There is an isomorphism of affine Poisson varieties
$$
\Phi :  \widehat{\cX}_{\bc}(W) \stackrel{\sim}{\longrightarrow} \widehat{\cX}_{\bc'}(W',V ) / \ \Xi(W'). 
$$
\end{thm}

In order to deduce Theorem \ref{pmax_parabolic_orbit} from Theorem \ref{thm:generalcompletion}, we require the following lemma. 

\begin{lem}\label{lem:leavescomplete}
The map $\mf{p} \mapsto \widehat{A}_{\mf{q}} \otimes_A  \mf{p}$ defines a bijection between $\PSpec_{(W')}(\cX_{\bc}(W))$ and the set of Poisson prime ideals of $\widehat{A}_{\mf{q}} \otimes_A Z$ of height $2 \dim \mf{k}$.
\end{lem} 

\begin{proof}
First, we must show that $\widehat{A}_{\mf{q}} \otimes_A  \mf{p}$ is prime in $ \widehat{A}_{\mf{q}} \otimes_A Z$. Let $Y = U \cap \h_{\reg}^{W'}/W$ and denote by $\overline{\pi(\mc{L}_{\mf{p}})}$ the closure of $\pi(\mc{L}_{\mf{p}}) \cap U$ in $U$. Recall that $ \pi(\mc{L}_{\mf{p}}) \cap \h_{\reg}^{W'}/W $ is dense in $\h_{\reg}^{W'}/W$. This implies that $\overline{\pi(\mc{L}_{\mf{p}})} \cap Y$ is dense in the closed, irreducible set $Y$, i.e. $\overline{\pi(\mc{L})} \cap Y = Y$. Similarly, $\overline{\pi(\mc{L}_{\mf{p}})} \cap Y$ is dense in $\overline{\pi(\mc{L}_{\mf{p}})}$. Thus, $\overline{\pi(\mc{L}_{\mf{p}})} = Y$. This implies that $\mf{q} \subset \mf{p} \cap A$ and hence $Z \cdot \mf{q} \subset \mf{p}$. Since $\widehat{A}_{\mf{q}}$ is flat over $A$, we have a short exact sequence 
$$
0 \rightarrow \widehat{A}_{\mf{q}} \otimes_A \mf{p} \rightarrow \widehat{A}_{\mf{q}} \otimes_A Z \rightarrow \widehat{A}_{\mf{q}} \otimes_A (Z / \mf{p}) \rightarrow 0.
$$
The order filtration on $\H_{\bc}(W)$ defines an increasing filtration $\mc{F}_i Z$ on $Z$ such that each piece is a coherent $A$-module. This restricts to a filtration on $\mf{p}$ and we have a short exact sequence $0 \rightarrow \mc{F}_i \mf{p} \rightarrow \mc{F}_i Z \rightarrow \mc{F}_i Z /  \mc{F}_i \mf{p} \rightarrow 0$ of coherent $A$-modules. Since tensor products commute with colimits, 
$$
\widehat{A}_{\mf{q}} \otimes_A (Z / \mf{p}) = \lim_{i \rightarrow \infty} \widehat{A}_{\mf{q}} \otimes_A (\mc{F}_i Z /  \mc{F}_i \mf{p}). 
$$
But $Z \cdot \mf{q} \subset \mf{p}$ implies that 
$$
\widehat{A}_{\mf{q}} \otimes_A (\mc{F}_i Z /  \mc{F}_i \mf{p}) = \lim_{\infty \leftarrow m} \frac{(\mc{F}_i Z /  \mc{F}_i \mf{p})}{\mf{q}^m (\mc{F}_i Z /  \mc{F}_i \mf{p})} = \mc{F}_i Z /  \mc{F}_i \mf{p}. 
$$
Thus, $\widehat{A}_{\mf{q}} \otimes_A (Z / \mf{p}) = Z / \mf{p}$ is a domain and $\widehat{A}_{\mf{q}} \otimes_A  \mf{p}$ is prime. It is clearly Poisson; see \cite[Lemma 3.5]{Cuspidal}. Moreover, the fact that $ \widehat{A}_{\mf{q}} \otimes_A (Z / \mf{p}) = Z/ \mf{p}$ shows that $\widehat{A}_{\mf{q}} \otimes_A  \mf{p}_1 = \widehat{A}_{\mf{q}} \otimes_A  \mf{p}_2$ if and only if $\mf{p}_1 = \mf{p}_2$. Lemma 3.3 of \textit{loc.\ cit.} says that $\mathrm{ht}(\widehat{A}_{\mf{q}} \otimes_A  \mf{p}) = \mathrm{ht}(\mf{p})$, which equals $2 \mathrm{rk} (W')$. Thus, the map we have written down is injective.   

On the other hand, if $\mf{p}'$ is a Poisson prime in $\widehat{A}_{\mf{q}} \otimes_A Z$ of height $2 \dim \mf{k}$, then Lemmata 3.3 and 3.5 of \textit{loc.\ cit.} say that $\mf{p} := \mf{p}' \cap Z$ is a Poisson prime of height $2 \dim \mf{k}$. Therefore, we just need to show that $Y \cap \pi(\mc{L}_{\mf{p}})$ is dense in $Y$. The prime $\mf{p}$ belongs to $\PSpec_{(W'')}(\cX_{\bc}(W))$ for some parabolic $W''$ of $W$ of the same rank as $W'$. The sets $U \cap \h^{W''}_{\reg}/W$ and $Y$ are disjoint if $W'' \notin (W')$, which implies that the image in $\widehat{A}_{\mf{q}}$ of the ideal defining $U \cap \h^{W''}_{\reg}/W$ is the whole of $\widehat{A}_{\mf{q}}$. Therefore the image of $\mf{p} \cap A$ in $\widehat{A}_{\mf{q}}$ would also be the whole of $ \widehat{A}_{\mf{q}}$ if $W'' \notin (W')$. But since $ \widehat{A}_{\mf{q}} \otimes_A  \mf{p}$ is contained in $\mf{p}'$, this cannot happen and thus $W'' \in (W')$ as required.  
\end{proof}

\begin{proof}[Proof of Theorem \ref{pmax_parabolic_orbit}]
Since the symplectic structure on $T^* V_{W'}$ is non-degenerate, the only Poisson prime in $\C [ {T}^* V_{W'} ]$ is the zero ideal. Therefore, every Poisson prime in $\ZH_{\bc'}(W') \otimes \C [ T^* V_{W'} ]$ has height at most $2 \dim \mf{k}$ and the Poisson primes of height $2 \dim \mf{k}$ are in bijection with the Poisson maximal ideals of $\ZH_{\bc'}(W')$. Repeating the arguments of Lemma \ref{lem:leavescomplete}, there is a bijection between the Poisson primes in $\ZH_{\bc'}(W') \otimes \C [ T^* V_{W'} ]$ of height $2 \dim \mf{k}$ and the Poisson primes of height $2 \dim \mf{k}$ in  $\C \left[ \widehat{\cX}_{\bc'}(W',V) \right]$.

By Lemma \ref{lem:leavescomplete} and Theorem \ref{thm:generalcompletion}, the set $\PSpec_{(W')}(\cX_{\bc}(W)) $ is in bijection with the symplectic leaves in $\widehat{\cX}_{\bc'}(W',V) / \ \Xi(W')$  of dimension $2 (\dim \h - \dim \mf{k})$. Since $\Xi(W')$ acts freely on $V_{W'}$ it also acts freely on $\widehat{\cX}_{\bc'}(W',V)$. Therefore the symplectic leaves in $\widehat{\cX}_{\bc'}(W',V) / \ \Xi(W')$  of dimension $2 (\dim \h - \dim \mf{k})$ are in bijection with the $\Xi(W')$-orbits of symplectic leaves in $\widehat{\cX}_{\bc'}(W',V)$ of dimension $2 (\dim \h - \dim \mf{k})$. But, as explained above, this is the same as the $\Xi(W')$-orbits of Poisson maximal ideals $\ZH_{\bc'}(W')$. 
\end{proof}

\subsection{Clifford theory} \label{clifford_theory}

Throughout this section we fix an irreducible complex reflection group $(\fh,W)$. Moreover, we assume that there exists a normal subgroup $K \lhd W$ such that $K$ acts, via inclusion in $W$, on $\h$ as a complex reflection group (though $\h$ need not be irreducible as a $K$-module). Since $K$ is normal in $W$, the group $W$ acts on $\Ref(K)$ by conjugation. Let us fix  a $W$-equivariant function $\bc : \Ref(K) \rightarrow \C$. We extend this to a $W$-equivariant function $\bc : \Ref(W) \rightarrow \C$ by setting $\bc(s) = 0$ for all $s \in \Ref(W) \smallsetminus \Ref(K)$. A $K$-equivariant function on $\Ref(K)$ is not always $W$-equivariant.
 For our choice of parameter $\bc$, the inclusion $K \hookrightarrow W$ extends to an algebra embedding $\H_{\bc}(K) \hookrightarrow \H_{\bc}(W)$, which is the identity on $\h$ and $\h^*$.  Let $\Gamma = W / K$. As explained in \cite[Section 4.1]{CMPartitions}, the group $W$ acts on $\H_{\bc}(K)$ by conjugation. Thus, it acts on $\ZH_{\bc}(K)$. This action factors through $\Gamma$. 
 
 We will require the following lemma. 

\begin{lem}\label{lem:stupidremark}
Under the graded $W$-module identification $\H_{\bc}(W) = \bigoplus_{w \in W} \C[\h] \otimes \C[\h^*] \otimes w$, every non-zero element $z = \sum_{w \in W} z_w \cdot w \in \ZH_{\bc}(W)$ satisfies $z_1 \neq 0$. 
\end{lem}

\begin{proof}
A reformulation of the PBW property is that, under the filtration $\mc{F}_i \H_{\bc}(W)$ putting $\h$ and $\h^*$ in degree one, $W$ in degree zero and $\mc{F}_{-1}:= 0$, the associated graded $\mathrm{gr}_{\mc{F}} \H_{\bc}(W)$ equals $\C[\h \oplus \h^*] \rtimes W$. An easy induction on $k$ shows that $\mc{F}_k \H_{\bc}(W) = (\C[\h] \otimes \C[\h^*])_{\le k} \otimes \C W$ as a $W$-module, where $(\C[\h] \otimes \C[\h^*])_{\le k}$ is the sum of all graded pieces of degree at most $k$. Then the short exact sequences $0 \rightarrow \mc{F}_{k-1} \rightarrow \mc{F}_k \rightarrow \mc{F}_{k} / \mc{F}_{k-1} \rightarrow 0$ can be identified, as short exact sequences of $W$-modules, with
$$
0 \rightarrow  (\C[\h] \otimes \C[\h^*])_{\le k-1} \otimes \C W \rightarrow  (\C[\h] \otimes \C[\h^*])_{\le k} \otimes \C W \rightarrow  (\C[\h] \otimes \C[\h^*])_{k} \otimes \C W \rightarrow 0.
$$
The image of $\ZH_{\bc}(W)$ under $\gr_{\mc{F}}$ equals $\C[\h \times \h^*]^W$. Therefore, $z = \sum_{w \in W} z_w \cdot w \in \mc{F}_k \smallsetminus \mc{F}_{k-1}$ then its (non-zero!) image in $ (\C[\h] \otimes \C[\h^*])_{k} \otimes \C W$ belongs to $ (\C[\h] \otimes \C[\h^*])_{k}^W$. In particular, $z_1 \neq 0$. 
\end{proof}

\begin{prop}\label{prop:ZWinZK}
The centre $\ZH_{\bc}(W)$ of $\H_{\bc}(W)$ equals the subalgebra $\ZH_{\bc}(K)^{\Gamma}$ of $\ZH_{\bc}(K)$. Moreover, the embedding $\ZH_{\bc}(W) \hookrightarrow \ZH_{\bc}(K)$ is as Poisson algebras. 
\end{prop}

\begin{proof}
Clearly, $\ZH_{\bc}(W) \cap \ZH_{\bc}(K) \subseteq  \ZH_{\bc}(K)^W$. Therefore, we just need to show that $\ZH_{\bc}(W) \subset \ZH_{\bc}(K)$. Fix coset representatives $1 = w_1, \ds, w_{\ell}$ of $K$ in $W$. Then $\H_{\bc}(W) = \bigoplus_{i = 1}^{\ell} \H_{\bc}(K) w_i$ as a left $\H_{\bc}(K)$-module. Let $z = \sum_{i = 1}^{\ell} z_i w_i$ denote an element in $\ZH_{\bc}(W)$ with $z_i \in \H_{\bc}(K)$ for all $i$. We wish to show that $z_i = 0$ for $i \neq 1$. Let $f \in \H_{\bc}(K)$. Then 
$$
[f,z] = \sum_{i = 1}^{\ell} \left( [f,z_i] + z_i (f - w_i(f)) \right) w_i. 
$$
Since $[f,z_i] + z_i (f - w_i(f)) \in \H_{\bc}(K)$ for all $i$, we must have $[f,z_i] + z_i (f - w_i(f)) = 0$. In particular, this implies that $z_1 \in \ZH_{\bc}(K) \cap \ZH_{\bc}(W)$. Without loss of generality, $z_1 = 0$. But now it follows from Lemma \ref{lem:stupidremark} that $z = 0$. Thus, $\ZH_{\bc}(W) = \ZH_{\bc}(K)^W$. 

It is clear that the embedding is as Poisson algebras; one can see this directly from the construction or simply by noting that the bracket is $\Gamma$-invariant and hence restricts to $\ZH_{\bc}(K)^{\Gamma}$.
\end{proof}

Thus, geometrically we have a Poisson morphism $\eta : \cX_{\bc}(K) \rightarrow \cX_{\bc}(W)$ identifying $\cX_{\bc}(W)$ with $\cX_{\bc}(K) / \Gamma$. It is a finite, surjective map which is generically a $\Gamma$-covering. This fits into a commutative diagram
\[
\begin{tikzcd}
\cX_\bc(K) \arrow[twoheadrightarrow]{r}{\eta} \arrow[twoheadrightarrow]{d}{\Upsilon_{\bc,K}} & \cX_\bc(W) \arrow[twoheadrightarrow]{d}{\Upsilon_{\bc,W}} \\
\fh/K \times \fh^*/K \arrow[twoheadrightarrow]{r} & \fh/W \times \fh^*/W
\end{tikzcd}
\]

\begin{lem}\label{lem:etaleaf}
If $\mc{L}$ is a leaf of $\cX_{\bc}(K)$, then $\eta(\mc{L})$ is a finite union of leaves of $\cX_{\bc}(W)$. 
\end{lem}

\begin{proof}
Since the stratification of $\cX_{\bc}(W)$ by symplectic leaves is finite, it suffices to show that  $\eta(\mc{L})$ is a union of leaves, i.e. invariant under Hamiltonian flows. After a suitable localization, we may assume that $\mc{L}$ is closed in $\cX_{\bc}(K)$. Then $\eta(\mc{L})$ is closed. It is invariant under Hamiltonian flows if and only if the semi-prime ideal $I(\eta(\mc{L}))$ is Poisson. But $ I(\eta(\mc{L})) = I(\mc{L}) \cap \ZH_{\bc}(K)^{\Gamma}$. Since $I(\mc{L})$ is Poisson and the bracket is invariant under $\Gamma$, if $z \in I(\eta(\mc{L})) \cap \ZH_{\bc}(K)^{\Gamma}$ and $h \in \ZH_{\bc}(K)^{\Gamma}$, then $\{ z, h \} 	\in  I(\eta(\mc{L})) \cap \ZH_{\bc}(K)^{\Gamma}$, as required. 
\end{proof}

Note that, in general, the preimage of a leaf of  $\cX_{\bc}(W)$ is not a leaf.  Let $\cX_{\bc}(K)^{\mathrm{sing}}$ be the singular locus of $\cX_{\bc}(K)$, let $\cX_{\bc}(K)^{\mathrm{sm}}$ be the smooth locus and let $\cX_{\bc}(K)^{\mathrm{free}}$ be the locus where $\Gamma$ acts freely. The following is the geometric counterpart of \cite[Lemma 4.12]{CMPartitions}.

\begin{prop}
The preimage $\eta^{-1}(\cX_{\bc}(W)^{\mathrm{sm}})$ equals $\cX_{\bc}(K)^{\mathrm{sm}} \cap \cX_{\bc}(K)^{\mathrm{free}}$. 
\end{prop}

\begin{proof}
Since $\Gamma$ preserves the Poisson structure on $\cX_{\bc}(K)$, for each $p \in \cX_{\bc}(K)^{\mathrm{sm}}$, the group $\Gamma_p$ acts symplectically on the tangent space $T_p \cX_{\bc}(K)^{\mathrm{sm}}$. Thus, $(T_p \cX_{\bc}(K)^{\mathrm{sm}}) / \Gamma_p$ is smooth if and only if $\Gamma_p = 1$. Using the fact that one can linearize the action of a finite group in the formal neighborhood of any fixed point, this implies that the smooth locus of $\cX_{\bc}(K)^{\mathrm{sm}} / \Gamma$ equals $(\cX_{\bc}(K)^{\mathrm{sm}} \cap \cX_{\bc}(K)^{\mathrm{free}}) / \Gamma$. Hence 
$$
\eta^{-1}(\cX_{\bc}(W)^{\mathrm{sm}}) \cap \cX_{\bc}(K)^{\mathrm{sm}}  = \cX_{\bc}(K)^{\mathrm{sm}} \cap \cX_{\bc}(K)^{\mathrm{free}}.
$$
On the other hand, $\cX_{\bc}(K)^{\mathrm{sing}}$ is a union of symplectic leaves $\mc{L}$ with $\dim \mc{L} < \dim \cX_{\bc}(K)$. Therefore Lemma \ref{lem:etaleaf} implies that $\eta(\cX_{\bc}(K)^{\mathrm{sing}}) \subset \cX_{\bc}(W)^{\mathrm{sing}}$. 
\end{proof}

The following was stated in \cite{CMPartitions} in the case $\Gamma$ is a cyclic group. We give a simple geometric proof.

\begin{thm}\label{thm:WtoK}
Let $\bc : \Ref(K) \rightarrow \C$ be $W$-equivariant. 
\begin{enum_thm}
\item The group $\Gamma$ acts on $\Omega_{\bc}(K)$ such that ${}^{\sigma} \mc{F} = \{ {}^{\sigma} \lambda \ | \ \lambda \in \mc{F} \}$ for $\sigma \in \Gamma$ and $\mc{F} \in \Omega_{\bc}(K)$.
\item There is a natural bijection between $\Omega_{\bc}(W)$ and $\Omega_{\bc}(K) / \Gamma$ given by 
$$
\Omega_{\bc}(W) \ni \mc{F} \longleftrightarrow \{ \lambda \in \Irr (K) \ | \  \lambda \subset \Res_{K}^W \mu \textrm{ for some $\mu \in \mc{F}$ } \} \in \Omega_{\bc}(K) / \Gamma. 
$$
\end{enum_thm}
\end{thm}

\begin{proof}
Recall the notation from \S\ref{sec:RRCA}. We will use the notation and results from \cite[\S3, \S4]{CMPartitions}. Let $\overline{\ZH}_{\bc}(W)$ denote the quotient of $\ZH_{\bc}(W)$ by the ideal generated by $D(W)_+$, $\overline{\ZH}_{\bc}(K)$ the quotient of  $\ZH_{\bc}(K)$ by the ideal generated by $D(K)_+$ and $\widetilde{\ZH}_{\bc}(K)$ the quotient of $\ZH_{\bc}(K)$ by the ideal generated by $D(W)_+$. We also let $\wt{\H}_\bc(K)$ denote the quotient of $\H_\bc(K)$ by the ideal generated by $D(W)_+$. The Satake isomorphism \cite[Theorem 3.1]{EG} implies that the natural map $\overline{\ZH}_{\bc}(W) \rightarrow \widetilde{\ZH}_{\bc}(K)$ is an embedding. The group $\Gamma$ acts on $ \widetilde{\ZH}_{\bc}(K)$ and Proposition  \ref{prop:ZWinZK} now implies that $\overline{\ZH}_{\bc}(W) = \widetilde{\ZH}_{\bc}(K)^{\Gamma}$. Thus, 
$$
\overline{\ZH}_{\bc}(W) = \widetilde{\ZH}_{\bc}(K)^{\Gamma} \hookrightarrow \widetilde{\ZH}_{\bc}(K) \twoheadrightarrow \overline{\ZH}_{\bc}(K). 
$$
The kernel of the surjection $\widetilde{\ZH}_{\bc}(K) \twoheadrightarrow \overline{\ZH}_{\bc}(K)$ is nilpotent. Therefore it identifies the primitive idempotents in both algebras. 

Let 
$$
\{ d_i \}_{i \in \Omega_{\bc}(W)}, \quad \{ b_j' \}_{j \in \Omega_{\bc}(K)}, \quad \{ b_j \}_{j \in \Omega_{\bc}(K)},
$$
denote the primitive idempotents in $\overline{\ZH}_{\bc}(W)$, resp. $\widetilde{\ZH}_{\bc}(K)$ and $\overline{\ZH}_{\bc}(K)$. Then $\Gamma$ acts on $\{ b_j' \}_{j \in \Omega_{\bc}(K)}$ and the rule
$$
b_j' \mapsto \sum_{\sigma \in \Gamma / \mathrm{Stab}_{\Gamma}(b_j') } {}^{\sigma} b_j'
$$
defines a bijection 
$$
\{ d_i \}_{i \in \Omega_{\bc}(W)} \stackrel{1:1}{\longleftrightarrow} \{ b_j' \} / \Gamma.
$$
There is a natural surjective map $\wt{\H}_\bc(K) \twoheadrightarrow \ol{\H}_\bc(K)$ and the kernel of this map is generated by certain central nilpotent elements in $\wt{\H}_\bc(K)$. In particular, the kernel is contained in the radical of $\wt{\H}_\bc(K)$ and so the map induces a bijection between the simple modules. We can thus consider any simple $\ol{\H}_\bc(K)$-module $L_\bc(\lambda)$ as a simple $\ol{\H}_\bc(K)$-module, and to be precise we denote this as $\wt{L}_\bc(\lambda)$. 

Now, $b_i' \cdot \tilde{L}_\bc(\lambda) \neq 0$ if and only if $({}^{\sigma} b_i' ) \cdot ({}^{\sigma} \tilde{L}_\bc(\lambda)) \neq 0$. The statements of the theorem then follow from the Clifford theoretic fact, compare \cite[Proposition 4.7]{CMPartitions}, that
$$
\Res^{A_W}_{A_K}  L_\bc(\lambda) = \bigoplus_{\sigma \in \Gamma / \mathrm{Stab}_{\Gamma}(\mu) } {}^{\sigma}  \tilde{L}_\bc(\mu),
$$
for some (any) simple summand $\mu$ of $\Res_K^W \lambda$, where $A_W \dopgleich \ol{\H}_\bc(W)/\Rad \ol{\H}_\bc(W)$ and $A_K = \wt{\H}_\bc(K)/\Rad \wt{\H}_\bc(K)$ are the maximal semisimple quotients of $\ol{\H}_\bc(W)$ and $\wt{\H}_\bc(K)$, respectively.
\end{proof}

\begin{remark}
Geometrically, Theorem \ref{thm:WtoK} is simply saying that $\Upsilon^{-1}_{\bc,K}(0) = \eta^{-1}(\Upsilon^{-1}_{\bc,W}(0))$ is a union of $\Gamma$-orbits. 
\end{remark}

Let $\Omega_{\bc}(W)^{\mathrm{rigid}}$ denote the set of Calogero–Moser $\bc$-families containing a rigid module. 

\begin{prop}\label{prop:rigidKcusp}
Let $\bc : \Ref(K) \rightarrow \C$ be $W$-equivariant. 
\begin{enum_thm}
\item The set $\Omega_{\bc}(K)^{\mathrm{cusp}}$ is $\Gamma$-stable and the bijection of Theorem \ref{thm:WtoK}(a) restricts to an embedding $\Omega_{\bc}(K)^{\mathrm{cusp}} / \Gamma \hookrightarrow \Omega_{\bc}(W)^{\mathrm{cusp}}$. 
\item The set $\Omega_{\bc}(K)^{\mathrm{rigid}}$ is $\Gamma$-stable and the bijection of Theorem \ref{thm:WtoK}(b) restricts to a bijection $\Omega_{\bc}(W)^{\mathrm{rigid}} \stackrel{1:1}{\longleftrightarrow} \Omega_{\bc}(K)^{\mathrm{rigid}} / \Gamma$. 
\end{enum_thm}
\end{prop}

\begin{proof}
Part (a) follows from Lemma \ref{lem:etaleaf} which implies that the image of a zero-dimensional leaf is a zero-dimensional leaf. If $L_\bc(\lambda)$ is a rigid $\H_{\bc}(W)$-module and $\lambda'$ an irreducible summand of $\Res^W_K \lambda$, then $L_{\bc}(\lambda')$ is a rigid $\H_{\bc}(K)$-module. Conversely, if $L_{\bc}(\mu)$ is a rigid $\H_{\bc}(K)$-module and $\mu'$ an irreducible summand of $\Ind^W_K \mu$, then $L_\bc(\mu')$ is a rigid $\H_{\bc}(W)$-module. This implies part (b). 
\end{proof}

\begin{remark}
The embedding of Proposition \ref{prop:rigidKcusp} (1) is not generally a bijection since the preimage of a zero-dimensional leaf under $\eta$ is not always a union of zero-dimensional leaves. 
\end{remark}

\section{Type $A$} \label{type_A}

Let $W$ be the Weyl group of type $A_n$. This is simply the symmetric group $\fS_{n+1}$. It has an $n$-dimensional irreducible reflection representation. There is just one conjugacy class of reflections so that our parameter $\bc$ for rational Cherednik algebras is just a complex number. By Lemma \ref{cuspidals_for_c_eq_0} we know that Conjecture \ref{cm_cusp_lusztigz_cusp_conjecture} holds for $\bc=0$, so we can assume that $\bc > 0$. 

Etingof and Ginzburg \cite[Proposition 16.4]{EG} have shown that the Calogero–Moser space $\cX_\bc(W)$ is smooth. Theorem \ref{singleton_smooth} now implies that the Calogero–Moser $\bc$-families are singletons and Lemma \ref{singleton_cm_noncuspidal} shows that none of the Calogero–Moser $\bc$-families is cuspidal.

Lusztig \cite[Lemma 22.5]{LusztigUnequalparameters} on the other hand has shown that for integral $\bc>0$ we have $\Con_\bc(W) = \Irr(W)$. Using Lemma \ref{lusztig_rescaling} we conclude that $\Con_\bc(W) = \Irr(W)$ for arbitrary real $\bc > 0$. It then follows that the Lusztig $\bc$-families are singletons and using Lemma \ref{singleton_lusztig_not_cuspidal} we furthermore see that no Lusztig $\bc$-family is cuspidal. 

Comparing both results proves Theorem \ref{thm:cm_equal_lusztig} and Theorem \ref{thm:mainresultintro} for $W$ of type $A$. 

\section{Type $B$} \label{type_B}

Weyl groups of type $B$ are much more difficult to handle than those of type $A$, in particular as we now have to deal with a two-dimensional parameter space. We have split the discussion into several parts, some just dealing with the Calogero–Moser families, some just dealing with the Lusztig families. At the very end we combine these results to obtain the proof of Theorem~\ref{thm:mainresultintro}. \\

\noindent \begin{longtable}{L{0.9cm}R{12.3cm}}
 \S\ref{typeB_the_group}. & The group\dotfill\pageref{typeB_the_group} \\

 \S\ref{typeB_parabolics}. & Reflections and parabolic subgroups\dotfill\pageref{typeB_parabolics} \\
 
 \S\ref{typeB_reps}. & Representations\dotfill\pageref{typeB_reps} \\
 
 \S\ref{typeB_RCA}. & The rational Cherednik algebra\dotfill\pageref{typeB_RCA} \\
 
 \S\ref{typeB_isomorphisms}. & Isomorphisms\dotfill\pageref{typeB_isomorphisms} \\
 
 \S\ref{typeB_symp_leaves}. & Symplectic leaves\dotfill\pageref{typeB_symp_leaves} \\
 
 \S\ref{typeB_leaf_parabolic}. & Parabolic subgroups attached to symplectic leaves\dotfill\pageref{typeB_leaf_parabolic} \\
 
 \S\ref{typeB_cm_fams}. & Calogero–Moser families\dotfill\pageref{typeB_cm_fams} \\
 
 \S\ref{typeB_nondeg_simples}. & Simple $\ol{\H}_\bc(W)$-modules in the degenerate case\dotfill\pageref{typeB_nondeg_simples} \\
 
 \S\ref{typeB_lusztig_families_section}. & Lusztig families in the non-degenerate case\dotfill\pageref{typeB_lusztig_families_section} \\
 
 \S\ref{typeB_lusztig_fam_deg}. & Lusztig families in the degenerate case\dotfill\pageref{typeB_lusztig_fam_deg} \\
  
 \S\ref{typeB_cm_vs_lus}. & Calogero–Moser families vs. Lusztig families\dotfill\pageref{typeB_lusztig_fam_deg} \\  
 
 \S\ref{typeB_cuspidal_lusztig_section}. & Cuspidal Lusztig families in the non-degenerate case\dotfill\pageref{typeB_cuspidal_lusztig}  \\
 
 \S\ref{typeB_cuspidal_lusztig_deg_sect}. & Cuspidal Lusztig families in the degenerate case\dotfill\pageref{typeB_cuspidal_lusztig_deg_sect} \\
 
 \S\ref{typeB_rigid_modules}. & Rigid modules\dotfill\pageref{typeB_rigid_modules} \\

\S\ref{typeB_final_sect}. & Cuspidal Lusztig families vs. cuspidal Calogero–Moser families\dotfill\pageref{typeB_final_sect} \\
  
\end{longtable}

\subsection{The group} \label{typeB_the_group}

Let $W$ be the Weyl group of type $B_n$. This group is isomorphic to the group $G(2,1,n)$ of generalized permutation matrices in $\GL_n(\bbC)$ with entries in $\mu_2 \dopgleich \lbrace 1, -1 \rbrace \subs \bbC$, and this defines at the same time an irreducible reflection representation of $B_n$. Note that $W = \mu_2^n \rtimes \fS_n$, where $\fS_n$ acts on $\mu_2^n$ by coordinate permutation. For each $1 \leq i \leq n$ we have a natural embedding $\eps_i$ of $\mu_2$ into $W$, sending $u \in \mu_2$ to the diagonal matrix $(1,\ldots,u,\ldots,1)$ with $u$ in the $i$-th place. For $1 \leq i < j \leq n$ let $s_{ij}$ be the transposition $(i,j) \in \fS_n$. For $u \in \mu_2$ set $s_{ij,u} \dopgleich s_{ij}\eps_i(u)^{-1}\eps_j(u)$. Note that $s_{ij,1} = s_{ij}$. The group $W$ is generated by $\eps_1(-1)$ and the transpositions $s_{ij}$.

\subsection{Reflections and parabolic subgroups} \label{typeB_parabolics}

Let $(y_1,\ldots,y_n)$ be the standard basis of $\fh \dopgleich \bbC^n$ with dual basis $(x_1,\ldots,x_n)$. For any $1 \leq j \leq n$ the element $\eps_j(-1)$ is a reflection with coroot $\alpha_j^\vee \dopgleich y_j$ and root $\alpha_j \dopgleich 2 x_j$. Also, for any $u \in \mu_2$ and $1 \leq i<j \leq n$ the element $s_{ij,u}$ is a reflection with coroot $\alpha_{ij,u}^\vee \dopgleich u y_i - y_j $ and root $\alpha_{ij,u} \dopgleich u^{-1}x_i - x_j = ux_i-x_j$. These elements are precisely the reflections in $W$. We can now easily compute that
\begin{equation} \label{typeB_cherednik_coeff_1}
(y_k,\alpha_j) (\alpha_j^\vee,x_l) = \left\lbrace \begin{array}{ll} 2 & \tn{if } k=j=l \\ 0 & \tn{else} \end{array} \right.
\end{equation}
and
\begin{equation} \label{typeB_cherednik_coeff_2}
(y_k,\alpha_{ij,u}) (\alpha_{ij,u}^\vee,x_l) = \left\lbrace \begin{array}{ll} 1 & \tn{if } k,l \in \lbrace i,j \rbrace \tn{ with } k = l \\ -u & \tn{if } k,l \in \lbrace i,j \rbrace \tn{ with } k \neq l \\ 0 & \tn{else.} \end{array} \right.
\end{equation}
The conjugacy classes of reflections in $W$ are 
\[
\mathcal{S}_0 \dopgleich \lbrace s_{ij,u} \mid u \in \mu_2, 1 \leq i < j \leq n \rbrace \quad \tn{and} \quad \mathcal{S}_1 \dopgleich \lbrace \eps_j(-1) \mid 1 \leq j \leq n \rbrace \;.
\]
We have $|\mathcal{S}_0| = n^2-n$ and $\mathcal{S}_1 = n$. The parabolic subgroups of $W$ are, up to conjugacy, of the form $\fS_\lambda \times B_{n-|\lambda|}$ for partitions $\lambda$ of integers $\leq n$.

\subsection{Representations} \label{typeB_reps}

Since $W = \mu_2^n \rtimes \fS_n = \mu_2 \wr \fS_n$, the irreducible representations of $W$ are labeled by bipartitions $\blambda = (\lambda^{(0)},\lambda^{(1)})$ of $n$. Let $\pi_{\blambda}$ denote the representation labeled by $\blambda$. The trivial representation of $W$ is $\pi_{(n,\emptyset)}$. The representation $\gamma \dopgleich \pi_{(\emptyset, n)}$ is a linear character of $W$ with $\gamma(s) = 1$ for all $s \in \mc{S}_0$ and $\gamma(s) = -1$ for $s \in \mc{S}_1$. We denote by $\gamma \pi_{\blambda}$ the $\gamma$-twist of $\pi_{\blambda}$.

The symmetric group $\s_n$ is a quotient of $B_n$ by sending $\eps_j(-1)$ to $1$. We can thus consider (irreducible) $\s_n$-modules $\pi_\lambda$ for partitions $\lambda$ of $n$ as (irreducible) $W$-modules. If $\blambda = (\lambda^{(0)},\lambda^{(1)})$ is a bipartition of $n$ and $r \dopgleich |\lambda^{(0)}|$, then $\pi_{\lambda^{(0)}} \boxtimes \gamma \pi_{\lambda^{(1)}}$ is an irreducible $(B_r \times B_{n-r})$-subrepresentation of $\pi_{\blambda}$ with 
\begin{equation}\label{eq:Bnrep}
\pi_{\blambda} = \Ind_{B_r \times B_{n-r}}^{B_n} \pi_{\lambda^{(0)}} \boxtimes \gamma \pi_{\lambda^{(1)}} \;.
\end{equation}

\subsection{The rational Cherednik algebra} \label{typeB_RCA}

Fix a $W$-equivariant function $\bc:\Ref(W) \rarr \bbC$ and define 
\[
c_1 \dopgleich \bc(\mathcal{S}_1) \quad \tn{and} \quad \kappa \dopgleich \bc(\mathcal{S}_0) \;.
\]
In terms of the Coxeter diagram of type $B_n$ the weight function $\bc$ is determined as follows:
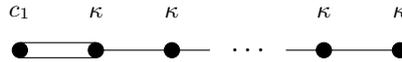
\[
\begin{tikzpicture}
        \draw (0,0.1) -- (1,0.1);
	\draw (0,-0.1) -- (1,-0.1);
	\draw (1, 0) -- (2,0);
	\draw (2, 0) -- (2.5,0);
	\draw (3.5, 0) -- (4,0);
	\draw (4,0) -- (5,0);

	\draw[fill=black] (0,0) circle (.1);
	\draw[fill=black] (1,0) circle (.1);
	\draw[fill=black] (2,0) circle (.1);
	\draw[fill=black] (4,0) circle (.1);
	\draw[fill=black] (5,0) circle (.1);
	
	\node at (3,0) {$\ldots$};
	
	\node at (0,0.5) {\footnotesize $c_1$};
	\node at (1,0.5) {\footnotesize $\kappa$};
	\node at (2,0.5) {\footnotesize $\kappa$};
	\node at (4,0.5) {\footnotesize $\kappa$};
	\node at (5,0.5) {\footnotesize $\kappa$};	
\end{tikzpicture} 
\]
Using equations (\ref{typeB_cherednik_coeff_1}) and (\ref{typeB_cherednik_coeff_2}) we see that the defining relation (\ref{eq:rel}) for $\H_\bc(W)$ becomes 
\begin{equation}\label{eq:Brel1}
\lbrack y_i,x_i \rbrack = - 2 c_1 \eps_i(-1) - \kappa \sum_{u \in \mu_2} \sum_{\substack{j=1 \\ j \neq i}}^n  s_{ij,u} 
\end{equation}
and
\begin{equation}\label{eq:Brel2}
\lbrack y_i,x_j \rbrack = \kappa \sum_{u \in \mu_2} u s_{ij,u} \;.
\end{equation}
for $i \neq j$. These are the same relations and parameters as in \cite{Martino-blocks-gmpn}. 

Recall from Lemma \ref{cuspidals_for_c_eq_0} that Conjecture \ref{cm_cusp_lusztigz_cusp_conjecture} holds for $\bc=0$.

\begin{leftbar}
\vspace{6pt}
We assume from now on that $\bc \neq 0$, i.e. $c_1 \neq 0$ or $\kappa \neq 0$.
\vspace{4pt}
\end{leftbar}

\subsection{Isomorphisms} \label{typeB_isomorphisms}

Recall that for any $\alpha \in \bbC^*$, the algebras $\H_{\bc}(W)$ and $\H_{\alpha \bc}(W)$ are  isomorphic. Given a bipartition $\blambda = (\lambda^{(0)},\lambda^{(1)})$, we define $\blambda^{\tau}$ to be $(\lambda^{(1)}, \lambda^{(0)})$. The following proposition follows from \cite[4.6B]{BonnafeRouquier}. 

\begin{prop}\label{prop:someauto}
The linear character $\gamma$ of $W$ defined in \S\ref{typeB_reps} extends to an isomorphism 
\[
\tau : \H_{(c_1,\kappa)}(B_n) \stackrel{\sim}{\longrightarrow} \H_{(-c_1,\kappa)}(B_n)
\]
with $\tau(x) = x$, $\tau(y) =y$ and $\tau(w) = \gamma(w) w$ for all $x \in \h^*, y \in \h$ and $w \in W$. Moreover, 
\begin{enum_thm}
\item ${}^{\tau^{-1}} L_{(c_1,\kappa)}(\blambda) \simeq L_{(-c_1,\kappa)}(\blambda^{\tau})$.
\item $\blambda$ and $\boldsymbol{\mu}$ belong to the same Calogero--Moser $(c_1,\kappa)$-family  if and only if $\blambda^{\tau}$ and $\boldsymbol{\mu}^{\tau}$ belong to the same Calogero--Moser $(-c_1,\kappa)$-family.
\item $\blambda$ is cuspidal, resp. rigid, for $\overline{\H}_{(c_1,\kappa)}(W)$ if and only if $\blambda^{\tau}$ is cuspidal, resp. rigid, for $\overline{\H}_{(-c_1,\kappa)}(W)$.
\end{enum_thm}
\end{prop}

In the case $\kappa = 0$ the defining relations (\ref{eq:Brel1}) and (\ref{eq:Brel2}) of $\H_\bc(W)$ show that we have an algebra isomorphism $\H_\bc(W) \simeq \H_{c_1}(\bbZ_2)^{\otimes n} \rtimes \fS_n$, where $\fS_n$ naturally acts on the $n$-fold tensor product of the rational Cherednik algebra at $c_1$ for the cyclic group of order $2$. From this we get an isomorphism of Poisson varieties $\cX_{\bc}(W) \simeq S^n(\cX_{c_1}(\Z_2))$, where $S^n$ denotes the $n$-th symmetric power. Since $c_1 \neq 0$, the Calogero–Moser space $\cX_{c_1}(\Z_2)$ is a smooth symplectic surface by \cite[16.2]{EG}.

\subsection{Symplectic leaves} \label{typeB_symp_leaves}

It was shown by Etingof and Ginzburg \cite[16.2]{EG} that the Calogero–Moser space of type $B$ is smooth for generic parameters. In this case the Calogero–Moser families are singletons by Theorem \ref{singleton_smooth} and none of them is cuspidal by Lemma \ref{singleton_cm_noncuspidal}. Using the relation between Calogero–Moser spaces and representation varieties of deformed preprojective algebras, Martino has determined in his Ph.D thesis \cite[Section 5]{MoThesis} for precisely which parameters the Calogero–Moser space is smooth and gave a parametrization of the symplectic leaves.\footnote{In \cite{MoThesis} the parameters are named $(c_\gamma,c_1)$ instead of $(c_1,\kappa)$.} To simplify notations we set $\lbrack a,b \rbrack \dopgleich \lbrace a,\ldots,b\rbrace$ and denote by $\pm \lbrack a,b \rbrack$ the set $\lbrack -b,-a \rbrack \cup \lbrack a,b \rbrack$ for integers $a \leq b$. Note that $\pm \lbrack 0,b \rbrack = \lbrack -b,b \rbrack$.

\begin{thm}[Martino] \label{thm:BSympleaves} 
Let $\bc = (\kappa,c_1)$.
\begin{enum_thm}
\item \label{thm:BSympleaves:singular}  $\cX_{\bc}(W)$ is singular if and only if $\kappa = 0$ or $c_1 = m \kappa$ for some $m \in \pm \lbrack 0,n-1 \rbrack$.  
\item If $\kappa = 0$, then the symplectic leaves of $\cX_{\bc}(W)$ are parameterised by the set $\mc{P}(n)$ of partitions of $n$. For $\lambda \in \mc{P}(n)$, the corresponding leaf $\mc{L}_{\lambda}$ has dimension $2 \ell (\lambda)$, where $\ell(\lambda)$ is the length of $\lambda$.
\item If $c_1 = m \kappa$, with $\kappa \neq 0$, then there is a bijection $k \mapsto \mc{L}_k$,
$$
\{ \textrm{symplectic leaves $\mc{L}$ of $\cX_{\bc}(W)$} \ \} \stackrel{1 : 1}{\longleftrightarrow} \{ k \in \bbN_{\ge 0} \ | \ k(k+m) \le n \} \;.
$$
Moreover, $\dim \mc{L}_{k} = 2(n - k(k+m))$.
\end{enum_thm}
\end{thm}

We say that $\bc$ is \word{singular} if $\cX_\bc(W)$ is singular. Moreover, we call singular parameters with $\kappa \neq 0$ \word{non-degenerate} and those with $\kappa = 0$ \word{degenerate}. 
By the formulas for the dimensions of the symplectic leaves we can immediately deduce when zero-dimensional leaves (and thus cuspidal Calogero–Moser families) exist. 

\begin{cor} \label{typeB_existence_cuspidals}
The space $\cX_\bc(W)$ has a zero-dimensional symplectic leaf if and only if $c_1 = m \kappa$ for some $m \in \pm \lbrack 0,n-1 \rbrack$ such that $n=k(k+m)$ for some $k > 0$. In this case there is a unique zero-dimensional leaf and thus a unique cuspidal Calogero–Moser family.
\end{cor}

If our parameter $\bc$ is as in Corollary \ref{typeB_existence_cuspidals} we say that it is  \word{cuspidal}. %

\begin{remark} \label{typeB_cuspidal_param_unique}
For a  given $n$ and $\pm m \in \lbrack 0,n-1 \rbrack$ there is at most one $k \geq 0$ with $n = k(k+m)$.
\end{remark}

\subsection{Parabolic subgroups attached to symplectic leaves} \label{typeB_leaf_parabolic}
 
We would like to parameterise the symplectic leaves of $\cX_{\bc}(B_n)$ by conjugacy classes of parabolic subgroups and work out the geometric ordering. 

\begin{lem}
If $c_1=m\kappa$ for some $m \in \pm \lbrack 0,n-1 \rbrack$, then the leaf $\mc{L}_k$ is labeled by the conjugacy class of the parabolic $B_{k(k+m)}$ and
$$
\mc{L}_k \prec \mc{L}_{k'} \quad \Longleftrightarrow \quad (B_{k(k+m)}) \le (B_{k'(k'+m)})  \quad \Longleftrightarrow \quad k \ge k'.  
$$
\end{lem}

\begin{proof}
If $\mc{L}_k$ is labeled by the parabolic $W'$ then $\cX_{\bc'}(W')$ contains at least one zero-dimensional leaf and $W'$ must have rank $n - \frac{1}{2} \dim \mc{L}_k = k(k+m)$. Since $\kappa \neq 0$, the parabolic must be of the form $B_m$ for some $m$. Hence $W' = B_{k(k+m)}$. It is a consequence of the proof of \cite[Proposition 5.7]{MoThesis} that $\mc{L}_k \prec \mc{L}_{k'}$ if and only if $k \ge k'$.
\end{proof}

In the degenerate case $\kappa = 0$, recall from \S\ref{typeB_isomorphisms} that there is an isomorphism of Poisson varieties $\cX_{\bc}(B_n) \simeq S^n(\cX_{c_1}(\Z_2))$. Then $\mc{L}_{\lambda} = S^{\lambda}(\cX_{c_1}(\Z_2))$, where $S^{\lambda}(X)$ is the image in $S^n (X)$ of the set $\left\{ \sum_{i = 1}^{\ell(\lambda)} \lambda_i \cdot x_i \ | \ x_i \neq x_j \in X \right\}$. This implies that $\mc{L}_{\lambda}$ is labeled by the class of the parabolic subgroup $\s_{\lambda} = \s_{\lambda_1} \times \cdots \times \s_{\lambda_{\ell(\lambda)}}$ and 
$$
\mc{L}_{\lambda} \prec \mc{L}_{\mu} \quad \Leftrightarrow \quad (\s_{\lambda}) \le (\s_{\mu}). 
$$
Moreover, in this case, if $\Upsilon_\bc^{-1}(0) = \{ p , q \}$ for $\H_{c_1}(\Z_2)$, where $p = \Supp L_{c_1}(1_{\bbZ_2})$ and $q = \Supp L_{c_1}(\sgn_{\bbZ_2})$, then in $\cX_{\bc}(B_n)$ we have
\begin{equation}\label{eq:UpsilonZ2}
\Upsilon_\bc^{-1}(0) = \{ n_1 \cdot p + n_2 \cdot q \ | \ n_1 + n_2 = n, n_i \ge 0 \}. 
\end{equation}
The point $n_1 \cdot p + n_2 \cdot q$ belongs to the leaf $\mc{L}_{(n_1,n_2)}$.

\subsection{Calogero–Moser families} \label{typeB_cm_fams}

The Calogero–Moser families in type $B_n$ have been first described by Gordon and Martino \cite{GordonMartinoCM} using the notion of $J$-hearts, and later by Martino \cite{Martino-blocks-gmpn} using the notion of \textit{residues}. We recall the description given in \cite{Martino-blocks-gmpn} now.

Let $\lambda = (\lambda_1,\lambda_2,\ldots)$ be a partition. We think of $\lambda$ as a stack of boxes, left justified, with the bottom row containing  $\lambda_1$ boxes, the next row containing $\lambda_2$ boxes and so forth. The \textit{content} $\mathrm{ct}(\Box)$ of a box $\Box = (i,j) \in \lambda$ is defined to be $j - i$. We consider the group ring $\bbZ \lbrack \bbC \rbrack$ of the additive group $\C$ and write $x^\alpha$ for the element corresponding to $\alpha \in \bbC$. The \word{residue} of $\lambda$ is the element 
\[
\Res_{\lambda}(x) \dopgleich \sum_{\Box \in \lambda} x^{\mathrm{ct}(\Box)} \in \bbZ \lbrack \bbZ \rbrack \subs \bbZ \lbrack \bbC \rbrack \;.
\]
Just as in \cite[\S3A]{BroueKim}, we define for a triple $\bm = (m_0,m_1,m')$  of complex numbers (the \word{charge}), and a bipartition $\blambda = (\lambda^{(0)}, \lambda^{(1)})$, the \word{charged residue} as
\[
\Res_{\blambda}^\bm(x) \dopgleich x^{m_0} \Res_{\lambda^{(0)}}(x^{m'}) + x^{m_1} \Res_{\lambda^{(1)}}(x^{m'}) \in \bbZ \lbrack \bbC \rbrack \;.
\] 
The following theorem is \cite[Theorem 5.5]{Martino-blocks-gmpn}. The additional parameters $(h,H_0,H_1)$ used in \textit{loc.\ cit.} are given by $h = - \kappa$, $H_0 = -c_1$, and $H_1 = c_1$.

\begin{theorem}[Martino] \label{typeB_cm_families}
Two bipartitions $\blambda$ and $\bmu$ lie in the same Calogero–Moser $\bc$-family if and only if $\Res_{\blambda}^{\hat{\bc}}(x) = \Res_{\bmu}^{\hat{\bc}}(x)$ with respect to the charge $\hat{\bc} \dopgleich (0,c_1,-\kappa)$.
\end{theorem}

\subsection{Simple $\ol{\H}_\bc(W)$-modules in the degenerate case} \label{typeB_nondeg_simples}

In the degenerate case $\kappa = 0$ it is possible to determine the structure of the simple $\ol{\H}_\bc(W)$-modules $L_\bc(\blambda)$ as $W$-modules.

\begin{lemma} \label{cm_families_typeB_degenerate}
If $\kappa = 0$, then $\blambda = (\lambda^{(0)}, \lambda^{(1)}) \in \mathcal{P}_2(n)$ and $\bmu =(\mu^{(0)},\mu^{(1)}) \in \mathcal{P}_2(n)$ lie in the same Calogero–Moser $\bc$-family if and only if $| \lambda^{(1)} | = |\mu^{(1)}|$. In particular, there are $n+1$ Calogero–Moser families $\mathcal{F}_{0;n}^{\mrm{deg}},\ldots,\mathcal{F}_{n;n}^{\mrm{deg}}$ with
\[
\mathcal{F}_{i;n}^{\mrm{deg}} = \lbrace \blambda \in \mathcal{P}_2(n) \mid |\lambda^{(0)}| = i \rbrace \;.
\]
\end{lemma}

\begin{proof}
In the case $\kappa=0$ we have 
\[
\Res_{\blambda}^{\hat{\bc}}(x) = \sum_{\Box \in \lambda^{(0)}} 1^{\mrm{ct} \Box} + x^{c_1} \sum_{\Box \in \lambda^{(1)}} 1^{\mrm{ct} \Box} = |\lambda^{(0)}| + x^{c_1} |\lambda^{(1)}| = n-|\lambda^{(1)}| + x^{c_1} |\lambda^{(1)}| \;.
\]
The claim follows directly from Theorem \ref{typeB_cm_families}.
\end{proof}

\begin{prop}
Assume that $\kappa = 0$. Then the family $\mc{F}_{i;n}^{\mrm{deg}}$ is labeled by the class of the parabolic $\s_i \times \s_{n-i} \subset B_n$ and we have a bijection $\Irr (\s_i \times \s_{n-i}) \stackrel{\sim}{\longrightarrow} \mc{F}_{i;n}^{\mrm{deg}}$, sending the pair of partitions $(\lambda^{(0)},\lambda^{(1)})$ to itself (thought of as a bipartition) such that 
$$
L_\bc(\lambda^{(0)},\lambda^{(1)}) \simeq \Ind_{\s_i \times \s_{n-i}}^{B_n} \pi_{\lambda^{(0)}} \boxtimes \pi_{\lambda^{(1)}} 
$$
as $W$-modules.
\end{prop}

\begin{proof}
Since $\H_{\bc}(B_n) \simeq (\H_{c_1}(\Z_2)^{\otimes 2}) \rtimes \s_n$ in this case (see \S\ref{typeB_isomorphisms}), we have 
$$
L_\bc(\lambda^{(0)},\lambda^{(1)}) \simeq \Ind_{(\H_{c_1}(\Z_2)^{\otimes 2}) \rtimes (\s_i \times \s_{n-i})}^{B_n} L_{c_1}(1_{\bbZ_2}) \otimes L_{c_1}(\sgn_{\bbZ_2}) \otimes (\pi_{\lambda^{(0)}} \boxtimes \pi_{\lambda^{(1)}}) \;,
$$
Since $c_1 \neq 0$, both $L_{c_1}(1_{\bbZ_2})$ and $L_{c_1}(\sgn_{\bbZ_2})$ are isomorphic to the regular representation as $\Z_2$-modules. 
Recall that we have described $\Upsilon_\bc^{-1}(0)$ in (\ref{eq:UpsilonZ2}). If $i = |\lambda^{(0)}|$, so that $n - i = |\lambda^{(1)}|$, then the support of $L_\bc(\lambda^{(0)},\lambda^{(1)})$ is $i \cdot p + (n-i) \cdot q$, which lies on the the leaf labeled by the parabolic $\s_{i} \times \s_{n-i}$. The result follows. 
\end{proof}

\subsection{Lusztig families in the non-degenerate case} \label{typeB_lusztig_families_section}

For the description of the Lusztig families in the non-degenerate case we first argue that we can restrict to the so-called \word{integral} case where $c_1$ is an integral multiple of $\kappa$. 

\begin{prop}\label{lusztig_fam_typeB_non_multiple}
Suppose that $\kappa > 0$. If there is no $m \in \bbN$ with $c_1 = m\kappa$, then $\Con_\bc W = \Irr W$ and so the Lusztig $\bc$-families are singletons.
\end{prop}

\begin{proof}
By Lemma \ref{lusztig_rescaling}, we may assume that $\kappa = 1$. The statement of the proposition has been shown by Lusztig \cite[Proposition 22.25]{LusztigUnequalparameters} when $c_1$ is rational. We reduce the general case to the rational case. Let $\lambda$ be an irreducible representation of a parabolic subgroup $W'$ of $W$. The explicit formula given for the Schur element $\mathbf{s}_{\lambda}$, see \cite[Theorem 10.5.2]{GeckPfeiffer} and \cite[Lemma 22.12]{LusztigUnequalparameters}, shows that there exist finitely many integers $r^{\lambda}_1, r^{\lambda}_2, \dots $, $s^{\lambda}_1, s^{\lambda}_2, \dots$, with $(r_i,s_i) \neq (r_j,s_j)$ for $i \neq j$, and rational numbers $f^{\lambda}_1,f^{\lambda}_2, \dots$ such that 
$$
\mathbf{s}_{\lambda} = \sum_i f^{\lambda}_i q^{2 \left( r^{\lambda}_i \kappa + s^{\lambda}_i c_1 \right)}. 
$$
Recall that $\kappa = 1$ and $c_1 \notin \Q$. We claim that $\mathbf{a}_{\lambda} = \min \{ r^{\lambda}_i \kappa + s^{\lambda}_i c_1 \ | \ i = 1, \dots \}$. Note that this is not the case in general since there might be some cancellation between the $f^{\lambda}_i$ when $r^{\lambda}_i \kappa + s^{\lambda}_i c_1 = r^{\lambda}_j \kappa + s^{\lambda}_j c_1$ for some $i \neq j$. However, in our case the fact that $\kappa = 1$ and $c_1$ is irrational implies that $r^{\lambda}_i \kappa + s^{\lambda}_i c_1 = r^{\lambda}_j \kappa + s^{\lambda}_j c_1$ if and only if $r^{\lambda}_i = r^{\lambda}_j$ and $s^{\lambda}_i = s^{\lambda}_j$, i.e. $i = j$. The claim follows. 

The definition of $\mathbf{j}$-induction and constructible representations makes it clear that if we are given two parameters $\bc$ and $\bc'$ such that 
\begin{equation}\label{eq:aeqaprime}
\mathbf{a}_{\lambda} =\mathbf{a}_{\mu} \quad \Leftrightarrow \quad \mathbf{a}_{\lambda}' =\mathbf{a}_{\mu}'
\end{equation}
for all irreducible representations $\lambda$ and $\mu$ of all parabolic subgroups of $W$, then $\Con_{\bc} W = \Con_{\bc'} W$. Since there are only finitely many $r^{\lambda}_i$ and $s^{\lambda}_j$ as $\lambda$ ranges over all irreducible representations of all parabolic subgroups of $W$, one can easily choose a rational number $c_1' > 0$ with $|c_1 - c_1' |$ very small and $c_1'$ not an integer such that 
$$
r^{\lambda}_i + s^{\lambda}_i c_1 < r^{\mu}_j + s^{\mu}_j c_1 \quad \Leftrightarrow \quad  r^{\lambda}_i + s^{\lambda}_i c_1' < r^{\mu}_j + s^{\mu}_j c_1'
$$
for all $\lambda$, $\mu$ and $i,j$. In particular, for $\bc' = (c_1',1)$ equation (\ref{eq:aeqaprime}) holds. Moreover, since $c_1'$ is rational, every constructible representation in $\Con_{\bc'} W$ is irreducible by \cite[Proposition 22.25]{LusztigUnequalparameters}. Hence $\Con_\bc W = \Irr W$, too. 
\end{proof}

We can thus restrict to the case $c_1 = m \kappa$ for some $m \in \bbN$, which by Lemma \ref{lusztig_rescaling} is the same as $\bc = (m,1)$. The Lusztig families in this case have been described by Lusztig \cite[\S22]{LusztigUnequalparameters} using the notion of symbols. We review the notion of symbols for general integral parameters.

\begin{leftbar}
\vspace{6pt}
We assume that $\kappa > 0$ and that $\bc = (c_1,\kappa) \geq 0$ is \textit{integral}.
\vspace{4pt}
\end{leftbar}

We can uniquely write $c_1 = m \kappa + r$ for some $m,r \in \bbN_{\geq 0}$ with $r < \kappa$. Fix an arbitrary integer $N > 0$. A \word{symbol} for $B_n$ with respect to $N$ at parameter $\bc$ is a list of the form
\begin{equation} \label{symbol_absract}
S = \left( \begin{array}{ccccccc}
\beta_1 & \beta_2 & \cdots & \cdots & \cdots & \beta_{N+m-1} & \beta_{N+m} \\
 \gamma_1 & \gamma_2 & \cdots & \gamma_{N} & 
\end{array} \right) \;,
\end{equation}
where $0 \le \beta_1 < \cdots < \beta_{N+m}$ are congruent to $r$ modulo $\kappa$ and $0 \le \gamma_1 < \cdots < \gamma_{N}$ are divisible by $\kappa$, such that
\begin{equation} \label{symbol_equation}
\sum_{i} \beta_i + \sum_j \gamma_j = n\kappa + \kappa N^2 + N(c_1-\kappa) + \kappa {m \choose 2} + rm \;.
\end{equation} 
Let $\Sy_{\bc;n}^N$ denote the set of all such symbols. We have an embedding $\Sy_{\bc;n}^N \hookrightarrow \Sy_{\bc;n}^{N+1}$ sending a symbol $S$ as above to the symbol
\begin{equation}
S[1] = \left( \begin{array}{cccccc}
r & \beta_1 + \kappa & \beta_2 + \kappa & \cdots & \cdots & \beta_{N+m} + \kappa \\
0 &  \gamma_1  + \kappa & \gamma_2 + \kappa & \cdots & \gamma_{N} + \kappa & 
\end{array} \right) \;.
\end{equation}
For $i \in \bbN$ we denote by $S \lbrack i \rbrack$ the $i$-fold composition of the above map applied to $S$ and call this the \word{shift} of $S$ by $i$. 
Let $\Sy_{\bc;n}$ be the direct limit of the $\Sy_{\bc;n}^N$ with respect to the above maps. We say that $N$ is \word{large enough} for a bipartition $\blambda = (\lambda^{(0)},\lambda^{(1)})$ of $n$ if $\lambda^{(0)}_{N+m+1} = 0 = \lambda^{(1)}_{N+1}$. We then define the corresponding symbol $\Sy_{\bc;n}^N(\blambda) = {\beta \choose \gamma} \in \Sy_{\bc;n}^N$ via 
\begin{equation}
\begin{array}{ll}
\beta_i \dopgleich \kappa\left( \lambda^{(0)}_{N+m-i+1}+i-1 \right) + r & \tn{for } i \in \lbrack 1,N+m \rbrack \\
\gamma_j \dopgleich \kappa \left( \lambda^{(1)}_{N-j+1}+j-1 \right) & \tn{for } j \in \lbrack 1,N \rbrack \;.
\end{array}
\end{equation}
If $N$ is large enough for all bipartitions of $n$, e.g., $N \geq n$, the map $\blambda \mapsto \Sy_{\bc;n}^N(\blambda)$ defines a bijection between the set $\mathcal{P}_2(n)$ of bipartitions of $n$ and $\Sy_{\bc;n}^N$. For a symbol $S$ we then denote by $\pi_S$ the representation of $W$ labeled by the bipartition corresponding to $S$. 
The \word{content} $\ct(S)$ of a symbol $S \in \Sy_{\bc;n}^N$ is the multiset of its entries, i.e., the list of entries with repetitions but ignoring positions. We can, and will, equally well write the content as a polynomial $\sum_{i \geq 0} n_i x^i$, where $n_i$ denotes the multiplicity of the entry $i$ in $S$. It is clear from the definition of a symbol that it has at least $N+m$ distinct entries and the multiplicity of an entry in a symbol is at most~$2$.

\begin{example} \label{typeB_symbol_example}
Let $\blambda = \left( \ \ydiagram{2,1} \ , \ \ydiagram{1} \ \right)$ and $(c_1,\kappa) = (1,1)$. Then
\[
\Sy_{(1,1);4}^3({\blambda}) = \begin{pmatrix} 0 & 1 &  3 & 5 \\ 0 & 1 & 3 \end{pmatrix} \in \Sy_{(1,1);4}^3 \;.
\]
This symbol is in fact the shift of $\begin{pmatrix} 1 & 3 \\ 1 \end{pmatrix} \in \Sy_{(1,1);4}^1$ by $2$.
\end{example}

\begin{theorem}[Lusztig] \label{typeB_lusztig_fams}
Let $\bc = (c_1,\kappa) \geq 0$ with $\kappa > 0$ and $c_1 = m\kappa$ for some $m \in \bbN$. Then two bipartitions $\blambda$ and $\bmu$ lie in the same Lusztig $\bc$-family if and only if $\ct(\Sy_{(m,1);n}^N({\blambda})) = \ct(\Sy_{(m,1);n}^N({\bmu}))$ for $N$ sufficiently large. 
\end{theorem}

\begin{proof}
Because of Lemma \ref{lusztig_rescaling} we can assume that $\bc=(m,1)$. Then the description of the $\bc$-constructible characters in \cite[Proposition 22.24]{LusztigUnequalparameters} along with \cite[Lemma 22.22]{LusztigUnequalparameters} proves the claim.
\end{proof}

\subsection{Lusztig families in the degenerate case} \label{typeB_lusztig_fam_deg}

The description of the Lusztig families in the degenerate case is given in \cite[Example 7.13]{Geck-Iancu-ordering} and follows from the general theory in \cite[\S2.4.3]{Geck.M;Jacon.N11Representations-of-H}.

\begin{lemma} \label{lusztig_fam_typeB_degenerate}
If $\kappa = 0$, then $\blambda = (\lambda^{(0)}, \lambda^{(1)}) \in \mathcal{P}_2(n)$ and $\bmu =(\mu^{(0)},\mu^{(1)}) \in \mathcal{P}_2(n)$ lie in the same Lusztig $\bc$-family if and only if $| \lambda^{(1)} | = |\mu^{(1)}|$. In particular, there are precisely $n+1$ Lusztig families $\mathcal{F}_{0;n}^{\mrm{deg}},\ldots,\mathcal{F}_{n;n}^{\mrm{deg}}$ with
\[
\mathcal{F}_{i;n}^{\mrm{deg}} = \lbrace \blambda \in \mathcal{P}_2(n) \mid |\lambda^{(0)}| = i \rbrace \;.
\]
\end{lemma}

\subsection{Calogero–Moser families vs. Lusztig families} \label{typeB_cm_vs_lus}

We can now prove Theorem \ref{thm:cm_equal_lusztig} for type $B_n$.

\begin{corollary} \label{cor:LusztigCMtypeB}
For type $B_n$ and any parameter $\bc \geq 0$ the Lusztig $\bc$-families are equal to the Calogero–Moser $\bc$-families.
\end{corollary}

\begin{proof}
If $\kappa = 0$, the claim follows from Lemma \ref{cm_families_typeB_degenerate} and Lemma \ref{lusztig_fam_typeB_degenerate}. Now, assume that $\kappa \neq 0$. If $c_1 \neq m\kappa$ for all $m \in \bbN_{\geq 0}$, then we know from Theorem \ref{thm:BSympleaves}\ref{thm:BSympleaves:singular} and Proposition \ref{lusztig_fam_typeB_non_multiple} that both the Calogero–Moser $\bc$-families and the Lusztig $\bc$-families are singletons. So, suppose that $c_1 = m \kappa$ for some $m \in \bbN_{\geq 0}$. Because of Lemma \ref{lusztig_rescaling} and Lemma \ref{cm_scaling} we can assume that $\kappa = 1$. It follows from \cite[Proposition 3.4]{BroueKim} that $\ct(\Sy_{\bc;n}^N(\blambda)) = \ct(\Sy_{\bc;n}^N(\bmu))$ for $N$ large enough if and only if $\Res_{\blambda}^{\hat{\bc}}(x) = \Res_{\bmu}^{\hat{\bc}}(x)$. By Theorem \ref{typeB_cm_families} and Theorem \ref{typeB_lusztig_fams} this shows that $\Omega_\bc(W) = \lus_\bc(W)$.
\end{proof}

\subsection{Cuspidal Lusztig families in the non-degenerate case} \label{typeB_cuspidal_lusztig_section}

In the non-singular case, Proposition \ref{lusztig_fam_typeB_non_multiple} and Lemma \ref{singleton_lusztig_not_cuspidal} immediately imply the following result.

\begin{lemma}
If $\kappa > 0$ and there is no $m \in \bbN$ with $c_1 = m \kappa$, then there are no cuspidal Lusztig $\bc$-families.
\end{lemma}

Lemma \ref{lusztig_rescaling} implies that we can restrict to the following situation.

\begin{leftbar}
\vspace{6pt}
We assume that $\kappa = 1$ and that $c_1 = m \kappa = m$ for some $m \in \bbN_{\geq 0}$. \vspace{4pt}
\end{leftbar}

At equal parameters, i.e. $m=1$, the cuspidal families are described by Lusztig in \cite[Section 8.1]{Lusztig-characters-reductive-groups}. It seems difficult to find an explicit description of the cuspidal families for unequal parameters. Therefore we derive the classification here in Theorem \ref{typeB_cuspidal_lusztig} using the results of \cite{LusztigUnequalparameters}. 

We choose $N$ sufficiently large for all bipartitions of $n$ (see \S\ref{typeB_lusztig_families_section}). For a Lusztig $\bc$-family $\mathcal{F}$ we denote by $\Sy_{\bc;n}^N(\mathcal{F})$ the set of symbols $\Sy_{\bc;n}^N(\blambda)$ with $\blambda \in \mathcal{F}$. Using the combinatorics of symbols, we can explicitly determine the size of $\mathcal{F}$.

\begin{lemma} \label{typeB_lusztig_family_size}
Let $\mathcal{F} \in \lus_\bc(W)$. Let $\sum_{i \geq 0} n_i x^i$ be the content of one (any) $S \in \Sy_{\bc;n}^N(\mathcal{F})$. Set $k_\mathcal{F} \dopgleich N - |\lbrace i \mid n_i = 2 \rbrace|$. Then $k_\mathcal{F} \geq 0$, $k_\mathcal{F}(k_\mathcal{F}+m) \leq n$, and $|\mathcal{F}| = {2k_\mathcal{F}+m \choose k_\mathcal{F}}$.
\end{lemma}

\begin{proof}
Let $S \in \Sy_{\bc;n}^N(\mathcal{F})$ and set $k \dopgleich k_\mathcal{F}$. The multiplicity of an entry in $S$ is at most equal to $2$ and $S$ has at least $N+m$ distinct entries. 
Since $S$ has exactly $2N+m$ entries with multiplicity, this immediately shows that $k \geq 0$. Let $E$ be the set (not multiset) of entries of $S$. By definition of $k$ we have $|E| = N+k+m$. Any $N$-element subset of $E$ containing the set $\lbrace i \mid n_i = 2 \rbrace$ defines a unique symbol in $\Sy_{\bc;n}^N(\mathcal{F})$, and in this way all symbols of $\Sy_{\bc;n}^N(\mathcal{F})$ are obtained. The number of such sets is equal to ${ N+k+m - (N-k) \choose N-(N-k)} = {2k+m \choose k }$. 

It remains to show that $k(k+m) \leq n$. Since $S \in \Sy_{\bc;n}^N$, equation (\ref{symbol_equation}) says that
\begin{equation} \label{lusztig_family_size_equation}
\sum_i \beta_i + \sum_j \gamma_j - N^2 - N(m-1) - {m \choose 2} = n \;, 
\end{equation}
where the $\beta_i$ and $\gamma_j$ are the entries of $S$. Hence, it suffices to show that the left  hand side is at least as big as $k(k+m)$. Recall that $N-k$ is equal to the number of pairs $(i,j)$ such that $\beta_i = \gamma_j$. Since $\beta_i,\gamma_j \geq 0$, the expression on the left is  minimal if $\beta_i=\gamma_i = i-1$ for $i=1,\ldots,N-k$ and the remaining $2k+m$ entries are in $\lbrace N-k,\ldots,N+k+m -1 \rbrace$. Then the left hand side of equation (\ref{lusztig_family_size_equation})  becomes
\[
\sum_{i=1}^{N-k} 2(i-1) + \sum_{i=N-k}^{N+k+m-1} i - N^2 - N(m-1) - {m \choose 2} = k(k+m)   \;.
\]
\end{proof}

\begin{definition}
A symbol $S \in \Sy_{\bc;n}^N$ with content $\sum_{i \geq 0} n_i x^i$ is called \word{cuspidal} if $n_i \ge n_{i+1}$ for all $i = 0,1, \dots$
\end{definition}

If $S$ is a cuspidal symbol then $S[1]$ is also cuspidal. 

Suppose that $n = k(k+m)$ for some $k \in \bbN_{>0}$. Then we have the box partition $(k^{k+m})$ of $n$. If $\lambda$ is any partition such that $\ell(\lambda) \le k+m$ and $\lambda_1 \le k$, then $\lambda \subs (k^{k+m})$. Adding zeros, we may assume that $\ell(\lambda) = k+m$. Define the partition $\lambda^\dagger$ by
\begin{align*}
\lambda^\dagger_i & \dopgleich \# \lbrace j \in \lbrack 1,k+m \rbrack \mid k-\lambda_{k+m+1-j} \geq i \rbrace \\ & \  = k+m+1 - \min \lbrace j \in \lbrack 1,k+m \rbrack \mid k-i \geq \lambda_j \rbrace \;.
\end{align*}
This is simply the transpose of the reverse of the complement $k-\lambda$ of $\lambda$ in the box $(k^{k+m})$. Since $|\lambda| + |\lambda^\dagger| = k(k+m) = n$, we get in this way a bipartition $(\lambda,\lambda^\dagger)$ of $n$. Let
\[
\mathcal{F}_{k,m}^{\mrm{cusp}} \dopgleich \lbrace (\lambda,\lambda^\dagger) \mid \ell(\lambda) \leq k+m, \lambda_1 \leq k \rbrace \;.
\]

\begin{example}
If $n=6$ and $m=1$, we can write $n=k(k+m)$ with $k=2$ and get
\begin{align*}
\mathcal{F}_{2,1}^{\mrm{cusp}} & = \left\lbrace \left( \  \ydiagram{1} \ , \ \ydiagram{3,2} \ \right), \left( \emptyset \ , \ \ydiagram{3,3} \ \right), \left( \ \ydiagram{2,2,2} \ , \ \emptyset \right), \left( \ \ydiagram{2} \ , \  \ydiagram{2,2} \ \right), \left( \ \ydiagram{2,1,1} \ , \ \ydiagram{2} \ \right), \right. \\ & \quad\quad \left.  \left( \ \ydiagram{2,2} \ , \ \ydiagram{1,1} \ \right), \left( \  \ydiagram{1,1} \ , \ \ydiagram{3,1} \ \right), \left( \ \ydiagram{2,2,1} \ , \ \ydiagram{1} \ \right), \left( \ \ydiagram{1,1,1} \ ,\  \ydiagram{3} \ \right), \left( \ \ydiagram{2,1} \ , \ \ydiagram{2,1} \ \right) \right\rbrace \;.
\end{align*}
If $n=3$ and $m=2$, we can write $n = k(k+m)$ with $k=1$ and get
\[
\mathcal{F}_{1,2}^{\mrm{cusp}} = \left\lbrace \left(  \emptyset \ , \ \ydiagram{3} \ \right), \left( \ \ydiagram{1} \ , \ \ydiagram{2} \ \right), \left( \ \ydiagram{1,1} \ , \ \ydiagram{1} \ \right), \left( \ \ydiagram{1,1,1} \ , \  \emptyset  \right) \right\rbrace \;.
\]
\end{example}

\begin{lemma} \label{typeB_cuspidal_family_minimal_reps}
Suppose that $n=k(k+m)$. The content of the symbol $\Sy_{\bc;n}^k(\lambda,\lambda^\dagger)$ is equal to $\sum_{i=0}^{2k+m-1} x^i$ for any $(\lambda,\lambda^\dagger) \in \mathcal{F}_{k,m}^{\mrm{cusp}}$. In particular, $\Sy_{\bc;n}^k(\lambda,\lambda^\dagger)$ is cuspidal and $\mathcal{F}_{k,m}^{\mrm{cusp}}$ is a Lusztig $\bc$-family with $|\mathcal{F}_{k,m}^{\mrm{cusp}}| = {2k+m \choose k}$.
\end{lemma}

\begin{proof}
First, note that $N=k$ is large enough for any $(\lambda, \lambda^\dagger) \in \mathcal{F}_{k,m}^{\mrm{cusp}}$. The symbol $\Sy_{\bc;n}^k(\lambda,\lambda^\dagger) ={\beta \choose \gamma}$ is then given by
\[
\begin{array}{ll}
\beta_i \dopgleich  \lambda_{k+m-i+1}+i-1 & \tn{for } i \in \lbrack 1,k+m \rbrack \\
\gamma_j \dopgleich \lambda^\dagger_{k-j+1}+j-1 & \tn{for } j \in \lbrack 1,k \rbrack \;.
\end{array}
\] 
Our assertion about the content of this symbol is equivalent to showing that the symbol contains the entry $0$, all entries are bounded above by $2k+m-1$, and that $\beta_i \neq \gamma_j$ for all $i,j$. First, we have
\[
\beta_1 = \lambda_{k+m} \quad \tn{and} \quad \gamma_1 = \lambda_{k}^\dagger = k+m+1 - \min \lbrace j \in \lbrack k+m \mid 0 \geq \lambda_j \rbrace = k+m-\ell(\lambda) \;.
\]
We immediately see that either $\beta_1 = 0$ or $\gamma_1 = 0$. On the other hand, we have
\[
\beta_{k+m} = \lambda_1 + k+m-1 \quad \tn{and} \quad \gamma_k = \lambda_1^\dagger + k - 1 = 2k+m - \min \lbrace j \in \lbrack 1,k+m \rbrack \mid k > \lambda_j \rbrace \;.
\]
This shows that the entries of the symbol are at most equal to $2k+m-1$. Showing that $\beta_i \neq \gamma_j$ for all $i \in \lbrack 1,k+m \rbrack$ and $j \in \lbrack 1,m \rbrack$ is equivalent to showing that $\beta_{k+m-i+1} \neq \gamma_{k-j+1}$ for all $i \in \lbrack 1,k+m \rbrack$ and $j \in \lbrack 1,m \rbrack$. Now,
\begin{align*}
\beta_{k+m-i+1} = \gamma_{k-j+1} & \Leftrightarrow \lambda_i+m-i = \lambda_j^\dagger + k - j \\ & \Leftrightarrow \lambda_i - i = k+1 - \min \lbrace l \in \lbrack 1,k+m \rbrack \mid k-j \geq \lambda_l \rbrace - j \numberthis \label{typeB_cuspidal_symbol_ct} \;.
\end{align*}
Suppose that $k-j \geq \lambda_i$. Then $\min \lbrace l \mid k-j \geq \lambda_l \rbrace \leq i$ and we get from equation (\ref{typeB_cuspidal_symbol_ct}) the estimate $\lambda_i - i \geq k+1 - i - j$. This implies $\lambda_i > k-j$, contradicting the assumption. On the other hand, if $k-j \leq \lambda_i$, we similarly deduce the estimate $\lambda_i < k-j$, again a contradiction. Hence, $\beta_i \neq \gamma_j$ for all $i,j$. 

The number of elements in $\mathcal{F}_{k,m}^{\mrm{cusp}}$ equals the number of sub-partitions $\lambda$ of $(k^{k+m})$, and this number is equal to ${2k+m \choose k}$. We have just seen that $\mathcal{F}_{k,m}^{\mrm{cusp}}$ is contained in a single Lusztig family $\mathcal{F}$. Since the multiplicity of each entry in the content we have just computed is equal to $1$, it follows from Lemma \ref{typeB_lusztig_family_size} that $|\mathcal{F}| = {2k+m \choose k}$. Hence, $\mathcal{F}_{k,m}^{\mrm{cusp}} = \mathcal{F}$ is a Lusztig family.
\end{proof}

The symbols in Lemma \ref{typeB_cuspidal_family_minimal_reps} are in fact the minimal representatives of cuspidal symbols.

\begin{lemma} \label{typeB_cuspidal_symbols}
Suppose that $S \in \Sy_{\bc;n}^N$  is a cuspidal symbol. Then $n=k(k+m)$ for some $k$ and $S \in \mathcal{F}_{k,m}^{\mrm{cusp}}$.
\end{lemma}

\begin{proof}
Recall that $S$ cuspidal means that $n_{i} \ge n_{i+1}$ for $i = 0,1, \dots$. Since there is at least one $i$ such that $n_i \neq 0$, we have $n_0 \neq 0$. As in Lemma \ref{typeB_lusztig_family_size}, let $k \dopgleich N- |\lbrace i \mid n_i =2 \rbrace|$. Because $S$ is cuspidal, the symbol $S' \dopgleich S[-(N-k)] \in \Sy_{\bc;n}^k$ is well-defined. By definition of shift, the content of this symbol is equal to $\sum_{i=0}^{2k+m-1} x^i$. Equation (\ref{symbol_equation}) for $S'$ says that
\[
n = \sum_{i=0}^{2k+m-1} i -k^2 - k(m-1) - {m \choose 2} = k(k+m) \;.
\]
Hence, $S \in \mathcal{F}_{k,m}^{\mrm{cusp}}$ by Lemma \ref{typeB_cuspidal_family_minimal_reps} and Theorem \ref{typeB_lusztig_fams}.
\end{proof}

For a symbol $S \in \Sy_{\bc;n}^N$ as in (\ref{symbol_absract}) we define the symbol $\ol{S} \in \Sy_{\bc;n}^{N'}$ for certain $N'$ as follows. Choose $t \geq \max \{ \beta_{N+m}, \gamma_N \}$. Note that $t \geq m$ since $S$ has at least $N+m$ distinct entries. Now, the first row of $\overline{S}$ is the set $\{ 0, 1, \ds, t \} \smallsetminus \{ t - \gamma_1, \ds, t - \gamma_{N} \}$ and the second row is  $\{ 0, 1, \ds, t \} \smallsetminus \{ t - \beta_1, \ds, t - \beta_{N+m} \}$. By \cite[22.8]{LusztigUnequalparameters}, the symbol $\ol{S}$ belongs to $\Sy_{\bc;n}^{t+1-N-m}$ and by \cite[Lemma 22.18]{LusztigUnequalparameters} we have  $\pi_S \otimes \sgn_W = \pi_{\overline{S}}$.

\begin{example}
Let $\blambda = \left( \ \ydiagram{2,1} \ , \ \ydiagram{1} \ \right)$ and $(c_1,\kappa) = (1,1)$. Recall from Example \ref{typeB_symbol_example} that 
\[
\Sy_{(1,1);4}^3({\blambda}) = \begin{pmatrix} 0 & 1 &  3 & 5 \\ 0 & 1 & 3 \end{pmatrix} \in \Sy_{(1,1);4}^3 \;.
\]
Choosing $t=5$ we get
\[
\ol{\Sy_{(1,1);4}^3({\blambda}) } = \begin{pmatrix} 0 & 1 & 3 \\ 1 & 3 \end{pmatrix} = \Sy_{(1,1);4}^2\left( \ \ydiagram{1} \ , \ \ydiagram{2,1} \ \right) \;.
\]
Indeed,
\[
\left( \ \ydiagram{2,1} \ , \ \ydiagram{1} \ \right) \otimes  \underbrace{\left( \emptyset \ , \ \ydiagram{1,1,1,1} \ \right)}_{= \ \sgn_W} = \left( \ \ydiagram{1} \ , \ \ydiagram{2,1} \ \right) \;.
\]
\end{example}

\begin{thm} \label{typeB_cuspidal_lusztig}
There exists a cuspidal Lusztig family if and only if there is a $k > 0$ such that $n = k(k+m)$. In this case there is a unique cuspidal family, and it is equal to $\mc{F}_{k,m}^{\mrm{cusp}}$.
\end{thm}

\begin{proof}
Because of the transitivity of $\bj$-induction, Lusztig families of $B_n$ induced from some parabolic subgroup are also induced from some maximal parabolic subgroup. These subgroups are all of the form $B_l \times \s_{n-l}$ for some $0 \le l \le n-1$.  
The restriction of the parameter $\bc$ to $\fS_{n-l}$ is equal to $\kappa > 0$ and so the Lusztig families of $\fS_{n-l}$ are singletons by \S\ref{type_A}. A Lusztig family of $B_l \times \s_{n-l}$ is thus of the form $\lbrace \pi_{S} \boxtimes \pi_\lambda \mid \pi_{S} \in \mathcal{F} \rbrace$ for some Lusztig family $\mathcal{F}$ of $B_{l}$ and some fixed $\lambda \in \mathcal{P}(n-l)$. Since any given number can only appear at most twice in $S$, either the set of $n-l$ largest entries in $S$ is well-defined or there is a choice of two possible "largest $n-l$-entries". Notice that this depends only on the content of $S$. By adding $1$ to the $n-l$ largest entries, we get either a new symbol $S'$ or two new symbols $S^I$ and $S^{II}$. Then, as explained in \cite[Section 22.15]{LusztigUnequalparameters}, $\bj_{  B_{l} \times \s_{n-l}}^{B_n} \pi_{S} \boxtimes \sgn_W$ equals $\pi_{S'}$ or $\pi_{S^I} \oplus \pi_{S^{II}}$. In the latter case, the $\bj$-induction of $\pi_S \boxtimes \pi_\lambda$ is not irreducible and so $\bj$-induction does not induce a Lusztig family of $B_n$. We thus assume we are in the former case. Let $\sum_{i \ge 0} n_i x^i$ be the content of $S'$. Here $n_i \in \{ 0, 1, 2\}$ and $n_i = 0$ for $i \gg 0$. Assume that there exists some $i$ such that $n_i > n_{i-1}$. There are two possibilities, either $n_{i-1} = 0$ or $(n_{i-1},n_i) = (1,2)$. Consider first the former. We let $l$ be defined such that the $n-l$ largest numbers in $S$ are $\{ n_i \cdot i, n_{i+1} \cdot (i+1), \ds \}$. Here, $n_i \cdot i$ means that $i$ occurs with multiplicity $n_i$. Since $n_{i-1} = 0$, we can remove $1$ from each of the $n-l$ largest numbers and still have a well-defined symbol $S''$. Moreover, $\bj_{  B_{l} \times \s_{n-l}}^{B_n} \pi_{S''} \otimes \sgn_W = \pi_{S'}$. This applies to all symbols in the family to which $S$ belongs. Hence this family is not cuspidal. The other case is where $(n_{i-1},n_i) = (1,2)$. In this case it suffices to show that $\pi_{\overline{S}} = \pi_{S} \otimes \sgn_W$ is not cuspidal. If the content of $S$ is $\sum_{i \ge 0} n_i x^i$, then the content of $\overline{S}$ equals $\sum_{i \ge 0} (2 - n_{t-i}) x^i$ for some $t \gg 0$. Thus, there exists some $j$ such that $(n_{j-1},n_j)$ equals $(0,1)$ in the content of $\overline{S}$. By our previous argument, the family to which $\overline{S}$ belongs is induced from some parabolic subgroup. 

Above, we assumed that there exists some $i$ such that $n_i > n_{i-1}$. When $n_{i} \le n_{i-1}$ for all $i$, Lemma \ref{typeB_cuspidal_symbols} implies that $\pi_{S}$ belongs to the family $\mc{F}_{k,m}^{\mrm{cusp}}$ for some $k$ with $n=k(k+m)$. If $\mc{F}'$ is a family in $\Irr (B_l \times \s_{n-l})$ for some $l < n$, then Lemma \ref{typeB_lusztig_family_size} implies that $|\mc{F}'| < |\mc{F}|$. Hence $\mc{F}$ cannot be induced and must be cuspidal.  
\end{proof} 

\subsection{Cuspidal Lusztig families in the degenerate case} \label{typeB_cuspidal_lusztig_deg_sect} In this section we consider the case $\kappa = 0$. Recall from Lemma \ref{lusztig_fam_typeB_degenerate} that the Lusztig families are labeled $\mathcal{F}_{i;n}^{\mrm{deg}}$ for $i = 0, \ds, n$.  

\begin{lemma} \label{lusztig_fam_typeB_deg_sgn}
For any $i$, tensoring with the sign character yields a bijection $\mathcal{F}_{i;n}^{\mrm{deg}} \stackrel{\sim}{\longrightarrow}  \mathcal{F}_{n-i;n}^{\mrm{deg}}$.
\end{lemma}

\begin{proof}
If $\blambda = (\lambda^{(0)}, \lambda^{(1)}) \in \mathcal{P}_2(n)$, then $\pi_{\blambda} \otimes \sgn_W = \pi_{( (\lambda^{(1)})^*, (\lambda^{(0)})^*)}$ by \cite[Theorem 5.5.6(c)]{GeckPfeiffer}. Hence, if $\pi_{\blambda} \in \mathcal{F}_{i;n}^{\mrm{deg}}$, then $\pi_{\blambda} \otimes \sgn_W \in \mathcal{F}_{n-i;n}^{\mrm{deg}}$. As $\sgn_W$ is an automorphism on $\Irr W$, the claim follows.
\end{proof}

\begin{proposition}
If $\kappa = 0$ but $c_1 \neq 0$, there are no cuspidal Lusztig $\bc$-families.
\end{proposition}

\begin{proof}
We claim that $\bj_{B_i \times \fS_{n-i}}^{B_n}$ induces a bijection between the Lusztig family $\mathcal{F}_{i;i}^{\mrm{deg}} \times \Irr \fS_{n-i}$ of the parabolic subgroup $B_i \times \fS_{n-i}$ of $B_n$ and $\mathcal{F}_{i;n}^{\mrm{deg}}$ for $0 \leq i < n$. This shows that all $\mathcal{F}_{i;n}^{\mrm{deg}}$ with $i<n$ are induced, and since $\mathcal{F}_{0;n}^{\mrm{deg}} \otimes \sgn_{W} = \mathcal{F}_{n;n}^{\mrm{deg}}$ by Lemma \ref{lusztig_fam_typeB_deg_sgn}, this proves the claim. 

So, assume that $0 \leq i < n$. In \cite[Example 7.13]{Geck-Iancu-ordering} (see also \cite[\S2.4.3]{Geck.M;Jacon.N11Representations-of-H}) it is shown that the $\ba$-invariant in the degenerate case of $\bmu \in \mathcal{P}_2(i)$ is $\ba_{\bmu} = c_1 |\mu^{(1)}|$. Moreover, in the degenerate case, the restriction of the parameter to $\fS_{n-i}$ is zero so that the $\ba$-invariants $a_\nu$ are zero for all $\nu \in \mathcal{P}(n-i)$ by Example \ref{lusztig_0}. Hence,
\[
\bj_{B_i \times \fS_{n-i}}^{B_n} \pi_{\bmu} \boxtimes \pi_\nu = \sum_{ \substack{\blambda \in \mathcal{P}_2(n) \\ |\lambda^{(1)}| = |\mu^{(1)}|}} \langle \Ind_{B_i \times \fS_{n-i}}^{B_n} \pi_{\bmu} \boxtimes \pi_\nu, \pi_{\blambda} \rangle \pi_{\blambda} 
\]
for $\bmu \in \mathcal{P}_2(i)$ and $\nu \in \mathcal{P}(n-i)$. We will show that if $({\bmu},\nu) \in \mathcal{F}_{i;i}^{\mrm{deg}} \times \Irr \fS_{n-i}$, then
\begin{equation} \label{typeB_deg_no_cusp_claim}
\bj_{B_i \times \fS_{n-i}}^{B_n} \pi_{\bmu} \boxtimes \pi_\nu = \pi_{(\nu,\mu^{(1)})} \;.
\end{equation}
From this equation, the claim follows immediately. Since $\fS_{n-i}$ is a subgroup of $B_{n-i}$, we get the following relation using the branching rules:
\begin{align*}
\Ind_{B_i \times \fS_{n-i}}^{B_n} \pi_{\bmu} \boxtimes \pi_\nu  & = \Ind_{B_i \times B_{n-i}}^{B_n} \Ind_{B_i \times \fS_{n_i}}^{B_i \times B_{n-i}} \pi_{\bmu} \boxtimes \pi_\nu \\
 & = \Ind_{B_i \times B_{n-i}}^{B_n} \left( \pi_{\bmu} \boxtimes \Ind_{\fS_{n-i}}^{B_{n-i}} \pi_\nu \right) \\
& = \Ind_{B_i \times B_{n-i}}^{B_n} \left( \pi_{\bmu} \boxtimes \sum_{\boldsymbol{\alpha} \in \mathcal{P}_2(n-i)} c_{\boldsymbol{\alpha}}^\nu \pi_{\boldsymbol{\alpha}} \right) \\
& = \sum_{\boldsymbol{\lambda} \in \mathcal{P}_2(n)} \sum_{\boldsymbol{\alpha} \in \mathcal{P}_2(n-i)} c_{\mu^{(0)},\alpha^{(0)}}^{\lambda^{(0)}} c_{\mu^{(1)},\alpha^{(1)}}^{\lambda^{(1)}} c_{\boldsymbol{\alpha}}^\nu \pi_{\boldsymbol{\lambda}} \;.
\end{align*}
Note that the sum runs only over those $\boldsymbol{\alpha}$ with $\alpha^{(0)} , \alpha^{(1)} \subs \nu$ and $|\nu| = |\alpha^{(0)}| + |\alpha^{(1)}|$, and similarly only over those $\boldsymbol{\lambda}$ which satisfy 
\begin{equation} \label{typeB_deg_no_cusp_inclusion}
\mu^{(j)}, \alpha^{(j)} \subs \lambda^{(j)}
\end{equation}
and 
\begin{equation} \label{typeB_deg_no_cusp_sum}
|\lambda^{(j)}| = |\mu^{(j)}| + |\alpha^{(j)}|
\end{equation}
for $j \in \lbrace 1,2 \rbrace$. To show (\ref{typeB_deg_no_cusp_claim}) we need to show that among those $\boldsymbol{\lambda} \in \mathcal{P}_2(n)$ occurring in this sum such that $|\lambda^{(1)}| =  |\mu^{(1)}|$ we have
$$
c_{\mu^{(0)},\alpha^{(0)}}^{\lambda^{(0)}} c_{\mu^{(1)},\alpha^{(1)}}^{\lambda^{(1)}} c_{\boldsymbol{\alpha}}^\nu = \left\lbrace \begin{array}{ll} 1 & \tn{if } \blambda = (\nu,\mu^{(1)}), \bmu = (\emptyset, \mu^{(1)}), \boldsymbol{\alpha} = (\nu,\emptyset) \\ 0 & \tn{else.} \end{array} \right.
$$
So, suppose that $c_{\mu^{(0)},\alpha^{(0)}}^{\lambda^{(0)}} c_{\mu^{(1)},\alpha^{(1)}}^{\lambda^{(1)}} c_{\boldsymbol{\alpha}}^\nu \neq 0$. By (\ref{typeB_deg_no_cusp_inclusion}) we have $\mu^{(1)} \subs \lambda^{(1)}$, which implies that $\mu^{(1)}_k \leq \lambda^{(1)}_k$ for all $k$. Hence, as $|\lambda^{(1)}| =  |\mu^{(1)}|$ by assumption, we must have $\mu^{(1)} = \lambda^{(1)}$. In combination with (\ref{typeB_deg_no_cusp_sum}) this shows that $\alpha^{(1)} = \emptyset$. 

By definition of the Littlewood–Richardson coefficients, the coefficient $c_{\boldsymbol{\alpha}}^\nu = c_{\alpha^{(0)},\emptyset}^\nu$ is equal to the coefficient of the symmetric polynomial $s_\nu$ in the product $s_{\alpha^{(0)}} \cdot s_\emptyset = s_{\alpha^{(0)}} \cdot 1 = s_{\alpha^{(0)}}$. Hence,
$$
c_{\boldsymbol{\alpha}}^\nu = \left\lbrace \begin{array}{ll} 1 & \tn{if } \nu = \alpha^{(0)} \\ 0 & \tn{else} \end{array} \right.
$$
and therefore we must have $\nu = \alpha^{(0)}$. With the same argumentation we see that
$$
c_{\mu^{(1)},\alpha^{(1)}}^{\lambda^{(1)}} = c_{\lambda^{(1)}, \emptyset}^{\lambda^{(1)}} = 1 \;.
$$
Since $\bmu \in \mathcal{F}_{i;i}^{\mrm{deg}}$, we have $\mu^{(0)} = \emptyset$. Hence, 
$$
c_{\mu^{(0)},\alpha^{(0)}}^{\lambda^{(0)}} = c_{\emptyset, \nu}^{\lambda^{(0)}} = \left\lbrace \begin{array}{ll} 1 & \tn{if } \lambda^{(0)} = \nu \\ 0 & \tn{else} \end{array} \right.
$$
so that $\lambda^{(0)} = \nu$. This proves the claim.
\end{proof}

\subsection{Rigid modules} \label{typeB_rigid_modules}

We will now show that rigid modules exist precisely in the cuspidal cases and describe them explicitly. In this section, we consider again an arbitrary complex parameter $\bc$.

\begin{thm}\label{thm:rigidB}
There is a rigid $\H_\bc(W)$-module if and only if $\bc$ is cuspidal, i.e., $c_1 = m \kappa$ for some $m \in \pm \lbrack 0, n-1 \rbrack$ with $n = k(k+m)$ for some $k>0$. In this case there are exactly two rigid $\H_{\bc}(W)$-modules, namely $L_\bc(\blambda)$ with 
$$
\blambda = ((k^{k+m}),\emptyset) \;, \quad \textrm{or} \quad \blambda = (\emptyset,(k+m)^k) \;.
$$
\end{thm}

\begin{proof}
First of all, by Theorem \ref{maintheorem:rigid_cuspidal} and Corollary \ref{typeB_existence_cuspidals} there can only exist a rigid $\H_\bc(W)$-module if $\bc$ is cuspidal. So, assume that $c_1 = m\kappa$ for some $m \in \pm \lbrack 0, n-1 \rbrack$ with $n=k(k+m)$ for some $k>0$. By Proposition \ref{prop:someauto}, we can assume that $m \in \lbrack 0,n-1 \rbrack$. Let $\blambda = (\lambda^{(0)}, \lambda^{(1)})$ be a bipartition of $n$. By Lemma \ref{rigidity_equation_lemma} and equations (\ref{eq:Brel1}) and (\ref{eq:Brel2}) the representation $\pi_{\blambda}$ is $\bc$-rigid if and only if
\begin{equation}\label{eq:killv1}
2m \eps_i(-1)\cdot v +  \sum_{\substack{j=1 \\ j \neq i}}^n   (s_{ij} + s_{ij,-1}) \cdot v = 0 \;,  \quad \forall \ v \in \pi_{\blambda}, \ i = 1, \ldots, n \;,
\end{equation}
and
\begin{equation}\label{eq:killv2}
 (s_{ij} - s_{ij,-1}) \cdot v = 0 \;, \quad \forall \ v \in \pi_{\blambda}, \ i \neq j \;. 
\end{equation} 
Let $r \dopgleich |\lambda^{(0)}|$. Take $v$ to be a non-zero vector in the irreducible $(B_r \times B_{n-r})$-subrepresentation $\pi_{\lambda^{(0)}} \otimes \gamma \pi_{\lambda^{(1)}}$ inducing $\pi_{\blambda}$; see equation (\ref{eq:Bnrep}). Suppose that $r \notin \lbrace 0,n \rbrace$. Then we can find $i < j$ with $i \leq r$ and $r < j$. Due to this choice, we have $\eps_i(-1) \cdot v = v$ and $\eps_j(-1) \cdot v = -v$ as we twist by $\gamma$ in the second component. Hence, 
\[
s_{ij,-1} \cdot v = s_{ij} \eps_i(-1) \eps_j(-1) \cdot v = - s_{ij} \cdot v \;.
\]
Equation (\ref{eq:killv2}) thus says that $2 s_{ij} \cdot v = 0$ and therefore already $v=0$. This is a contradiction, so we must have $r \in \lbrace 0,n \rbrace$.
Assume that $r = n$. Then $\pi_{\blambda} = \pi_{\lambda^{(0)}}$. Now, equation (\ref{eq:killv2}) says that 
$$
\sum_{ \substack{j=1 \\ j \neq i}}^n s_{ij} \cdot v = - m v \;, \quad \forall \ v \in \pi_{\lambda^{(0)}}, \ i = 1, \ldots n \;.
$$
In other words, $\sum_{j \neq i} s_{ij}$ acts by a scalar on $\pi_{\lambda^{(0)}}$. This holds in particular for $i=1$, and now a standard result (see \cite[Lemma 2.4]{MontaraniEtingof}) implies that $\lambda^{(0)} = (l^b)$ is a rectangle for some positive integers $l,b$ with $lb = n$. The $n$-th Jucys–Murphy element $z_n = \sum_{j < n} s_{jn}$ acts on $\pi_{(l^b)}$ by multiplication by $l - b$ since every standard tableaux on $(a^b)$ must have $n$ in the top corner. Thus, $lb = n$ and $l + m = b$, so $n = l(l+m)$. Because of Remark \ref{typeB_cuspidal_param_unique} we must have $l=k$, proving the claim. If $r = 0$, then $\pi_{\blambda} = \gamma \pi_{\lambda^{(1)}}$ and the same argument shows that $\lambda^{(1)} = ((k+m)^k)$.   
\end{proof}

\begin{remark}
The proof of Theorem \ref{thm:rigidB} can be adapted to all the groups $G(\ell,1,n) = \Z_\ell \wr \fS_n$. 
\end{remark}

\subsection{Cuspidal Lusztig families vs. cuspidal Calogero–Moser families} \label{typeB_final_sect}

Combing all of the above results, we arrive at the proof of Theorem \ref{thm:mainresultintro} for type $B_n$.

\begin{thm}\label{thm:cuspidaltypeB}
Assume that $\bc \geq 0$. Then 
\begin{enum_thm}
\item A family (Lusztig = Calogero–Moser) is cuspidal in the sense of Lusztig if and only if it is cuspidal as a Calogero–Moser family. 
\item If $\kappa \neq 0$ and $n = k(k+m)$ for some $k,m \in \bbN$, then both rigid modules lie in the (unique) cuspidal family $\mc{F}_{k,m}^{\mrm{cusp}}$. 
\end{enum_thm}
\end{thm}

\section{Type $D$} \label{type_D}

The group $D_n$ is a normal subgroup of $B_n$ of index two. By setting $\bc = (0,\kappa)$, we get an embedding $\H_{\bc}(D_n) \hookrightarrow \H_{\bc}(B_n)$. Thus, we are in the situation of section \S\ref{clifford_theory}. We assume that $\kappa \neq 0$. Recall that the irreducible representations of $D_n$ are essentially given by unordered bipartitions of $n$. More precisely, if $\lambda$ and $\mu$ are a pair of partitions such that $\lambda \neq \mu$ and $|\lambda| + |\mu| = n$ then the set $\{ \lambda,\mu \}$ labels a simple $D_n$-module. If $\lambda = \mu$, then there are two non-isomorphic simple modules $\{ \lambda \}_1$ and $\{ \lambda\}_2$ labeled by $\lambda$. These modules are defined by
$$
\pi_{(\lambda,\mu)} |_{D_n} = \pi_{\{ \lambda,\mu \} } \quad \textrm{for } \lambda \neq \mu \;,
$$
and $\pi_{(\lambda,\lambda)} |_{D_n} = \pi_{\{ \lambda \}_1 } \oplus \pi_{\{ \lambda \}_2 }$.

\begin{lem}\label{lem:rigidD}
If there exists $k$ such that $n = k^2$ then there is a unique rigid $\H_\bc(D_n)$-module, which is $L_{\bc}(\{ (k^k),\emptyset \})$. Otherwise, there are no rigid modules.
\end{lem}

\begin{proof}
By Theorem \ref{thm:rigidB}, if $n = k^2$ for some $k$, then the modules $L_\bc((k^k),\emptyset)$ and $L_\bc(\emptyset, (k^k))$ are the two rigid modules for $ \H_{\bc}(B_n)$ and if there exists no $k$ such that $n = k^2$, then there exists no rigid modules. Therefore, Proposition \ref{prop:rigidKcusp} implies that, in this latter case, there exist no rigid modules for $\H_{\bc}(D_n)$. Moreover, in the case $n = k^2$, the rigid $ \H_{\bc}(D_n)$-modules are precisely the modules of the form $L_\bc(\lambda)$, where $\pi_{\lambda}$ is an irreducible $D_n$-submodule of $\pi_{((k^k),\emptyset)}$ or $\pi_{(\emptyset, (k^k))}$. But both of these $B_n$-modules restrict to the irreducible $D_n$-module $\pi_{ \{ (k^k),\emptyset \} }$.
\end{proof}

\begin{thm} \label{thm:Dleaves}
Assume that $\kappa \neq 0$. The symplectic leaves of $\cX_{\bc}(D_n)$ are in bijection with the set $\{ k \ge 1 \ | \ k^2 \le n \}$, such that
\begin{enum_thm}
\item $\dim \mc{L}_k = 2 (n - k^2)$,
\item the leaf $\mc{L}_k$ is labeled by the conjugacy class of the parabolic subgroup $D_{k^2}$,
\item $\mc{L}_{k} \prec \mc{L}_{k'}$ if and only if $(D_{k^2}) \le (D_{(k')^2})$, if and only if $k \ge k'$.
\end{enum_thm}
\end{thm}

\begin{proof}
By Lemma \ref{lem:rigidD} and Theorem \ref{maintheorem:rigid_cuspidal}, there exists at least one zero-dimensional leaf in $\cX_{\bc}(D_n)$ when $n = k^2$. But we know that there is exactly one zero-dimensional leaf in $\cX_{\bc}(B_n)$ when $n = k^2$. Thus, Lemma \ref{lem:etaleaf} implies that $\cX_{\bc}(D_n)$ contains exactly one zero-dimensional leaf when $n = k^2$. Since $\cX_{\bc}(B_n)$ contains no zero-dimensional leaves when $n \neq k^2$, Lemma \ref{lem:etaleaf} implies that $\cX_{\bc}(D_n)$ also contains no zero-dimensional leaves in this case.

Now, we apply Theorem \ref{pmax_parabolic_orbit}. By Lemma \ref{gmmn_parabolics}, the proper parabolic subgroups of $D_n$ are all conjugate to a subgroup of the form $D_m \times \s_{\lambda}$, where $0 \le m < n$ and $\lambda$ is a partition of $n - m$. We denote by $\bc'$ the restriction of $\bc$ to $D_m \times \fS_\lambda$. Let us consider when $\cX_{\bc'}(D_m \times \s_{\lambda}) = \cX_{\bc'}(D_m) \times \cX_{\bc}(\s_{\lambda})$ admits a zero dimensional leaf. Since $\kappa \neq 0$, there is a zero-dimensional leaf in $\cX_{\bc'}(\s_{\lambda})$ if and only if $\s_{\lambda} = \{ 1 \}$, i.e. if $\lambda = (1^{n-m})$. In this case $\cX_{\bc}(\s_{\lambda})$ is a point. Moreover, $\cX_{\bc'}(D_m)$ has a zero-dimensional leaf if and only if there exists a $k$ such that $m = k^2$. Thus, either $m = k^2$ and $\lambda = (1^{n-m})$, in which case there is exactly one zero-dimensional leaf in $\cX_{\bc'}(D_m \times \s_{\lambda})$, or there are no zero-dimensional leaves in $\cX_{\bc'}(D_m \times \s_{\lambda})$. Hence Theorem \ref{pmax_parabolic_orbit} implies the statements of the theorem.
\end{proof}

By \cite[Corollary 6.10]{CMPartitions}, the Calogero--Moser families for $\H_{\bc}(D_n)$, with $\bc \neq 0$, are described as follows. If $\lambda$ is a partition of $n/2$, then the two representations $\{ \lambda \}_{1}$ and $\{ \lambda \}_{2}$ each form a singleton family. Otherwise, $\{ \lambda,\mu \}$ and $\{ \lambda',\mu' \}$ are in the same family if and only if $\Res_{\{ \lambda,\mu \}}(x) = \Res_{\{ \lambda' ,\mu' \} }(x)$, where $\Res_{ \{ \lambda,\mu \} }(x) := \Res_{\lambda}(x) + \Res_{\mu}(x)$.

\begin{thm}\label{thm:typeDcusp}
Assume that $\bc \geq 0$.
\begin{enum_thm}
\item The Lusztig $\bc$-families for $D_n$ equal the Calogero--Moser $\bc$-families.
\item The cuspidal Lusztig $\bc$-families equal the cuspidal Calogero--Moser $\bc$-families.
\end{enum_thm}
\end{thm}

\begin{proof}
By Lemma \ref{lusztig_rescaling}, Lemma \ref{cm_scaling} and Lemma \ref{cuspidals_for_c_eq_0} we may assume that $(c_1,\kappa) = (0,1)$. The first part of the theorem follows from Corollary \ref{cor:LusztigCMtypeB}, \cite[Section 22.26]{LusztigUnequalparameters}, and the above description of the Calogero--Moser families.
As shown in \cite[Section 8.1]{Lusztig-characters-reductive-groups}, there is a unique cuspidal Lusztig family when $n = k^2$ and none otherwise. In the case $n = k^2$, it is the unique family containing the symbol
\begin{equation}\label{eq:cuspLusD}
\left( \begin{array}{cccc}
0, & 2, & \ds, & 2k - 2 \\
1, & 3, & \ds, & 2k - 1
\end{array}\right).
\end{equation}
The content of this symbol is $\sum_{i = 0}^{2k - 1} x^i$, which is the same as the content of the symbol
$$
S = \left( \begin{array}{cccc}
k, & k+1, & \ds, & 2k-1 \\
0, & 1, &\ds, & k-1
\end{array}\right).
$$
This is the symbol of $((k^k),\emptyset)$ in $\Sy_{(1,0);n}^k$, which implies that the cuspidal Lusztig family corresponding to the content of the symbol (\ref{eq:cuspLusD}) is the same as the Calogero–Moser family containing $\lbrace (k^k),\emptyset \rbrace$. By Lemma \ref{lem:rigidD} and Theorem \ref{thm:Dleaves}, this is the unique cuspidal Calogero–Moser family. 
\end{proof}

\section{Type $I_2(m)$} \label{dihedral_groups}

In this section we treat the case of dihedral groups. We show that almost all representations of the restricted rational Cherednik algebra are rigid. From this we easily obtain the proof of Theorem \ref{thm:mainresultintro}. We note that the results here together with \cite[Appendix B]{Thiel:2015aa} give a complete description of the representation theory of restricted rational Cherednik algebras for dihedral groups at all parameters.

\subsection{The group}
Throughout, we assume that $m \geq 5$ and choose a primitive $m$-th root of unity $\zeta \in \bbC$. Let $W$ be the Coxeter group of type $I_2(m)$. This is the dihedral group of order $2m$. It has two natural presentations, namely the Coxeter presentation $\langle s,t \mid s^2=t^2=(st)^m=1 \rangle$ and the geometric presentation $\langle s,r \mid r^m=1, s^2=1, s^{-1}rs = r^{-1} \rangle$ with a generating rotation $r \dopgleich st$ for the symmetries of a regular $m$-gon.

\subsection{Representations}

The representation theory of $W$ depends on the parity of $m$. In the following we use the same notation for the representations as in \cite[5.3.4]{GeckPfeiffer}, which essentially is also the same as in \cite{Geck.M;Jacon.N11Representations-of-H}.

If $m$ is odd, the conjugacy classes of $W$ are
\[
\lbrace 1 \rbrace, \; \lbrace r^{\pm 1} \rbrace, \; \lbrace r^{\pm 2 } \rbrace, \; \ldots, \; \lbrace r^{\pm (m-1)/2 } \rbrace, \; \lbrace r^ls \mid 0 \leq  l \leq m-1 \rbrace  \;,
\]
and so the total number of conjugacy classes is $(m+3)/2$. There are two irreducible one-dimensional representations: the trivial one $1_W$ and the sign representation $\eps:W \rarr \bbC$ with
\[
\eps(s) = -1 \;, \quad \eps(t) = -1\;,  \quad \eps(r) = 1 \;.
\]
The remaining $(m+3)/2-2 = (m-1)/2$ irreducible representations $\varphi_i$, $1 \leq i \leq (m-1)/2$, are all two-dimensional and are given by
\[
\varphi_i(s) = \begin{pmatrix} 0 & 1 \\ 1 & 0 \end{pmatrix}, \; \quad \varphi_i(t) \dopgleich \begin{pmatrix} 0 & \zeta^{-i} \\ \zeta^i & 0 \end{pmatrix} \;, \quad \varphi_i(r) = \begin{pmatrix} \zeta^i & 0 \\ 0 & \zeta^{-i} \end{pmatrix} \;.
\]
We denote the character of $\varphi_i$ by $\chi_i$.

If $m$ is even, then the conjugacy classes of $W$ are
\begin{align*}
& \lbrace 1 \rbrace \;, \; \lbrace r^{\pm 1} \rbrace \;, \; \lbrace r^{\pm 2 } \rbrace \;, \; \ldots,\lbrace r^{\pm m/2 } \rbrace \; , \\
& \lbrace r^{2k}s \mid 0 \leq k \leq (m/2)-1 \rbrace \;,\; \lbrace r^{2k+1}s \mid  0 \leq k \leq (m/2)-1 \rbrace \;,
\end{align*}
and so the total number of conjugacy classes is $(m+6)/2$. There are four irreducible one-dimensional representations: the trivial one $1_W$, the sign representation $\eps$, and two further representations $\eps_1,\eps_2$ with
\[
\begin{array}{llll}
\eps(s) = -1 \;, & \eps(t) = -1\;, & \eps(r) = 1 \;, \\
\eps_1(s) = 1 \;, & \eps_1(t) = -1 \;, & \eps_1(r) = -1 \;, \\
\eps_2(s) = -1 \;, & \eps_2(t) = 1 \;, & \eps_2(r) = -1 \;.
\end{array}
\]
The remaining $(m+6)/2-4 = (m-2)/2$ irreducible representations $\varphi_i$, $1 \leq i \leq (m-2)/2$, are all two-dimensional and are defined as in case $m$ is odd. Again, we denote the character of $\varphi_i$ by $\chi_i$.

\subsection{Reflections and parameters} \label{dihedral_reflections}
The two-dimensional faithful irreducible representation $\varphi_1$ of $W$ is a reflection representation in which precisely the elements $s_l \dopgleich r^l s$ for $0 \leq l \leq m-1$ act as reflections. We will always fix this representation as the reflection representation of $W$. Let $(y_1,y_2)$ be the standard basis of $\fh \dopgleich \bbC^2$ and let $(x_1,x_2)$ be the dual basis. We can easily verify that roots and coroots for the reflections $s_l$ are given by
\[
\alpha_{s_l} = x_1 - \zeta^{-l} x_2 \quad \tn{and} \quad \alpha_{s_l}^\vee = y_1 - \zeta y_2 \;.
\]
With this we see that the Cherednik coefficients $(y_i,x_j)_{s_l} = -(y_i,\alpha_{s_l})(\alpha_{s_l}^\vee, x_j)$ are:
\[
  (y_1,x_1)_{s_l} = -1 \;, \quad (y_1,x_2)_{s_l} =  \zeta^{-l} \;, \quad (y_2,x_1)_{s_l} = \zeta^{l}  \;, \quad (y_2,x_2)_{s_l} = -1 \;.
\]
If $m$ is odd, there is just one conjugacy class of reflections in $W$, namely the one of $s$ which is $\lbrace s_l \mid 0 \leq l \leq m-1 \rbrace$. If $m$ is even, there are two conjugacy classes of reflections in $W$, namely the one of $s$ which is $\lbrace s_{2l} \mid 0 \leq l \leq \frac{m}{2}-1 \rbrace$ and the one of $t$ which is $\lbrace s_{2l+1} \mid 0 \leq l \leq \frac{m}{2} -1 \rbrace$. 
Note that 
\[
\varphi_i(s_l) = \begin{pmatrix} 0 & \zeta^{il} \\ \zeta^{-il} & 0 \end{pmatrix} \;.
\]

If $\bc:\Ref(W) \rarr \bbC$ is a function which is constant on conjugacy classes, then, as in \cite[1.3.7]{Geck.M;Jacon.N11Representations-of-H}, we define
\[
b \dopgleich \bc(s) \;, \quad a \dopgleich \bc(t) \;.
\]
We fix such a function from now on and assume that $\bc \neq 0$. Note that if $m$ is odd, we have $a=b$.

\subsection{Summary}

We start with a tabular summary of the description of (cuspidal) Calogero–Moser families and rigid representations. To simplify notation we denote by $\mathcal{F}$ the set of two-dimensional irreducible characters of $W$. 
To allow a presentation which is independent of the parity of $m$ we set 
\[
\mathcal{R} \dopgleich \left\lbrace \begin{array}{ll} \lbrace \varphi_i \mid 1 < i \leq (m-1)/2 \rbrace = \mathcal{F} \setminus \lbrace \varphi_1 \rbrace& \tn{if } m \tn{ is odd} \\ \lbrace \varphi_i \mid 1 < i < (m-2)/2 \rbrace = \mathcal{F} \setminus \lbrace \varphi_1, \varphi_{(m-2)/2} \rbrace& \tn{if } m \tn{ is even.} \end{array} \right.
\]
We make the convention that we ignore $\eps_1$ and $\eps_2$ whenever $m$ is odd.

\begin{theorem} \label{dihedral_CM_theorem}
The (cuspidal) Calogero–Moser families and rigid representations of $\ol{\H}_{\bc}(W)$ are as listed in Table \ref{dihedral_bigtable}. 
\begin{table}[htbp]
\begin{tabular}{|c|c|c|c|}
\hline
Parameters & CM families & \specialcell{rigid \\ representations} & \specialcell{cuspidal \\ CM families} \\ \hline \hline
$a,b \neq 0$ and $a \neq \pm b$ & $\lbrace 1 \rbrace$, $\lbrace \eps \rbrace$, $\lbrace \eps_1 \rbrace$, $\lbrace \eps_2 \rbrace$, $\mathcal{F}$ & $\mathcal{R}$ & $\mathcal{F}$ \\ \hline
$a = 0$ and $b \neq 0$ & $\lbrace 1, \eps_2 \rbrace$, $\lbrace \eps,\eps_1 \rbrace$, $\mathcal{F}$ &  $\mathcal{R}$ & $\mathcal{F}$ \\ \hline
$a \neq 0$ and $b = 0$ &  $\lbrace 1, \eps_1 \rbrace$, $\lbrace \eps,\eps_2 \rbrace$, $\mathcal{F}$ & $\mathcal{R}$ & $\mathcal{F}$ \\ \hline
$a = b \neq 0$ & $\lbrace 1 \rbrace$, $\lbrace \eps \rbrace$, $\lbrace \eps_1, \eps_2 \rbrace \cup \mathcal{F}$ & $\eps_1$, $\eps_2$, $\varphi_{|\mathcal{F}|}$, $\mathcal{R}$ & $\lbrace \eps_1, \eps_2 \rbrace \cup \mathcal{F}$ \\ \hline
$a = -b \neq 0$ & $\lbrace \eps_1 \rbrace$, $\lbrace \eps_2 \rbrace$, $\lbrace 1,\eps \rbrace \cup \mathcal{F}$ & $1,\eps,\varphi_1$, $ \mathcal{R}$ & $\lbrace 1,\eps \rbrace \cup \mathcal{F}$ \\ \hline
\end{tabular}
\caption{The (cuspidal) Calogero–Moser families and rigid representations for dihedral groups.}
\label{dihedral_bigtable}
\end{table}
\end{theorem}

In the next three sections we will prove this theorem.

\subsection{Calogero–Moser families}

We recall from \cite{Thiel:2014aa} the notion of \word{Euler $\bc$-families}. These are defined by the action of the (central) Euler element of $\ol{\H}_\bc(W)$ on the simple modules and are coarser than the Calogero–Moser $\bc$-families. In \cite[Corollary 1]{Thiel:2014aa} a simple character theoretic formula for determining these families is given: two irreducible characters $\lambda$ and $\mu$ of $W$ lie in the same Euler $\bc$-family if and only if 
\[
\sum_{x \in \Ref(W)} \bc(x) \left( \frac{\lambda(x)}{\lambda(1)} - \frac{\mu(x)}{\mu(1)} \right) = 0 \;.
\]
This formula is in our case equivalent to
\[
a \left( \frac{\lambda(s)}{\lambda(1)} - \frac{\mu(s)}{\mu(1)} \right) + b \left( \frac{\lambda(t)}{\lambda(1)} - \frac{\mu(t)}{\mu(1)} \right) = 0\;.
\] 
From this it is easy to deduce that the Euler families are as in Table \ref{dihedral_bigtable}. In \cite{Bellamy:2010aa} the first author has proven that, for any $\bc$, the Euler $\bc$-families are in fact already Calogero–Moser $\bc$-families when $W$ is a dihedral group.

\subsection{Rigid representations}

We split the proof of the description of rigid representations into two parts, depending on the parity of $m$.

\begin{proposition} \label{dihedral_rigid_odd}
Assume that $m$ is odd. The following holds:
\begin{enum_thm}
\item The representations $1$, $\eps$, and $\varphi_1$ are not rigid.
\item The representation $\varphi_i$ is a rigid representation for all $1 < i \leq (m-1)/2$.
\end{enum_thm}
\end{proposition}

\begin{proof}
  The representation $\varphi_i$ for $1 \leq i \leq (m-1)/2$ is rigid if and only if 
  \[
  0 = a \sum_{l=0}^{m-1} (y_k,x_j)_{s_l} \varphi_i(s_l) 
  \]
  for all $k,j \in \lbrace 1,2 \rbrace$. 
  As $a \neq 0$, this is equivalent to
  \begin{equation} \label{dihedral_group_rigiditiy_cond}
  \sum_{l=0}^{m-1} (y_k,x_j)_{s_l} \varphi_i(s_l) = 0
  \end{equation}
  for all $k,j \in \lbrace 1,2 \rbrace$. Note that 
  \[
  \sum_{l=0}^{m-1} (\zeta^q)^l = \left\lbrace \begin{array}{ll} m & \tn{if } q \in m\bbZ \\ 0 & \tn{else.} \end{array} \right.
  \]
  Using the Cherednik coefficients computed in \S\ref{dihedral_reflections}, equation  (\ref{dihedral_group_rigiditiy_cond}) is for $k=1=j$ and for $k=2=j$ equivalent to
  \[
  \sum_{l=0}^{m-1} \begin{pmatrix} 0 & \zeta^{il} \\ \zeta^{-il} & 0 \end{pmatrix}   = 0  \Longleftrightarrow \sum_{l=0}^{m-1} (\zeta^{i})^l = 0 \quad \tn{and} \quad \sum_{l=0}^{m-1} (\zeta^{-i})^l = 0
  \]
  and due to the aforementioned, this condition is satisfied if and only if $i \notin m \bbZ$. Since $1 \leq i \leq (m-1)/2$, this is always satisfied. For $k=1$ and $j=2$ equation (\ref{dihedral_group_rigiditiy_cond}) is equivalent to
  \begin{align*}
 &\sum_{l=0}^{m-1} \zeta^{-l} \begin{pmatrix} 0 & \zeta^{il} \\ \zeta^{-il} & 0 \end{pmatrix}   = 0 \Longleftrightarrow \sum_{l=0}^{m-1} \zeta^{il-l} = 0 \quad \tn{and} \quad \sum_{l=0}^{m-1} \zeta^{-il-l} = 0 \\
 & \Longleftrightarrow \sum_{l=0}^{m-1} (\zeta^{i-1})^l = 0 \quad \tn{and} \quad \sum_{l=0}^{m-1} (\zeta^{-i-1})^l = 0 \\
 &\Longleftrightarrow i-1 \notin m\bbZ \tn{ and } i+1 \notin m\bbZ \;.
  \end{align*}
  Due to $1 \leq i \leq (m-1)/2$ the condition $i+1 \notin m\bbZ $ is always satisfied. Hence, (\ref{dihedral_group_rigiditiy_cond}) holds for $k=1$ and $j=2$ if and only if $i \neq 1$. Finally, for $k=2$ and $j=1$ equation (\ref{dihedral_group_rigiditiy_cond}) is equivalent to 
  \begin{align*}
  &\sum_{l=0}^{m-1} \zeta^{l} \begin{pmatrix} 0 & \zeta^{il} \\ \zeta^{-il} & 0 \end{pmatrix}   = 0 \Longleftrightarrow \sum_{l=0}^{m-1} \zeta^{il+l} = 0 \quad \tn{and} \quad \sum_{l=0}^{m-1} \zeta^{-il+l} = 0 \\
 & \Longleftrightarrow \sum_{l=0}^{m-1} (\zeta^{i+1})^l = 0 \quad \tn{and} \quad \sum_{l=0}^{m-1} (\zeta^{-i+1})^l = 0 \\
 & \Longleftrightarrow i+1 \notin m\bbZ \tn{ and } i-1 \notin m\bbZ \;.  
  \end{align*}
  Again, the condition $i+1 \notin \bbZ$ is always satisfied and $i-1 \notin m\bbZ$ holds if and only if $i\neq 1$. This proves the claim.
\end{proof}

\begin{proposition}
Assume that $m$ is even. The following holds: 
\begin{enum_thm}
\item For any $1 < i < (m-2)/2$ the representation $\varphi_i$ is rigid.
\item The representation $\varphi_1$ is rigid if and only if $a = -b$.
\item The representation $\varphi_{(m-2)/2}$ is rigid if and only if $a = b$.
\item The representations $1$ and $\eps$ are rigid if and only if $a = -b$.
\item The representations $\eps_1$ and $\eps_2$ are rigid if and only if $a = b$.
\end{enum_thm}
\end{proposition}

\begin{proof} 
This follows from a similar direct computation as in the proof of Proposition \ref{dihedral_rigid_odd}. We omit the details here.
\end{proof}

\subsection{Cuspidal Calogero–Moser families}

For $m \geq 5$ the first author has shown in \cite[\S5.5]{Cuspidal} that independent of the parameter $\bc$ there is exactly one cuspidal Calogero–Moser family. It thus remains to identify this family. Since $m \geq 5$ we have $\mathcal{R} \neq \emptyset$, and as $\mathcal{R}$ is always contained in the Calogero–Moser family which in Table \ref{dihedral_bigtable} is claimed to be cuspidal, it follows from Theorem \ref{maintheorem:rigid_cuspidal} that this family is indeed the unique cuspidal one.

\subsection{Lusztig families} \label{dihedral_lusztig_families}

From now on we assume that $\bc \geq 0$. The Lusztig families of $W$ are listed in Table \ref{dihedral_lusztig_families_table} which is taken from \cite[1.7.3]{Geck.M;Jacon.N11Representations-of-H}.

\begin{table}[htbp]
\begin{tabular}{|c|c|} 
\hline
Parameters & Lusztig families \\ \hline \hline
$b=a > 0$ & $\lbrace 1_W \rbrace$, $\lbrace \eps \rbrace$, $\lbrace \eps_1,\eps_2 \rbrace \cup \mathcal{F}$ \\ \hline
$b>a > 0$ or $a>b>0$ & $\lbrace 1_W \rbrace$, $\lbrace \eps \rbrace$, $\lbrace \eps_1 \rbrace$, $\lbrace \eps_2 \rbrace$, $\mathcal{F}$ \\ \hline 
$b>a =0$ & $\lbrace 1_W, \eps_1 \rbrace$, $\lbrace \eps,\eps_2 \rbrace$, $\mathcal{F}$ \\ \hline 
$a>b=0$ & $\lbrace 1_W, \eps_2 \rbrace$, $\lbrace \eps,\eps_1 \rbrace$, $\mathcal{F}$ \\ \hline
\end{tabular}
\caption{Lusztig families}
\label{dihedral_lusztig_families_table}
\end{table}

Comparison with the Calogero–Moser families immediately yields the proof of Theorem \ref{thm:cm_equal_lusztig} for dihedral groups:

\begin{corollary} \label{dihedral_lusztig_qual_cm}
For any $\bc \geq 0$ the Lusztig $\bc$-families are equal to the Calogero–Moser $\bc$-families.
\end{corollary}

\subsection{Cuspidal Lusztig families}
In order to determine which of the Lusztig families are cuspidal we explicitly compute the $\bj$-induction. The group $W$ has two non-trivial proper parabolic subgroups: the group $P_1 \dopgleich \langle s \rangle$ and the group $P_2 \dopgleich \langle t \rangle$, which are both Coxeter groups of type $\rA_1$. Let $\psi_i$ be the non-trivial irreducible character of $P_i$ and note that this is the sign representation of this Coxeter group. It is not hard to compute that
\[
\Ind_{P_1}^W 1_{P_1} = 1_W + \eps_1 + \sum_{j} \chi_j \;, \quad \Ind_{P_1}^W \psi_1 = \eps + \eps_2 + \sum_{j} \chi_j + \delta 1_W \;,
\]
\[
\Ind_{P_2}^W 1_{P_2} = 1_W + \eps_2 + \sum_{j} \chi_j \;, \quad \Ind_{P_2}^W \psi_2 = \eps + \eps_1 + \sum_{j} \chi + \delta 1_W \;,
\]
where 
\[
\delta \dopgleich \left\lbrace \begin{array}{ll} 0 & \tn{if } m \tn{ is even} \\ 1 & \tn{if } m \tn{ is odd} \end{array} \right.
\]
and the sums are taken over all two-dimensional characters.

Lusztig's $\ba$-functions $\ba_\chi$ of the irreducible characters $\chi$ of $W$ with respect to $\bc$ are listed in Table \ref{dihedral_a_function} which is taken from \cite[1.3.7]{Geck.M;Jacon.N11Representations-of-H}, where the last row follows by symmetry.
\begin{table}[htbp]
\begin{tabular}{|c||c|c|c|c|c|}
\hline
Parameters & $\varphi_i$ & $1_W$ & $\eps$ & $\eps_1$ & $\eps_2$ \\ \hline \hline
$b = a > 0$ & $a$ & $0$ & $ma$ & $a$ & $a$ \\ \hline
$b > a \geq 0$ & $b$ & $0$ & $\frac{m}{2}(a+b)$ & $a$ & $\frac{m}{2}(b-a) + a$ \\ \hline
$a > b \geq 0$ & $a$ & $0$ & $\frac{m}{2}(a+b)$ & $\frac{m}{2}(a-b) + b$ & $b$\\ \hline
\end{tabular}
\caption{The $\ba$-function}
\label{dihedral_a_function}
\end{table}
Using \cite[1.3.3]{Geck.M;Jacon.N11Representations-of-H} we see that the $\ba$-functions for the irreducible characters of the parabolic subgroups with respect to the restriction of $\bc$ to these groups are as in Table \ref{dihedral_parabolic_a_function}.
\begin{table}[htbp]
\begin{tabular}{|c||c|c|c|c|}
\hline
$\chi$ & $1_{P_1}$ & $\psi_1$ & $1_{P_2}$ & $\psi_2$ \\ \hline \hline
$\ba_\chi$ & $0$ & $b$ & $0$ & $a$ \\ \hline
\end{tabular}
\caption{The $\ba$-function for the parabolic subgroups $P_i$.}
\label{dihedral_parabolic_a_function}
\end{table}
From these tables we can deduce that Lusztig's $\bj$-induction is as in Table \ref{dihedral_j_induction}.
\begin{table}[htbp]
\begin{tabular}{|c||c|c|c|c|}
\hline
Parameters & $\bj_{P_1}^W 1_{P_1}$ & $\bj_{P_1}^W \psi_1$ & $\bj_{P_2}^W 1_{P_2}$ & $\bj_{P_2}^W \psi_2$ \\ \hline \hline
$b=a>0$ & $1_W$ & $\eps_2 + \sum_j \chi_j$ & $1_W$ & $\eps_1 + \sum_j \chi_j$ \\ \hline
$b>a>0$ & $1_W$ & $\sum_j \chi_j$ & $1_W$ & $\eps_1$ \\ \hline
$b>a=0$ & $1_W + \eps_1$ & $\sum_j \chi_j$ & $1_W$ & $\eps_1$ \\ \hline \
$a>b>0$ & $1_W$ & $\eps_2$ & $1_W$ & $\sum_j \chi_j$ \\ \hline
$a>b=0$ & $1_W$ & $\eps_2$ & $1_W + \eps_2$ & $\sum_j \chi_j$ \\ \hline
\end{tabular}
\caption{$\bj$-induction.}
\label{dihedral_j_induction}
\end{table}

Using the table of $\bj$-inductions we can now easily determine the cuspidal Lusztig families.

\begin{lemma}
Let $\mathbf{c} \geq 0$. There is a unique cuspidal Lusztig family. This family is equal to $\lbrace \eps_1,\eps_2 \rbrace \cup \mathcal{F}$ if $b=a$, and otherwise it is equal to $\mathcal{F}$.
\end{lemma}

\begin{proof}
The Lusztig families of the parabolic subgroup $P_i$ are $\lbrace 1_{P_i} \rbrace$ and $\lbrace \psi_i \rbrace$ if $b \neq 0$, respectively $a \neq 0$, and they are $\lbrace 1_{P_i},\psi_i \rbrace$ if $b=0$, respectively $a=0$. The claim follows now easily from the definition of cuspidality using the table of $\bj$-inductions.
\end{proof}

Comparison with the cuspidal Calogero–Moser families completes the proof of Theorem~\ref{thm:mainresultintro}.

\end{document}